\documentclass[11pt,a4paper,reqno]{amsart}
\usepackage[margin=2cm]{geometry}
\usepackage{hyperref,amssymb,amsfonts,amsthm,amsrefs,mathtools,verbatim,enumerate,bm}
\usepackage{tikz}
\usepackage{microtype}

\newtheorem{Theorem}{Theorem}[section]
\newtheorem{Proposition}[Theorem]{Proposition}
\newtheorem{Lemma}[Theorem]{Lemma}
\newtheorem{Corollary}[Theorem]{Corollary}

\theoremstyle{definition}
\newtheorem{Definition}[Theorem]{Definition}

\newtheorem{Question}[Theorem]{Question}

\newcommand{\rca}{\mathsf{RCA}_0}
\newcommand{\wwklz}{\mathsf{WWKL}_0}
\newcommand{\wklz}{\mathsf{WKL}_0}
\newcommand{\aca}{\mathsf{ACA}_0}

\newcommand{\wkl}{\mathsf{WKL}}
\newcommand{\rt}{\mathsf{RT}}
\newcommand{\crt}{\mathsf{cRT}}
\newcommand{\srt}{\mathsf{SRT}}
\newcommand{\Dd}{\mathsf{D}}
\newcommand{\ads}{\mathsf{ADS}}
\newcommand{\adc}{\mathsf{ADC}}
\newcommand{\coh}{\mathsf{COH}}
\newcommand{\ocoh}{\mathsf{oCOH}}
\newcommand{\ocohd}{\mathsf{oCOH\mbox{-}in\mbox{-}}\mathbf{\Delta}^0_2}
\newcommand{\infone}{\mathsf{INForONE}}
\newcommand{\infoned}{\mathsf{INForONE\mbox{-}in\mbox{-}}\mathbf{\Delta}^0_2}
\newcommand{\infonesd}{\mathsf{INForONE\mbox{-}in}^*\mathsf{\mbox{-}}\mathbf{\Delta}^0_2}
\newcommand{\tdnr}{\mathrm{DNR}}
\newcommand{\pdnr}{\mathsf{DNR}}
\newcommand{\npdnr}[1]{#1\mbox{-}\mathsf{DNR}}
\newcommand{\rsg}{\mathsf{RSg}}
\newcommand{\rsgr}{\mathsf{RSgr}}
\newcommand{\wrsg}{\mathsf{wRSg}}
\newcommand{\wrsgr}{\mathsf{wRSgr}}
\newcommand{\id}{\mathsf{id}}
\newcommand{\sads}{\mathsf{SADS}}
\newcommand{\sadc}{\mathsf{SADC}}
\newcommand{\cads}{\mathsf{CADS}}
\newcommand{\cadsst}{\mathsf{CADS}_\mathrm{strict}}
\newcommand{\cac}{\mathsf{CAC}}
\newcommand{\scac}{\mathsf{SCAC}}

\newcommand{\pa}{\mathrm{PA}}

\newcommand{\lpo}{\mathsf{LPO}}

\newcommand{\iso}{\mathsf{I}\Sigma^0_1}
\newcommand{\bpo}{\mathsf{B}\Pi^0_1}
\newcommand{\bst}{\mathsf{B}\Sigma^0_2}

\DeclareMathOperator{\dom}{\mathrm{dom}}
\DeclareMathOperator{\ran}{\mathrm{ran}}
\DeclareMathOperator{\set}{\mathrm{set}}

\newcommand{\andd}{\wedge}

\newcommand{\la}{\langle}
\newcommand{\ra}{\rangle}
\newcommand{\da}{{\downarrow}}
\newcommand{\ua}{{\uparrow}}
\newcommand{\imp}{\rightarrow}

\newcommand{\biimp}{\leftrightarrow}

\newcommand{\Nb}{\mathbb{N}}

\newcommand{\smf}{\smallfrown}
\newcommand{\rst}{{\restriction}}
\newcommand{\rra}{\rightrightarrows}
\newcommand{\wsubseteq}{\!\!\subseteq\!}

\newcommand{\mc}[1]{\mathcal{#1}}
\newcommand{\pb}[1]{\mathsf{#1}}

\newcommand{\leqT}{\leq_\mathrm{T}}

\newcommand{\equivT}{\equiv_\mathrm{T}}

\newcommand{\leqW}{\leq_\mathrm{W}}
\newcommand{\nleqW}{\nleq_\mathrm{W}}
\newcommand{\ltW}{<_\mathrm{W}}

\newcommand{\equivW}{\equiv_\mathrm{W}}

\newcommand{\leqsW}{\leq_\mathrm{sW}}
\newcommand{\ltsW}{<_\mathrm{sW}}

\newcommand{\nleqsW}{\nleq_\mathrm{sW}}
\newcommand{\equivsW}{\equiv_\mathrm{sW}}
\newcommand{\nequivsW}{\not\equiv_\mathrm{sW}}

\newcommand{\leqc}{\leq_\mathrm{c}}
\newcommand{\lstc}{<_\mathrm{c}}
\newcommand{\nleqc}{\nleq_\mathrm{c}}
\newcommand{\equivc}{\equiv_\mathrm{c}}

\newcommand{\ltlex}{<_\mathrm{lex}}

\newcommand{\ol}[1]{\overline{#1}}
\newcommand{\code}[1]{\ulcorner #1 \urcorner}
\newcommand{\mbf}[1]{\bm{#1}}

\newcommand{\wh}[1]{\widehat{#1}}

\newcommand{\Pf}{\mathcal{P}_\mathrm{f}}

\newcommand{\asc}{\mathrm{asc}}
\newcommand{\dec}{\mathrm{dec}}

\newcommand{\last}{\mathsf{last}}

\AtBeginDocument{%
   \def\MR#1{}
}

\title{An inside/outside Ramsey theorem and recursion theory}

\author{Marta Fiori-Carones}
\address{Institute of Mathematics, University of Warsaw, Banacha 2, 02-097 Warszawa, Poland}
\email{marta.fioricarones@outlook.it}
\urladdr{https://martafioricarones.github.io}

\author{Paul Shafer}
\address{School of Mathematics\\
University of Leeds\\
Leeds\\
LS2 9JT\\
United Kingdom}
\email{p.e.shafer@leeds.ac.uk}
\urladdr{http://www1.maths.leeds.ac.uk/~matpsh/}

\author{Giovanni Sold\`a}
\address{School of Mathematics\\
University of Leeds\\
Leeds\\
LS2 9JT\\
United Kingdom}
\email{mmgs@leeds.ac.uk}

\date{\today}

\begin{document}

\begin{abstract}
Inspired by Ramsey's theorem for pairs, Rival and Sands proved what we refer to as an \emph{inside/outside Ramsey theorem}:  every infinite graph $G$ contains an infinite subset $H$ such that every vertex of $G$ is adjacent to precisely none, one, or infinitely many vertices of $H$.  We analyze the Rival--Sands theorem from the perspective of reverse mathematics and the Weihrauch degrees.  In reverse mathematics, we find that the Rival--Sands theorem is equivalent to arithmetical comprehension and hence is stronger than Ramsey's theorem for pairs.  We also identify a weak form of the Rival--Sands theorem that is equivalent to Ramsey's theorem for pairs.  We turn to the Weihrauch degrees to give a finer analysis of the Rival--Sands theorem's computational strength.  We find that the Rival--Sands theorem is Weihrauch equivalent to the double jump of weak K\"{o}nig's lemma.  We believe that the Rival--Sands theorem is the first natural theorem shown to exhibit exactly this strength.  Furthermore, by combining our result with a result of Brattka and Rakotoniaina, we obtain that solving one instance of the Rival--Sands theorem exactly corresponds to simultaneously solving countably many instances of Ramsey's theorem for pairs.  Finally, we show that the uniform computational strength of the weak Rival--Sands theorem is weaker than that of Ramsey's theorem for pairs by showing that a number of well-known consequences of Ramsey's theorem for pairs do not Weihrauch reduce to the weak Rival--Sands theorem.  We also address an apparent gap in the literature concerning the relationship between Weihrauch degrees corresponding to the ascending/descending sequence principle and the infinite pigeonhole principle.
\end{abstract}

\maketitle

\section{Introduction}

The logical analysis of Ramsey-theoretic statements is a long-standing program in recursion theory, going back to Specker's example of a recursive graph on the natural numbers with no infinite recursive homogeneous set~\cite{Specker}.  In this tradition, we analyze the result of Rival and Sands stated in Theorem~\ref{thm-RSg} below, which we refer to as an \emph{inside/outside Ramsey theorem}.  Our title is an homage to Jockusch's classic and ever-inspiring work \emph{Ramsey's theorem and recursion theory}~\cite{JockuschRamsey}.

\begin{Theorem}[Rival and Sands~\cite{RivalSands}]\label{thm-RSg}
Every infinite graph $G$ contains an infinite subset $H$ such that every vertex of $G$ is adjacent to precisely none, one, or infinitely many vertices of $H$.
\end{Theorem}

In Theorem~\ref{thm-RSg} and throughout this work, a \emph{graph} is a simple graph, that is, an undirected graph without loops or multiple edges.  Furthermore, we consider only countable graphs whose vertices are sets of natural numbers.

Compare Theorem~\ref{thm-RSg} to Ramsey's theorem for pairs and two colors, which in terms of graphs is phrased as follows.

\begin{Theorem}
Every countably infinite graph $G$ contains an infinite subset $H$ that is \emph{homogeneous} for $G$, i.e., is either a \emph{clique} (every pair of distinct vertices from $H$ is adjacent) or an independent set (no pair of distinct vertices from $H$ is adjacent).
\end{Theorem}

The \emph{Rival--Sands theorem} is thus an analog of Ramsey's theorem for pairs which prescribes the relation between \emph{every} vertex of the graph and the distinguished set $H$.  In other words, the Rival--Sands theorem provides information about vertices both \emph{inside} and \emph{outside} $H$.  This is why we describe the Rival--Sands theorem as an inside/outside Ramsey theorem.  In the words of Rival and Sands:
\begin{quote}
While Ramsey's result completely describes the adjacency structure of the distinguished subgraph, it provides no information about those edges which join vertices inside the subgraph to vertices external to it.

Our main result is, in a sense, a trade-off:  it provides information about the adjacency structure of an arbitrary vertex, with respect to the vertices of a distinguished infinite subgraph; on the other hand, the information that it provides about the adjacency structure of vertices internal to the subgraph is incomplete.~\cite{RivalSands}
\end{quote}
The trade-off is necessary.  In~\cite{RivalSands}*{Section~3}, Rival and Sands give an example of a graph in which no infinite homogeneous set satisfies the conclusion of their theorem.

We analyze the axiomatic and computational strength of the Rival--Sands theorem from the perspective of \emph{reverse mathematics} and the \emph{Weihrauch degrees}.  Ramsey's theorem for pairs is studied extensively in both of these contexts, which we introduce further in Sections~\ref{sec-RMPrelim} and~\ref{sec-WPrelim}.  The Rival--Sands theorem was inspired by Ramsey's theorem for pairs, and the two statements look very similar, so it is natural to ask how their strengths compare.

We analyze the axiomatic strength of the Rival--Sands theorem from the perspective of reverse mathematics in Section~\ref{sec-RMRS}.  \emph{Reverse mathematics} is a foundational program initiated by H.\ Friedman~\cite{Friedman} designed to answer questions of the sort \emph{What axioms of second-order arithmetic are required to prove a given classical theorem of countable mathematics?}  Reasoning takes place over a base theory called $\rca$, which corresponds to recursive mathematics.  Theorems are classified according to what extra axioms must be added to $\rca$ in order to prove them.  See Section~\ref{sec-RMPrelim} for further introduction.

Theorem~\ref{thm-ACA_RS_eq} shows that the Rival--Sands theorem is equivalent to the axiom system $\aca$ over $\rca$.  On the other hand, Ramsey's theorem for pairs is strictly weaker than $\aca$ by a classic result of Seetapun and Slaman~\cite{SeetapunSlaman}.  Therefore the axiomatic strength of the Rival--Sands theorem is strictly greater than that of Ramsey's theorem for pairs.

With input from Jeffry Hirst and Steffen Lempp, we also consider a weakened \emph{inside only} form of the Rival--Sands theorem, in which the conclusion need only hold for the vertices of $H$, not all vertices of $G$:  every infinite graph $G$ contains an infinite subset $H$ such that every vertex of $H$ is adjacent to precisely none, one, or infinitely many vertices of $H$.  This so-called \emph{weak Rival--Sands theorem} is a trivial consequence of Ramsey's theorem for pairs because homogeneous sets satisfy the conclusion of the weak Rival--Sands theorem.  Degree-theoretically, the weak Rival--Sands theorem is weaker than Ramsey's theorem for pairs.  The Turing jump of an infinite graph $G$ suffices to produce an $H$ as in the conclusion of the weak Rival--Sands theorem (see Proposition~\ref{prop-wRSgCompLoc}), but it does not necessarily suffice to produce a homogeneous set as in the conclusion of Ramsey's theorem for pairs.  Axiomatically, however, Theorem~\ref{thm-RT_WRSG_eq} shows that the weak Rival--Sands theorem and Ramsey's theorem for pairs are equivalent over $\rca$.  We may therefore think of the weak Rival--Sands theorem as being a weaker formalization of Ramsey's theorem for pairs.  We hope that this equivalence will prove useful to future studies of Ramsey's theorem for pairs.

The original proof of Theorem~\ref{thm-RSg} by Rival and Sands consists of intricate yet elementary reasoning that can be formalized in $\aca$ with a little engineering.  Instead of proceeding along these lines, we give a short, satisfying new proof using cohesive sets.  The reversal of the Rival--Sands theorem is unsatisfying, however, because $\aca$ is too coarse to see the difference between the computational strength used to implement the proof and the computational strength extracted by the reversal.  To produce the set $H$, our proof of the Rival--Sands theorem uses (after some refinement) a function of $\pa$ degree relative to the double Turing jump of $G$, whereas the reversal need only code a single Turing jump of $G$ into $H$.  Moreover, the reversal can be implemented using only \emph{locally finite} graphs, that is, graphs in which each vertex has only finitely many neighbors.

We therefore turn to the \emph{Weihrauch degrees} to give a precise characterization of the Rival--Sands theorem's computational strength.  What has come to be known as \emph{Weihrauch reducibility} was introduced by Weihrauch in~\cites{WeihrauchDegrees, WeihrauchTTE} and is used to compare mathematical theorems according to their uniform computational strength.  In this setting we think of theorems in terms of \emph{inputs and outputs} or \emph{instances and solutions}.  For example, for the Rival--Sands theorem, an input or instance is any countable graph $G$, and an output or solution to $G$ is any $H$ as in the conclusion of the theorem.  Theorems that can be phrased in terms of instances and solutions in this way are sometimes referred to as \emph{problems}.  Roughly, problem $\pb{P}$ \emph{Weihrauch reduces} to problem $\pb{Q}$ if there are uniform computational procedures converting instances of $\pb{P}$ into instances of $\pb{Q}$ and then converting solutions of the converted $\pb{Q}$-instance back into solutions to the original $\pb{P}$-instance.  Two problems are \emph{Weihrauch equivalent} if they Weihrauch reduce to each other.  Weihrauch equivalent problems are also said to have the same \emph{Weihrauch degree}.  Precise definitions and background are given in Section~\ref{sec-WPrelim}.

Section~\ref{sec-ADSvsRT1Weihrauch} addresses an apparent gap in the literature concerning Weihrauch degrees related to the ascending/descending sequence principle and the infinite pigeonhole principle (i.e., Ramsey's theorem for singletons).  Specifically, Theorem~\ref{thm-RT13vsADC} shows that the infinite pigeonhole principle for three colors Weihrauch reduces to the ascending/descending chain principle (a weaker version of the ascending/descending sequence principle studied in~\cite{AstorDzhafarovSolomonSuggs}); whereas Theorem~\ref{thm-RT15vsADS} shows that the infinite pigeonhole principle for five colors does not Weihrauch reduce to the ascending/descending sequence principle.  The question of the infinite pigeonhole principle for four colors remains open.  The ascending/descending sequence principle and the infinite pigeonhole principle are considered in relation to the Weihrauch degree of the weak Rival--Sands theorem in Section~\ref{sec-wRSgWeihrauch}.

We analyze the Rival--Sands theorem in the Weihrauch degrees in Section~\ref{sec-WRS}.  Theorem~\ref{thm-RSGequivWKLjj} shows that the Rival--Sands theorem is Weihrauch equivalent to the double jump of weak K\"{o}nig's lemma, which is the problem of finding an infinite path through an infinite binary tree described via a $\mbf{\Delta}^0_3$-approximation.  To our knowledge, the Rival--Sands theorem is the first ordinary mathematical theorem shown to exhibit exactly this strength.  Furthermore, by~\cite{BrattkaRakotoniaina}*{Corollary~4.18}, the double jump of weak K\"{o}nig's lemma is also Weihrauch equivalent to the parallelization of Ramsey's theorem for pairs, which is the problem of solving a countable sequence of Ramsey's theorem instances.  Thus, in terms of uniform computational strength, solving one instance of the Rival--Sands theorem exactly corresponds to solving countably many instances of Ramsey's theorem for pairs simultaneously in parallel.

We position the weak Rival--Sands theorem in the Weihrauch degrees in Section~\ref{sec-wRSgWeihrauch}.  We find that the uniform computational strength of the weak Rival--Sands theorem is much less than that of Ramsey's theorem for pairs, despite the two theorems being equivalent over $\rca$.  For example, the ascending/descending sequence principle, and in fact the weaker stable ascending/descending chain principle of~\cite{AstorDzhafarovSolomonSuggs}, does not Weihrauch reduce to the weak Rival--Sands theorem.  The diagonally non-recursive principle does not Weihrauch reduce to the weak Rival--Sands theorem either.  This gives a further sense in which the weak Rival--Sands theorem is a weaker formalization of Ramsey's theorem for pairs.  Nonetheless, the weak Rival--Sands theorem does exhibit non-trivial uniform computational strength.  The infinite pigeonhole principle for any finite number of colors and the cohesiveness principle both Weihrauch reduce to the weak Rival--Sands theorem.  Theorem~\ref{thm-wRSgWeiLoc} summarizes the location of the weak Rival--Sands theorem in the Weihrauch degrees.

The main results of this work are the following.
\begin{itemize}
\item The Rival--Sands theorem is equivalent to $\aca$ over $\rca$ (Theorem~\ref{thm-ACA_RS_eq}).  The forward direction of this result gives a new proof of the Rival--Sands theorem for countable graphs.

\medskip

\item The weak Rival--Sands theorem is equivalent to Ramsey's theorem for pairs over $\rca$ (Theorem~\ref{thm-RT_WRSG_eq}, with Hirst and Lempp).

\medskip

\item The infinite pigeonhole principle for three colors Weihrauch reduces to the ascending/descending chain principle (Theorem~\ref{thm-RT13vsADC}), but the infinite pigeonhole principle for five colors does not Weihrauch reduce to the ascending/descending sequence principle (Theorem~\ref{thm-RT15vsADS}).

\medskip

\item The Rival--Sands theorem is Weihrauch equivalent to the double jump of weak K\"{o}nig's lemma (Theorem~\ref{thm-RSGequivWKLjj}) and is therefore also Weihrauch equivalent to the parallelization of Ramsey's theorem for pairs.  To our knowledge, the Rival--Sands theorem is the first example of a natural theorem being equivalent to the double jump of weak K\"{o}nig's lemma.

\medskip

\item The diagonally non-recursive principle and the stable ascending/descending chain principle do not Weihrauch reduce to the weak Rival--Sands theorem, but the infinite pigeonhole principle and the cohesiveness principle do (Theorem~\ref{thm-wRSgWeiLoc}).  The weak Rival--Sands theorem may therefore be thought of as a weaker formalization of Ramsey's theorem for pairs.
\end{itemize}

We make a few final remarks concerning the Rival--Sands theorem.  First, Rival and Sands prove their theorem as stated above for graphs of all infinite cardinalities.  However, their proof produces a countable $H$ regardless of the cardinality of $G$.  Gavalec and Vojt\'{a}\v{s}~\cite{GavalecVojtas} examine possible generalizations of the Rival--Sands theorem to uncountable graphs $G$, where $|H| = |G|$ is also required in the conclusion.

Second, Rival and Sands observe that by restricting the class of graphs under consideration, it is possible to improve the conclusion of their theorem and obtain statements of the following form.  Every infinite graph $G$ from a certain class contains an infinite subset $H$ such that every vertex of $G$ is adjacent to none, or infinitely many, vertices of $H$.  The class they consider is the class of comparability graphs of partial orders of finite width, and the theorem they obtain is the following.

\begin{Theorem}[Rival and Sands~\cite{RivalSands}]\label{thm-RSpo}
Every infinite partially ordered set $P$ of finite width contains an infinite chain $C$ such that every element of $P$ is comparable with none, or infinitely many, of the elements of $C$.
\end{Theorem}

In work with Alberto Marcone~\cite{Fiori-CaronesMarconeShaferSolda}, we analyze the reverse mathematics of this Rival--Sands theorem for partial orders.

\section{Reverse mathematics preliminaries}\label{sec-RMPrelim}

We remind the reader of the essential details of mathematics in second-order arithmetic, the systems $\rca$ and $\aca$, and combinatorial principles related to Ramsey's theorem for pairs.  For further information, the standard reference for reverse mathematics is Simpson's~\cite{SimpsonSOSOA}.  A robust account of the reverse mathematics of Ramsey's theorem for pairs can be found in~\cite{HirschfeldtBook}.

The language of second-order arithmetic is two-sorted.  Objects of the first sort are thought of as natural numbers, and objects of the second sort are thought of as sets of natural numbers.  Lower-case letters $a$, $b$, $c$, $x$, $y$, $z$, etc.\ usually denote first-order variables ranging over the first sort, and capital letters $A$, $B$, $C$, $X$, $Y$, $Z$, etc.\ usually denote second-order variables ranging over the second sort.  The constant, function, and relation symbols are $0$, $1$, $<$, $+$, $\times$, and $\in$.  As is customary, we use $\Nb$ as an abbreviation to denote the first-order part of whatever structure is under consideration.

The axioms of $\rca$ (for \emph{recursive comprehension axiom}) are
\begin{itemize}
\item a first-order sentence expressing that the natural numbers form a discretely ordered commutative semi-ring with identity; 

\medskip

\item the \emph{$\Sigma^0_1$ induction scheme} (denoted $\iso$), which consists of the universal closures (by both first- and second-order quantifiers) of all formulas of the form
\begin{align*}
[\varphi(0) \andd \forall n(\varphi(n) \imp \varphi(n+1))] \imp \forall n \varphi(n),
\end{align*}
where $\varphi$ is $\Sigma^0_1$; and 

\medskip

\item the \emph{$\Delta^0_1$ comprehension scheme}, which consists of the universal closures (by both first- and second-order quantifiers) of all formulas of the form
\begin{align*}
\forall n (\varphi(n) \biimp \psi(n)) \imp \exists X \forall n(n \in X \biimp \varphi(n)),
\end{align*}
where $\varphi$ is $\Sigma^0_1$, $\psi$ is $\Pi^0_1$, and $X$ is not free in $\varphi$.
\end{itemize}
The `$0$' in `$\rca$' refers to the restriction of the induction scheme to $\Sigma^0_1$ formulas.  $\rca$ is able to implement the typical codings of finite sets and sequences of natural numbers as single natural numbers.  See~\cite{SimpsonSOSOA}*{Section~II.2} for details on how this is done.

The axioms of $\aca$ (for \emph{arithmetical comprehension axiom}) are obtained by adding the \emph{arithmetical comprehension scheme} to $\rca$.  The arithmetical comprehension scheme consists of the universal closures of all formulas of the form
\begin{align*}
\exists X \forall n(n \in X \biimp \varphi(n)),
\end{align*}
where $\varphi$ is an arithmetical formula in which $X$ is not free.

When proving that some statement implies $\aca$ over $\rca$, a common strategy is to show that the range of an arbitrary injection exists and then invoke the following well-known lemma.

\begin{Lemma}[{\cite{SimpsonSOSOA}*{Lemma~III.1.3}}]\label{lem-ACAinjection}
The following are equivalent over $\rca$.
\begin{enumerate}
\item $\aca$.

\smallskip

\item If $f \colon \Nb \imp \Nb$ is an injection, then there is a set $X$ such that $\forall n(n \in X \biimp \exists s(f(s) = n))$.
\end{enumerate}
\end{Lemma}

For a set $X \subseteq \Nb$, $[X]^2$ denotes the set of two-element subsets of $X$.  It is often convenient to think of $[X]^2$ as $\{(x,y) : x,y \in X \andd x < y\}$.  A function $c \colon [\Nb]^2 \imp \{0,1\}$ is called a \emph{$2$-coloring of pairs}.  \emph{Ramsey's theorem for pairs and two colors} states that for every $2$-coloring of pairs, there is an infinite set $H$ such that all pairs made from the members $H$ have the same color.  Such an $H$ is called \emph{homogeneous} for $c$.  Ramsey's theorem for pairs and two colors famously decomposes into the conjunction of its \emph{stable version} and the \emph{cohesive principle}.  Stable Ramsey's theorem for pairs is in turn equivalent to the principle $\Dd^2_2$, which formalizes the assertion that for every set that can be approximated in a $\bm{\Delta}^0_2$ way, there is either an infinite subset of the set or an infinite subset of the set's compliment.

\begin{Definition}{\ }
\begin{itemize}
\item \emph{Ramsey's theorem for pairs and two colors} $(\rt^2_2)$ is the following statement.  For every $c \colon [\Nb]^2 \imp \{0,1\}$, there is an infinite $H \subseteq \Nb$ and an $i \in \{0,1\}$ such that for all $(x,y) \in [H]^2$, $c(x,y) = i$.

\medskip

\item A coloring $c \colon [\Nb]^2 \imp \{0,1\}$ is \emph{stable} if $\lim_s c(n,s)$ exists for every $n \in \Nb$.

\medskip

\item \emph{Stable Ramsey's theorem for pairs and two colors} $(\srt^2_2)$ is $\rt^2_2$ restricted to stable colorings.

\medskip

\item $\Dd^2_2$ is the following statement.  For every function $f \colon \Nb \times \Nb \imp \{0,1\}$ such that $\lim_s f(n,s)$ exists for every $n \in \Nb$, there are an infinite $H \subseteq \Nb$ and an $i \in \{0,1\}$ such that for all $n \in H$, $\lim_s f(n,s) = i$.
\end{itemize}
\end{Definition}

\begin{Definition}{\ }
\begin{itemize}
\item For sets $A, C \subseteq \Nb$, $C \subseteq^* A$ denotes that $C \setminus A$ is finite, and $A =^* C$ denotes that both $C \subseteq^* A$ and $A \subseteq^* C$.

\medskip

\item Let $\vec{A} = (A_i : i \in \Nb)$ be a sequence of subsets of $\Nb$.  A set $C \subseteq \Nb$ is called \emph{cohesive for $\vec{A}$} (or simply \emph{$\vec{A}$-cohesive}) if $C$ is infinite and for each $i \in \Nb$, either $C \subseteq^* A_i$ or $C \subseteq^* \ol{A_i}$.

\medskip

\item $\coh$ is the following statement.  For every sequence $\vec{A}$ of subsets of $\Nb$, there is a set $C \subseteq \Nb$ that is cohesive for $\vec{A}$.
\end{itemize}
\end{Definition}

\begin{Proposition}[\cite{CholakJockuschSlaman}*{Lemma~7.10 and Lemma~7.11}; see also~\cite{CholakJockuschSlamanCorrigendum}]\label{prop-rt22decomp}
{\ }
\begin{enumerate}
\item\label{it-SRT22D22Equiv} $\rca \vdash \srt^2_2 \biimp \Dd^2_2$.

\smallskip

\item $\rca \vdash \rt^2_2 \biimp (\coh \andd \srt^2_2)$.

\smallskip

\item\label{it-D22Decomp} Hence $\rca \vdash \rt^2_2 \biimp (\coh \andd \Dd^2_2)$.
\end{enumerate}
\end{Proposition}

For our purposes, it is helpful to notice that in $\coh$, the cohesive set $C$ can be taken to be a subset of any given infinite set $Z$.

\begin{Lemma}\label{lem-COH_helper}
$\rca + \coh$ proves the following statement.  For every sequence $\vec{A}$ of subsets of $\Nb$ and every infinite $Z \subseteq \Nb$, there is a set $C \subseteq Z$ that is cohesive for $\vec{A}$.
\end{Lemma}

\begin{proof}
Let $Z \subseteq \Nb$ be infinite, and let $f \colon \Nb \imp Z$ be an increasing bijection.  Let $\vec{A} = (A_i : i \in \Nb)$ be a sequence of subsets of $\Nb$, and define another sequence of sets $\vec{B} = (B_i : i \in \Nb)$ by $B_i = f^{-1}(A_i)$ for each $i$.  By $\coh$, let $C$ be cohesive for $\vec{B}$.  Then $D = f(C)$ is cohesive for $\vec{A}$, and in this case $f(C) = \{n : (\exists m \leq n)(m \in C \andd f(m) = n)\}$ exists by $\Delta^0_1$ comprehension because $f$ is increasing.
\end{proof}

A useful consequence of $\rt^2_2$ is the \emph{ascending/descending sequence principle}, which states that in every infinite linear order there is either an infinite ascending sequence or an infinite descending sequence.  The principle $\ads$ in turn proves both $\coh$ and the \emph{$\Sigma^0_2$ bounding scheme} $(\bst)$.

\begin{Definition}\label{def-ADS}
Let $(L, <_L)$ be a linear order.
\begin{itemize}
\item A set $S \subseteq L$ is an \emph{ascending sequence} in $L$ if $(\forall x, y \in S)(x < y \imp x <_L y)$.

\smallskip

\item A set $S \subseteq L$ is a \emph{descending sequence} in $L$ if $(\forall x, y \in S)(x < y \imp y <_L x)$.

\smallskip

\item The \emph{ascending/descending sequence principle} ($\ads$) is the following statement.  For every infinite linear order $(L, <_L)$, there is an infinite $S \subseteq L$ that is either an ascending sequence or a descending sequence.
\end{itemize}
\end{Definition}

In the above definition, $\ads$ is stated as it is in~\cite{HirschfeldtShore}.  In practice, $\rca$ proves that any sequence $x_0, x_1, x_2, \dots$ of distinct numbers can be thinned to an increasing subsequence $x_{i_0} < x_{i_1} < x_{i_2} < \cdots$, in which case $\{x_{i_n} : n \in \Nb\}$ exists as a set.  Therefore, $\ads$ is equivalent to the following statement over $\rca$.  For every infinite linear order $(L, <_L)$, there is an infinite sequence $x_0, x_1, x_2 \dots$ of elements of $L$ such that either $x_0 <_L x_1 <_L x_2 <_L \cdots$ or $x_0 >_L x_1 >_L x_2 >_L \cdots$.

\begin{Definition}
The \emph{$\Sigma^0_2$ bounding scheme} $(\bst)$ consists of the universal closures of all formulas of the form
\begin{align*}
\forall a[(\forall n < a)(\exists m)\varphi(n,m) \imp \exists b(\forall n < a)(\exists m < b)\varphi(n,m)],
\end{align*}
where $\varphi$ is $\Sigma^0_2$ and $a$ and $b$ are not free in $\varphi$.
\end{Definition}

Over $\rca$, $\bst$ is equivalent to the \emph{$\Pi^0_1$ bounding scheme} $(\bpo)$, which is the analogous scheme in which $\varphi$ is required to be $\Pi^0_1$ instead of $\Sigma^0_2$ (see~\cite{HajekPudlak}*{Section~I.2}).  $\bst$ is closely related to the infinite pigeonhole principle (i.e., Ramsey's theorem for singletons).

\begin{Definition}\label{def-RT1SOA}
{\ }
\begin{itemize}
\item For every $k > 0$, \emph{Ramsey's theorem for singletons and $k$ colors} $(\rt^1_k)$ is the following statement.  For every $c \colon \Nb \imp \{0, 1, \dots, k-1\}$, there are an infinite $H \subseteq \Nb$ and an $i \in \{0, 1, \dots, k-1\}$ such that for all $x \in H$, $c(x) = i$.

\medskip

\item \emph{Ramsey's theorem for singletons} $(\rt^1_{<\infty})$ is the following statement.  For every $k > 0$ and every $c \colon \Nb \imp \{0, 1, \dots, k-1\}$, there are an infinite $H \subseteq \Nb$ and an $i \in \{0, 1, \dots, k-1\}$ such that for all $x \in H$, $c(x) = i$.
\end{itemize}
\end{Definition}

The system $\rca$ proves $\rt^1_k$ for each fixed, standard $k > 0$.  However, $\rt^1_{<\infty}$ is equivalent to $\bst$ over $\rca$ by a well-known theorem of Hirst~\cite{HirstThesis}.  Via the correspondence with $\rt^1_{<\infty}$, one also shows that $\bst$ is equivalent to the statement ``every finite union of finite sets is finite,'' in the sense that if $X_0, X_1, \dots, X_{k-1}$ are $k$ subsets of $\Nb$ for some $k$ and $X_i$ has an upper bound for each $i < k$, then the union $\bigcup_{i<k}X_i$ also has an upper bound.

As mentioned above, $\rt^2_2$ proves $\ads$, which proves both $\coh$ and $\bst$.

\begin{Proposition}[{\cite{HirschfeldtShore}*{Proposition~2.10 and Proposition~4.5}}]\label{prop-rt22andADS}
{\ }
\begin{enumerate}
\item $\rca + \rt^2_2 \vdash \ads$.

\smallskip

\item\label{it-ADSimpCOH} $\rca + \ads \vdash \coh$.

\smallskip

\item\label{it-ADSimpBST} $\rca + \ads \vdash \bst$.
\end{enumerate}
\end{Proposition}
\noindent
On the other hand, $\rca + \ads \nvdash \rt^2_2$;  indeed, $\rca + \ads \nvdash \srt^2_2$~\cite{HirschfeldtShore}.

Much of the work in the reverse mathematics of combinatorics has been inspired by major open, and now answered, questions concerning the strength of $\rt^2_2$.  The first breakthrough was by Seetapun and Slaman~\cite{SeetapunSlaman}, who showed that $\rca + \rt^2_2 \nvdash \aca$.  More recently, Liu~\cite{LiuRT22vsWKL} showed that $\rca + \rt^2_2 \nvdash \wklz$ and that $\rca + \rt^2_2 \nvdash \wwklz$~\cite{LiuRT22vsWWKL}.  We now also know that $\rca + \srt^2_2 \nvdash \rt^2_2$, which was first shown by Chong, Slaman, and Yang via a non-$\omega$-model~\cite{ChongSlamanYang}, and again shown by Monin and Patey~\cite{MoninPatey} via an $\omega$-model.  Furthermore, Patey and Yokoyama showed that $\rt^2_2$ is $\Pi^0_3$-conservative over $\rca$~\cite{PateyYokoyama}.

\section{The reverse mathematics of the Rival--Sands theorem}\label{sec-RMRS}

We give our reverse mathematics analysis of the Rival--Sands theorem.  In their original formulation of Theorem~\ref{thm-RSg},  Rival and Sands add that we may also require every vertex of $H$ to be adjacent to none or to infinitely many other vertices of $H$.  We call the version of the  theorem with this additional requirement the \emph{refined} version of the Rival--Sands theorem.  We show that the Rival--Sands theorem and its refined version are equivalent to $\aca$ over $\rca$ and that these equivalences remain valid when the theorem is restricted to locally finite graphs.  We then analyze the inside-only weak Rival--Sands theorem and its refined version, where the conclusion need only hold for the vertices of $H$.  Although the weak Rival--Sands theorem and its refined version look potentially weaker than $\rt^2_2$,  we show that they are indeed equivalent to $\rt^2_2$ over $\rca$.

We consider only simple graphs.  Thus we represent a graph $G = (V, E)$ by sets $V \subseteq \Nb$ and $E \subseteq [V]^2$, where we think of $[V]^2$ as encoded as $\{(x,y) : x, y \in V \andd x < y\}$. We slightly abuse notation by writing $(x,y) \in E$ as an abbreviation for $(\min\{x,y\}, \max\{x,y\}) \in E$, meaning that $\{x,y\}$ is an edge of the graph.  For a graph $G = (V,E)$ and an $x \in V$, $N(x) = \{y \in V : (x,y) \in E\}$ denotes the set of \emph{neighbors} of $x$.  Recall that a set $H \subseteq V$ is a \emph{clique} if $(x,y) \in E$ for every pair of distinct $x, y \in H$ and that $H$ is an \emph{independent set} if $(x, y) \notin E$ for every $x, y \in H$.

The precise formulations of the Rival--Sands theorem and its weak version that we consider are given below.  Our terminology and notation emphasize that these are formulations of the Rival--Sands theorem \emph{for graphs} to distinguish Theorem~\ref{thm-RSg} from Theorem~\ref{thm-RSpo}.  The reverse mathematics of the Rival--Sands theorem for partial orders (i.e., Theorem~\ref{thm-RSpo}) is studied in~\cite{Fiori-CaronesMarconeShaferSolda}.
\begin{Definition}{\ }
\begin{itemize}
\item \emph{The Rival--Sands theorem for graphs} $(\rsg)$ is the following statement.  For every infinite graph $G = (V,E)$, there is an infinite $H \subseteq V$ such that for every $x \in V$, either $H \cap N(x)$ is infinite or $|H \cap N(x)| \leq 1$.

\medskip

\item \emph{The Rival--Sands theorem for graphs, refined} $(\rsgr)$ is the following statement.  For every infinite graph $G = (V,E)$, there is an infinite $H \subseteq V$ such that
\begin{itemize}
\item for every $x \in V$, either $H \cap N(x)$ is infinite or $|H \cap N(x)| \leq 1$; and moreover

\smallskip

\item for every $x \in H$, either $H \cap N(x)$ is infinite or $H \cap N(x) = \emptyset$.
\end{itemize}

\medskip

\item \emph{The weak Rival--Sands theorem for graphs} $(\wrsg)$ is the following statement.  For every infinite graph $G = (V,E)$, there is an infinite $H \subseteq V$ such that for every $x \in H$, either $H \cap N(x)$ is infinite or $|H \cap N(x)| \leq 1$.

\medskip

\item \emph{The weak Rival--Sands theorem for graphs, refined} $(\wrsgr)$ is the following statement.  For every infinite graph $G = (V,E)$, there is an infinite $H \subseteq V$ such that for every $x \in H$, either $H \cap N(x)$ is infinite or $H \cap N(x) = \emptyset$.
\end{itemize}
\end{Definition}

As mentioned in the introduction, the original proof of $\rsg$ by Rival and Sands involves detailed elementary reasoning that can be formalized in $\aca$ with some engineering.  We give a quick new proof using cohesive sets.

\begin{Theorem}\label{thm-RSgInACA}
$\aca \vdash \text{The Rival--Sands theorem for graphs, refined $(\rsgr)$}$.
\end{Theorem}

\begin{proof}
Let $G = (V,E)$ be an infinite graph.  Let $F = \{x \in V : \text{$N(x)$ is finite}\}$, which may be defined in $\aca$.  There are two cases, depending on whether or not $F$ is finite.  If $F$ is finite, simply take
\begin{align*}
H = V \setminus \bigcup_{x \in F}N(x).
\end{align*}
In this case, the union $\bigcup_{x \in F}N(x)$ is finite because we work in $\aca$, $\aca$ proves $\bst$, and $\bst$ proves that a finite union of finite sets is finite as discussed following Definition~\ref{def-RT1SOA}.  Thus $H$ contains almost every member of $V$.  Consider an $x \in V$.  If $x \in F$, then $H \cap N(x) = \emptyset$.  If $x \notin F$, then $N(x)$ is infinite, so $H \cap N(x)$ is also infinite.  So in this case, for every $x \in V$, either $H \cap N(x)$ is infinite or $H \cap N(x) = \emptyset$.

Suppose instead that $F$ is infinite.  $\aca$ proves $\coh$, so by Lemma~\ref{lem-COH_helper}, let $C \subseteq F$ be cohesive for the sequence $(N(x) : x \in V)$.  As we work in $\aca$, we may also define a function $f \colon V \imp \{0,1\}$ by
\begin{align*}
f(x) = 
\begin{cases}
0 & \text{if $C \subseteq^* \ol{N(x)}$}\\
1 & \text{if $C \subseteq^* N(x)$}.
\end{cases}
\end{align*}
Define $H = \{x_0, x_1, \dots\} \subseteq C \subseteq F$ by the following procedure.  Let $x_0$ be the first element of $C$.  Suppose that $x_0 < x_1 < \cdots < x_n$ have been defined.  Let $Y = \bigcup_{i \leq n}N(x_i)$, which is finite by $\bst$ because each $x_i$ is in $F$.  For each $y \in Y$, if $f(y) = 0$, then $C \subseteq^* \ol{N(y)}$; and if $f(y) = 1$, then $C \subseteq^* N(y)$.  Thus by $\bst$, there is a bound $b$ such that for all $y \in Y$ and all $z \in C$ with $z > b$, if $f(y) = 0$ then $z \in \ol{N(y)}$ and if $f(y) = 1$, then $z \in N(y)$.  Choose $x_{n+1}$ to be the first member of $C \setminus Y$ with $x_n < x_{n+1}$ and such that, for every $y \in Y$, if $f(y) = 0$, then $x_{n+1} \in \ol{N(y)}$; and if $f(y) = 1$, then $x_{n+1} \in N(y)$.  This completes the construction.

To verify that $H$ satisfies the conclusion of $\rsgr$, consider a $v \in V$.  If $H \cap N(v) \neq \emptyset$, let $m$ be least such that $x_m \in N(v)$ (and hence also least such that $v \in N(x_m)$).  If $f(v) = 0$, then every $x_n$ with $n > m$ is chosen from $\ol{N(v)}$, so $|H \cap N(v)| = 1$.  If $f(v) = 1$, then every $x_n$ with $n > m$ is chosen from $N(v)$, so $H \cap N(v)$ is infinite.  Thus for every $v \in V$, either $H \cap N(v)$ is infinite or $|H \cap N(v)| \leq 1$.  Furthermore, if $v \in H$, then $H \cap N(v) = \emptyset$ because if $m < n$, then $x_n$ is chosen from $\ol{N(x_m)}$.
\end{proof}

Before giving the reversal for the Rival--Sands theorem, we observe that $\rca$ suffices to prove its refined version for highly recursive graphs.

\begin{Definition}
For a set $X \subseteq \Nb$, let $\Pf(X)$ denote the set of (coded) finite subsets of $X$.
\begin{itemize}
\item A graph $G = (V,E)$ is \emph{locally finite} if $N(x)$ is finite for each $x \in V$.

\smallskip

\item A graph $G = (V,E)$ is \emph{highly recursive} if it is locally finite, and additionally there is a function $b \colon V \imp \Pf(V)$ such that $b(x) = N(x)$ for each $x \in V$.
\end{itemize}
\end{Definition}

It would be more natural to use the word \emph{bounded} in place of \emph{highly recursive}, but there is already an unrelated notion of a \emph{bounded graph}.  Every highly recursive graph is locally finite by definition.  That every locally finite graph is highly recursive requires $\aca$ in general.

\begin{Proposition}\label{prop-RSgHighRec}
$\rca \vdash \text{The Rival--Sands theorem for highly recursive graphs, refined}$.
\end{Proposition}

\begin{proof}
Let $G = (V,E)$ be a highly recursive infinite graph, and let $b \colon V \imp \Pf(V)$ be such that $b(x) = N(x)$ for all $x \in V$.  Define an infinite $H = \{x_0, x_1, \cdots \} \subseteq V$ with $x_0 < x_1 < \cdots$ as follows.  Let $x_0$ be the first member of $V$.  Given $x_0 < x_1 < \cdots < x_n$, let $Y$ be the finite set
\begin{align*}
Y = \{x_i : i \leq n\} \cup \bigcup_{i \leq n}b(x_i) \cup \bigcup_{\substack{i \leq n\\y \in b(x_i)}}b(y)
\end{align*}
consisting of all vertices that are of distance $\leq\! 2$ from an $x_i$ with $i \leq n$.  Notice that $Y$ is a union of finitely many coded finite sets, so $\rca$ suffices to prove that $Y$ is finite (and in fact is a coded finite set itself).  Choose $x_{n+1}$ to be the first member of $V \setminus Y$ with $x_n < x_{n+1}$.  Then no two distinct members of $H$ are of distance $\leq\! 2$, so $H$ satisfies the conclusion of $\rsgr$.
\end{proof}

\begin{Theorem}\label{thm-ACA_RS_eq}
The following are equivalent over $\rca$.
\begin{enumerate}
\item\label{it-aca} $\aca$.

\smallskip

\item\label{it-rsg} The Rival--Sands theorem for graphs $(\rsg)$.

\smallskip

\item\label{it-rsgr} The Rival--Sands theorem for graphs, refined $(\rsgr)$.

\smallskip

\item\label{it-rsgLocFin} The Rival--Sands theorem for locally finite graphs.

\smallskip

\item\label{it-rsgrLocFin} The Rival--Sands theorem for locally finite graphs, refined.
\end{enumerate}
\end{Theorem}

\begin{proof}
Notice that (\ref{it-rsgr}) trivially implies (\ref{it-rsg}), (\ref{it-rsgLocFin}), and (\ref{it-rsgrLocFin}).  Therefore (\ref{it-aca}) implies (\ref{it-rsg})--(\ref{it-rsgrLocFin}) by Theorem~\ref{thm-RSgInACA}.  Notice also that (\ref{it-rsg}), (\ref{it-rsgr}), and (\ref{it-rsgrLocFin}) each trivially imply (\ref{it-rsgLocFin}).  Thus to finish the proof, it suffices to show that (\ref{it-rsgLocFin}) implies (\ref{it-aca}) over $\rca$.

By Lemma~\ref{lem-ACAinjection}, it suffices to show that $\rsg$ for locally finite graphs implies that the ranges of injections exist.  Thus let $f \colon \Nb \imp \Nb$ be an injection.  Let $G  = (\Nb, E)$ be the graph where $E = \{(v,s) \in [\Nb]^2 : (v < s) \andd (f(s) < f(v))\}$, which exists by $\Delta^0_1$ comprehension.  To see that $G$ is locally finite, consider a $v \in \Nb$.  The function $f$ is injective, so there are only finitely many $s > v$ with $f(s) < f(v)$.  Therefore there are only finitely many $s > v$ that are adjacent to $v$.

Apply $\rsg$ for locally finite graphs to $G$ to get an infinite $H \subseteq \Nb$ such that $|H \cap N(x)| \leq 1$ for every $x \in \Nb$.  Enumerate $H$ in increasing order as $x_0 < x_1 < x_2 < \cdots$.  We show that, for any $n \in \Nb$, if $\exists s (f(s) = n)$, then $(\exists s \leq x_{n+1})(f(s) = n)$.  Suppose that $f(s) = n$.  Then $s$ is adjacent to all but at most $n$ of the vertices $v < s$.  This is because if $v < s$, then $(v, s) \notin E$ if and only if $f(v) \leq f(s)$.  The function $f$ is an injection, so there are at most $n = f(s)$ many vertices $v < s$ with $f(v) \leq f(s)$.  Thus there are at most $n$ vertices $v < s$ to which $s$ is not adjacent.  At most one neighbor of $s$ is in $H$, and therefore there are at most $n+1$ many vertices in $H$ that are $<\! s$.  Thus $x_{n+1} \geq s$.  Thus $n$ is in the range of $f$ if and only if $(\exists s \leq x_{n+1})(f(s) = n)$.  So the range of $f$ exists by $\Delta^0_1$ comprehension.
\end{proof}

We finish this section by showing that both the weak Rival--Sands theorem and its refined version are equivalent to $\rt^2_2$ over $\rca$.  This was proved in collaboration with Jeffry Hirst and Steffen Lempp.

\begin{Theorem}[Fiori-Carones, Hirst, Lempp, Shafer, Sold\`a]\label{thm-RT_WRSG_eq}
The following are equivalent over $\rca$.
\begin{enumerate}
\item\label{it-rt} $\rt^2_2$.

\smallskip

\item\label{it-wrsg} The weak Rival--Sands theorem $(\wrsg)$.

\smallskip

\item\label{it-wrsgr} The weak Rival--Sands theorem, refined $(\wrsgr)$.
\end{enumerate}
\end{Theorem}

\begin{proof}
For an infinite graph $G$, every infinite homogeneous set satisfies the conclusion of $\wrsgr$, so (\ref{it-rt}) implies (\ref{it-wrsgr}).  Trivially (\ref{it-wrsgr}) implies (\ref{it-wrsg}).  It remains to show that (\ref{it-wrsg}) implies (\ref{it-rt}) over $\rca$.

We show that $\rca + \wrsg \vdash \ads \andd \Dd^2_2$, from which it follows that $\rca + \wrsg \vdash \rt^2_2$ by Proposition~\ref{prop-rt22decomp} item~(\ref{it-D22Decomp}) and Proposition~\ref{prop-rt22andADS} item~(\ref{it-ADSimpCOH}).  We start by showing that $\rca + \wrsg \vdash \ads$.

Let $L = (L, <_L)$ be an infinite linear order.  Let $G = (L, E)$ be the graph with $E = \{(x, y) \in [L]^2 : (x < y) \andd (x <_L y)\}$.  Let $H$ be as in the conclusion of $\wrsg$ for $G$.  Then for every $x \in H$, either $H \cap N(x)$ is infinite or $|H \cap N(x)| \leq 1$.

First suppose that $|H \cap N(x)| \leq 1$ for all $x \in H$.  Then $H$ has no $<_L$-minimum element.  Given $x \in H$, let $y_0, y_1 \in H$ be such that $x < y_0, y_1$.  Then at most one of $(x, y_0)$ and $(x, y_1)$ is in $E$, so either $y_0 <_L x$ or $y_1 <_L x$.  Define a descending sequence $x_0 >_L x_1 >_L x_2 >_L \cdots$ by choosing $x_0$ to be the first member of $H$ and by choosing each $x_{n+1}$ to be the first member of $H$ that is $<_L$-below $x_n$.

Now suppose that $H \cap N(x)$ is infinite for some $x \in H$, but further suppose that $|H \cap N(y)| \leq 1$ for all but finitely many $y \in H \cap N(x)$.  Let $b$ be a bound such that $|H \cap N(y)| \leq 1$ whenever $y \in H \cap N(x)$ and $y > b$.  Let $y_0 < y_1 < y_2 < \cdots$ enumerate in $<$-increasing order the elements of $H \cap N(x)$ that are $> b$.  Then $y_0 >_L y_1 >_L y_2 >_L \cdots$ is a descending sequence in $L$.  This is because if $y_n <_L y_{n+1}$ for some $n$, then $(y_n, y_{n+1}) \in E$, so both $x$ and $y_{n+1}$ are in $H \cap N(y_n)$, which is a contradiction.

Finally, suppose that there is an $x \in H$ with $H \cap N(x)$ infinite and furthermore that whenever $x \in H$ and $H \cap N(x)$ is infinite, then also $H \cap N(y)$ is infinite for infinitely many $y \in H \cap N(x)$.  We define an ascending sequence $x_0 <_L x_1 < x_2 <_L \cdots$ where $x_n \in H$ and $H \cap N(x_n)$ is infinite for each $n$.  Recall that for $x \in H$, $H \cap N(x)$ is infinite if and only if $|H \cap N(x)| \geq 2$ because $H$ satisfies the conclusion of $\wrsg$.  Let $x_0$ be any element of $H$ with $|H \cap N(x_0)| \geq 2$.  Given $x_n \in H$ with $|H \cap N(x_n)| \geq 2$, we know by assumption that there are infinitely many $y \in H \cap N(x_n)$ with $|H \cap N(y)| \geq 2$.  Let $\la y, w, z \ra$ be the first (code for a) triple with $y \in H \cap N(x_n)$, $x_n < y$, $w \neq z$, and $w, z \in H \cap N(y)$.  Then $x_n <_L y$ because $x_n < y$ and $(x_n,y) \in E$, so put $x_{n+1} = y$.  This completes the proof of $\ads$.

Now we show that $\rca + \wrsg \vdash \Dd^2_2$.  Let $f \colon \Nb \times \Nb \imp \{0,1\}$ be such that $\lim_s f(n,s)$ exists for every $n \in \Nb$, and let $G = (\Nb, E)$ be the graph with $E = \{(x,y) \in [\Nb]^2 : (x < y) \andd (f(x,y) = 1)\}$.  Let $H$ be as in the conclusion of $\wrsg$ for $G$.  Thus for every $x \in H$, $H \cap N(x)$ is infinite if and only if $|H \cap N(x)| \geq 2$.  Furthermore, for every $x \in \Nb$, either $\lim_s f(x,s) = 0$, in which case $N(x)$ is finite and hence $H \cap N(x)$ is finite; or $\lim_s f(x,s) = 1$, in which case $N(x)$ is co-finite and hence $H \cap N(x)$ is infinite.  Putting this together yields that for all $x \in H$, $\lim_s f(x,s) = 0$ if and only if $|H \cap N(x)| \leq 1$.

First suppose that there are only finitely many $x \in H$ with $|H \cap N(x)| \geq 2$.  Let $b$ be a bound such that $|H \cap N(x)| \leq 1$ whenever $x \in H$ and $x > b$, and let $K = \{x \in H : x > b\}$.  Then $K$ is infinite, and $\lim_s f(x,s) = 0$ for all $x \in K$.  Thus $K$ satisfies the conclusion of $\Dd^2_2$ for $f$.

Finally, suppose instead that there are infinitely many $x \in H$ with $|H \cap N(x)| \geq 2$.  Define an infinite set $K = \{x_0, x_1, \dots\}$, where $x_n \in H$ and $|H \cap N(x_n)| \geq 2$ for each $n$ as follows.  Let $x_0$ be the first element of $H$ with $|H \cap N(x_0)| \geq 2$.  Given $x_0 < x_1 < \cdots < x_n$, let $\la y, w, z \ra$ be the first (code for a) triple with $x_n < y$, $y \in H$, $w \neq z$, and $w, z \in H \cap N(y)$.  Then put $x_{n+1} = y$.  The set $K$ is infinite, and every $x \in K$ satisfies $x \in H$ and $|H \cap N(x)| \geq 2$.  Thus every $x \in K$ satisfies $\lim_s f(x,s) = 1$, so $K$ satisfies the conclusion of $\Dd^2_2$ for $f$.  This completes the proof of $\Dd^2_2$.
\end{proof}

\section{Weihrauch degrees preliminaries}\label{sec-WPrelim}

The Weihrauch degrees classify mathematical theorems in terms of their uniform computational content.  Many theorems from analysis, topology, and combinatorics have been studied in this setting.  We refer the reader to~\cite{BrattkaGherardiPaulySurvey} for an introduction to the subject.

We think in terms of multi-valued partial functions $\mc{F} \colon \wsubseteq \mc{X} \rra \mc{Y}$ between two sets (in this context called \emph{spaces}) $\mc{X}$ and $\mc{Y}$.  Here the `$\subseteq\! \mc{X}$' notation denotes that the domain of $\mc{F}$ is a subset of $\mc{X}$, and the `$\rra$' notation denotes that $\mc{F}$ is multi-valued.  Multi-valued partial functions $\mc{F}$ are thought of as representing mathematical problems.  Inputs $x \in \dom(\mc{F})$ are thought of as instances of the problem, and the possible outputs $y$ to $\mc{F}(x)$ (written $y \in \mc{F}(x)$) are thought of as the valid solutions to instance $x$.  Of course, a multi-valued partial function $\mc{F} \colon \wsubseteq \mc{X} \rra \mc{Y}$ is nothing more than a relation $\mc{F} \subseteq \mc{X} \times \mc{Y}$ (or a partial function $\mc{F} \colon \wsubseteq \mc{X} \imp \mc{P}(\mc{Y}) \setminus \{\emptyset\}$).  However, for working with computations it is helpful to think in terms of inputs and outputs in the way described above.  Henceforth we drop the word \emph{partial} for simplicity and refer to multi-valued partial functions as \emph{multi-valued functions}.

Any mathematical statement of the form $(\forall x \in \mc{X})(\varphi(x) \imp (\exists y \in \mc{Y})\psi(x,y))$ may be thought of as a multi-valued function $\mc{F} \colon \wsubseteq \mc{X} \rra \mc{Y}$ and hence as a mathematical problem.  The domain of $\mc{F}$ is $\{x \in \mc{X} : \varphi(x)\}$, and, for $x \in \dom(\mc{F})$, $\mc{F}(x)$ consists of all $y \in \mc{Y}$ such that $\psi(x,y)$.  For example, we may think of \emph{weak K\"{o}nig's lemma} as the multi-valued function $\wkl \colon \wsubseteq \mc{X} \rra \mc{Y}$ where $\mc{X}$ is the space of subtrees of $2^{<\omega}$, $\dom(\wkl)$ consists of the infinite trees in $\mc{X}$, and for $T \in \dom(\wkl)$, $\wkl(T)$ consists of all infinite paths through $T$.  Generally we present multi-valued functions in the style of the following definition.

\begin{Definition}
$\wkl$ is the following multi-valued function.
\begin{itemize}
\item Input/instance:  An infinite tree $T \subseteq 2^{<\omega}$.

\smallskip

\item Output/solution:  An $f \in 2^\omega$ that is a path through $T$.
\end{itemize}
\end{Definition}
Notice that we now use $\omega$ to denote the natural numbers.  This is because we consider only the standard natural numbers when working in the Weihrauch degrees.

For two multi-valued functions $\mc{F} \colon \wsubseteq \mc{X} \rra \mc{Y}$ and $\mc{G} \colon \wsubseteq \mc{W} \rra \mc{Z}$, $\mc{F}$ \emph{Weihrauch reduces} to $\mc{G}$ if there are uniform computational procedures for encoding any $\mc{F}$-instance $x \in \dom(\mc{F})$ into a $\mc{G}$-instance $w \in \dom(\mc{G})$ and decoding any $\mc{G}$-solution $z \in \mc{G}(w)$ along with the original $x$ into an $\mc{F}$-solution $y \in \mc{F}(x)$.  To make this notion precise, the points of spaces $\mc{X}$, $\mc{Y}$, $\mc{W}$, and $\mc{Z}$ must be represented in ways that Turing functionals can access, and $\mc{F}$ and $\mc{G}$ must be thought of in terms of the functions they induce on the representations of their inputs and outputs.  In its full glory, Weihrauch reducibility between multi-valued functions on represented spaces is defined as follows.

\begin{Definition}\label{def-WredFull}{\ }
\begin{itemize}
\item Let $\mc{X}$ be any space.  A \emph{representation} of $\mc{X}$ is a surjective partial function $\delta \colon \wsubseteq \omega^\omega \imp \mc{X}$.  A space together with a representation $(\mc{X}, \delta)$ is called a \emph{represented space}.

\medskip

\item Let $\mc{F} \colon \wsubseteq (\mc{X}, \delta_{\mc{X}}) \rra (\mc{Y}, \delta_{\mc{Y}})$ be a multi-valued function between represented spaces.  A partial function $F \colon \wsubseteq \omega^\omega \imp \omega^\omega$ \emph{realizes} $\mc{F}$ (or is a \emph{realizer} for $\mc{F}$) if $\delta_{\mc{Y}}(F(p)) \in \mc{F}(\delta_{\mc{X}}(p))$ for all $p \in \dom(\mc{F} \circ \delta_{\mc{X}})$.
\end{itemize}

\medskip

Now let $\mc{F}$ and $\mc{G}$ be two multi-valued functions on represented spaces.

\begin{itemize}
\item $\mc{F}$ \emph{Weihrauch reduces} to $\mc{G}$ (written $\mc{F} \leqW \mc{G}$) if there are Turing functionals $\Phi, \Psi \colon \wsubseteq \omega^\omega \imp \omega^\omega$ such that the functional $p \mapsto \Psi(\la p, G(\Phi(p)) \ra)$ is a realizer for $\mc{F}$ whenever $G$ is a realizer for $\mc{G}$.

\medskip

\item $\mc{F}$ and $\mc{G}$ are \emph{Weihrauch equivalent} (written $\mc{F} \equivW \mc{G}$) if $\mc{F} \leqW \mc{G}$ and $\mc{G} \leqW \mc{F}$.  In this case, $\mc{F}$ and $\mc{G}$ are said to have the same \emph{Weihrauch degree}.

\medskip

\item $\mc{F}$ \emph{strongly Weihrauch reduces} to $\mc{G}$ (written $\mc{F} \leqsW \mc{G}$) if there are Turing functionals $\Phi, \Psi \colon \wsubseteq \omega^\omega \imp \omega^\omega$ such that the functional $p \mapsto \Psi(G(\Phi(p)))$ is a realizer for $\mc{F}$ whenever $G$ is a realizer for $\mc{G}$.

\medskip

\item $\mc{F}$ and $\mc{G}$ are \emph{strongly Weihrauch equivalent} (written $\mc{F} \equivsW \mc{G}$) if $\mc{F} \leqsW \mc{G}$ and $\mc{G} \leqsW \mc{F}$.  In this case, $\mc{F}$ and $\mc{G}$ are said to have the same \emph{strong Weihrauch degree}.
\end{itemize}
\end{Definition}
In the above definition, $\la f, g \ra$ denotes the usual pairing of functions $f, g \in \omega^\omega$, where  $\la f, g \ra(2n) = f(n)$ and $\la f, g \ra(2n+1) = g(n)$ for each $n \in \omega$.  The paired function $\la f, g \ra$ is often also denoted by $f \oplus g$.

Both $\leqW$ and $\leqsW$ are quasi-orders, so both $\equivW$ and $\equivsW$ are equivalence relations.  We emphasize that the difference between Weihrauch reductions and strong Weihrauch reductions is that in strong Weihrauch reductions, the decoding functional $\Psi$ is not given explicit access to the representation $p$ of the original $\mc{F}$-input.  For multi-valued functions $\mc{F}$ and $\mc{G}$, if $\mc{F} \leqsW \mc{G}$, then also $\mc{F} \leqW \mc{G}$; and therefore if $\mc{F} \equivsW \mc{G}$, then also $\mc{F} \equivW \mc{G}$.

The mathematical problems we consider all involve spaces of countable combinatorial objects (such as subtrees of $2^{<\omega}$ in the case of $\wkl$) that can be coded as elements of $\omega^\omega$ in straightforward ways.  We therefore follow the style of~\cite{DoraisDzhafarovHirstMiletiShafer} by treating representations implicitly and restricting to multi-valued functions $\mc{F} \colon \wsubseteq \omega^\omega \rra \omega^\omega$, thereby dispensing with some of the notational encumbrances of Definition~\ref{def-WredFull}.  In this context, Weihrauch reducibility and strong Weihrauch reducibility may be defined as follows.

\begin{Definition}[see~\cite{DoraisDzhafarovHirstMiletiShafer}*{Definition~1.5} and~\cite{DoraisDzhafarovHirstMiletiShafer}*{Appendix~A}]\label{def-Wred}
Let $\mc{F}, \mc{G} \colon \wsubseteq \omega^\omega \rra \omega^\omega$ be multi-valued functions.  
\begin{itemize}
\item $\mc{F} \leqW \mc{G}$ if there are Turing functionals $\Phi, \Psi \colon \wsubseteq \omega^\omega \imp \omega^\omega$ such that $\Phi(f) \in \dom(\mc{G})$ for all $f \in \dom(\mc{F})$ and $\Psi(\la f, g \ra) \in \mc{F}(f)$ for all $f \in \dom(\mc{F})$ and $g \in \mc{G}(\Phi(f))$.

\medskip

\item $\mc{F} \leqsW \mc{G}$ if there are Turing functionals $\Phi, \Psi \colon \wsubseteq \omega^\omega \imp \omega^\omega$ such that $\Phi(f) \in \dom(\mc{G})$ for all $f \in \dom(\mc{F})$ and $\Psi(g) \in \mc{F}(f)$ for all $f \in \dom(\mc{F})$ and $g \in \mc{G}(\Phi(f))$.
\end{itemize}
\end{Definition}

Our analysis of the weak Rival--Sands theorem in Section~\ref{sec-wRSgWeihrauch} also considers the non-uniform analog of Weihrauch reducibility, which is called \emph{computable reducibility} and is defined as follows for multi-valued functions $\mc{F}, \mc{G} \colon \wsubseteq \omega^\omega \rra \omega^\omega$.

\begin{Definition}[see~\cite{DzhafarovStrongRed}*{Definition~1.1}]
Let $\mc{F}, \mc{G} \colon \wsubseteq \omega^\omega \rra \omega^\omega$ be multi-valued functions.
\begin{itemize}
\item $\mc{F}$ \emph{computably reduces} to $\mc{G}$ (written $\mc{F} \leqc \mc{G}$) if for every $f \in \dom(\mc{F})$ there is an $\wh{f} \leqT f$ with $\wh{f} \in \dom(\mc{G})$ such that for every $\wh{g} \in \mc{G}(\wh{f})$ there is a $g \leqT f \oplus \wh{g}$ with $g \in \mc{F}(f)$.

\medskip

\item $\mc{F}$ and $\mc{G}$ are \emph{computably equivalent} (written $\mc{F} \equivc \mc{G}$) if $\mc{F} \leqc \mc{G}$ and $\mc{G} \leqc \mc{F}$.  In this case, $\mc{F}$ and $\mc{G}$ are said to have the same \emph{computable degree}.
\end{itemize}
\end{Definition}

For certain multi-valued functions $\mc{F}$, it is possible to encode both a given $p \in \omega^\omega$ and a given $f \in \dom(\mc{F})$ into another $g \in \dom(\mc{F})$ in such a way that every element of $\mc{F}(g)$ encodes both the given $p$ and an element of $\mc{F}(f)$.  Such an $\mc{F}$ is called a \emph{cylinder} and has the property that every multi-valued function that Weihrauch reduces to $\mc{F}$ automatically strongly Weihrauch reduces to $\mc{F}$.

\begin{Definition}{\ }
\begin{itemize}
\item Let $\mc{F}$ and $\mc{G}$ be multi-valued functions.  Then $\mc{F} \times \mc{G}$ is the following multi-valued function.
\begin{itemize}
\item Input/instance:  A pair $\la f, g \ra \in \dom(\mc{F}) \times \dom(\mc{G})$.

\smallskip

\item Output/solution:  An element of $\mc{F}(f) \times \mc{G}(g)$.
\end{itemize}

\medskip

\item Let $\id \colon \omega^\omega \imp \omega^\omega$ denote the identity function on $\omega^\omega$.  A multi-valued function $\mc{F}$ is a \emph{cylinder} if $\id \times \mc{F} \leqsW \mc{F}$, in which case $\mc{F} \equivsW \id \times \mc{F}$.
\end{itemize}
\end{Definition}

It is straightforward to check that if $\mc{F}$ and $\mc{G}$ are multi-valued functions where $\mc{F}$ is a cylinder and $\mc{G} \equivsW \mc{F}$, then $\mc{G}$ is also a cylinder.

\begin{Proposition}[\cite{BrattkaGherardi}*{Corollary~3.6}]\label{prop-cylinder}
Let $\mc{F}$ be a multi-valued function that is a cylinder.  Then for every multi-valued function $\mc{H}$, $\mc{H} \leqW \mc{F}$ if and only if $\mc{H} \leqsW \mc{F}$.
\end{Proposition}

Our characterization of the Rival--Sands theorem in the Weihrauch degrees relies on the important notion of the \emph{jump} of a multi-valued function $\mc{F}$, which is obtained by replacing $\dom(\mc{F})$ by functions approximating elements of $\dom(\mc{F})$.  The jump of a multi-valued function was originally defined in~\cite{BrattkaGherardiMarcone}.  We present the definition using terminology that we find convenient.

\begin{Definition}{\ }
\begin{itemize}
\item A \emph{$\mbf{\Delta}^0_2$-approximation} to a function $g \colon \omega \imp \omega$ is a function $f \colon \omega \times \omega \imp \omega$ such that $\forall n(g(n) = \lim_s f(n, s))$.  Let $\lim f = g$ denote that $f$ is a $\mbf{\Delta}^0_2$-approximation to $g$.

\medskip

\item A \emph{$\mbf{\Delta}^0_2$-approximation} to a set $Z \subseteq \omega$ is a $\mbf{\Delta}^0_2$-approximation to the characteristic function of $Z$.
\end{itemize}
\end{Definition}

\begin{Definition}
Let $\mc{F} \colon \wsubseteq \omega^\omega \rra \omega^\omega$ be a multi-valued function.  The \emph{jump} of $\mc{F}$, denoted $\mc{F}'$, is the following multi-valued function.
\begin{itemize}
\item Input/instance:  A $\mbf{\Delta}^0_2$-approximation $f$ to a $g \in \dom(\mc{F})$.

\smallskip

\item Output/solution:  An element of $\mc{F}(g)$.
\end{itemize}
\end{Definition}

For example, $\wkl'$ is the multi-valued function whose inputs are $\mbf{\Delta}^0_2$-approximations to infinite trees $T \subseteq 2^{<\omega}$ and whose outputs are paths through the approximated input trees.  Brattka, Gherardi, and Marcone showed that $\wkl'$ is strongly Weihrauch equivalent to several versions of the Bolzano--Weierstra{\ss} theorem~\cite{BrattkaGherardiMarcone}.

In Section~\ref{sec-wRSgWeihrauch}, we consider the limit operation as a function in its own right.

\begin{Definition}
$\lim$ is the following function.
\begin{itemize}
\item Input/instance:  An $f \in \omega^{\omega \times \omega}$ such that $\lim f$ exists.

\smallskip

\item Output/solution: $\lim f$.
\end{itemize}
\end{Definition}

Jumps preserve strong Weihrauch reducibility and therefore preserve strong Weihrauch equivalence.

\begin{Proposition}[\cite{BrattkaGherardiMarcone}*{Proposition~5.6}]\label{prop-WJumpPres}
For any multi-valued functions $\mc{F}$ and $\mc{G}$, if $\mc{F} \leqsW \mc{G}$, then $\mc{F}' \leqsW \mc{G}'$.  Consequently, if $\mc{F} \equivsW \mc{G}$, then $\mc{F}' \equivsW \mc{G}'$.
\end{Proposition}

We find it convenient to rephrase the jumps of $\wkl$ in terms of the jumps of the problem of producing $\{0,1\}$-valued diagonally non-recursive functions.  For multi-valued functions $\mc{F} \colon \wsubseteq \omega^\omega \rra \omega^\omega$, $\mc{F}^{(n)}$ denotes the $n$\textsuperscript{th} jump of $\mc{F}$, with $\mc{F}^{(0)} = \mc{F}$, $\mc{F}^{(1)} = \mc{F}'$, $\mc{F}^{(2)} = \mc{F}''$, and so on.  For functions $f \in \omega^\omega$, $f'$ denotes the usual Turing jump of $f$, and $f^{(n)}$ denotes the $n$\textsuperscript{th} Turing jump of $f$.  Let $(\Phi_e : e \in \omega)$ denote the usual effective list of Turing functionals.

\begin{Definition}{\ }
\begin{itemize}
\item A function $f \in \omega^\omega$ is $\tdnr$ (for \emph{diagonally non-recursive}) relative to a function $p \in \omega^\omega$ if $\forall e(\Phi_e(p)(e)\da \imp f(e) \neq \Phi_e(p)(e))$.

\medskip

\item A function $f \in \omega^\omega$ is $\tdnr_2$ (for \emph{2-bounded diagonally non-recursive}) relative to a function $p \in \omega^\omega$ if $f$ is $\tdnr$ relative to $p$ and additionally $\ran(f) \subseteq \{0,1\}$.
\end{itemize}
\end{Definition}

\begin{Definition}{\ }
\begin{itemize}

\item $\pdnr$ is the following multi-valued function.
\begin{itemize}
\item Input/instance:  A function $p \in \omega^\omega$.

\smallskip

\item Output/solution:  A function $f \in \omega^\omega$ that is $\tdnr$ relative to $p$.
\end{itemize}

\medskip

\item For each $n \in \omega$, $\npdnr{(n+1)}_2$ is the following multi-valued function.
\begin{itemize}
\item Input/instance:  A function $p \in \omega^\omega$.

\smallskip

\item Output/solution:  A function $f \in 2^\omega$ that is $\tdnr_2$ relative to $p^{(n)}$.
\end{itemize}
When $n = 0$, we may write `$\pdnr_2$' instead of `$\npdnr{1}_2$'.
\end{itemize}
\end{Definition}

The $\npdnr{n}_2$ notation used here is intended to follow the $n$-$\mathsf{MLR}$ and $\npdnr{n}$ notations from the reverse mathematics literature.  The formal difference between $\pdnr_2^{(n)}$ and $\npdnr{(n+1)}_2$ is that $\pdnr_2^{(n)}$ takes as input a $\mbf{\Delta}^0_{n+1}$-approximation to a function $p$ (in the sense of $n$ iterated limits) and outputs a function that is $\tdnr_2$ relative to $p$, whereas $\npdnr{(n+1)}_2$ takes as input a function $q$ and outputs a function that is $\tdnr_2$ relative to $q^{(n)}$.

One may likewise define multi-valued functions $\npdnr{n}$ and $\npdnr{n}_k$ for each $n \geq 1$ and $k \geq 2$ (where the subscript $k$ denotes that the output function must be $k$-bounded).  Our analysis of the Rival--Sands theorem and its variations mostly makes use of $\pdnr$ and $\npdnr{3}_2$.

The basis of the next lemma is the simple fact that if $f$ is $\tdnr_2$ relative to $q$ and $p \leqT q$, then there is a $g \leqT f$ that is $\tdnr_2$ relative to $p$.  We give a proof emphasizing that this fact can be implemented by a strong Weihrauch reduction and therefore respects jumps.

\begin{Lemma}\label{lem-WKLequivDNR2}
For every $n \in \omega$, $\wkl^{(n)} \equivsW \pdnr_2^{(n)} \equivsW \npdnr{(n+1)}_2$.  In particular, $\wkl'' \equivsW \npdnr{3}_2$.
\end{Lemma}

\begin{proof}
We start by showing that $\npdnr{(n+2)}_2 \equivsW (\npdnr{(n+1)}_2)'$ for every $n \in \omega$.

$\npdnr{(n+2)}_2 \leqsW (\npdnr{(n+1)}_2)'$:  Let $p \colon \omega \imp \omega$ be an input to $\npdnr{(n+2)}_2$.  Let $\Phi$ be the functional given by
\begin{align*}
\Phi(p)(e,s) =
\begin{cases}
0 & \text{if $\Phi_{e,s}(p)(e)\ua$}\\
1 & \text{if $\Phi_{e,s}(p)(e)\da$},
\end{cases}
\end{align*}
so that $\Phi(p)$ is a $\mbf{\Delta}^0_2$-approximation to $p'$.  Let $f \in (\npdnr{(n+1)}_2)'(\Phi(p))$.  Let $\Psi$ be the identity functional.  The function $f$ is $\tdnr_2$ relative to the $n$\textsuperscript{th} Turing jump of $p'$.  That is, $f$ is $\tdnr_2$ relative to $(p')^{(n)} = p^{(n+1)}$.  Thus $f \in \npdnr{(n+2)}_2(p)$, so $\Phi$ and $\Psi$ witness that $\npdnr{(n+2)}_2 \leqsW (\npdnr{(n+1)}_2)'$.

$(\npdnr{(n+1)}_2)' \leqsW \npdnr{(n+2)}_2$:  Fix recursive functions $c \colon \omega \times \omega \imp \omega$ and $j, d \colon \omega \imp \omega$ with the following properties.
\begin{itemize}
\item For every $p, q \in \omega^\omega$, if $p = \Phi_e(q)$, then $\Phi_{c(e,i)}(q)$ computes the same partial function as $\Phi_i(p)$.

\smallskip

\item For every $p, q \in \omega^\omega$, if $p = \Phi_e(q)$, then $p' = \Phi_{j(e)}(q')$.

\smallskip

\item For every $p, q \in \omega^\omega$ and every $f \in 2^\omega$, if $p = \Phi_e(q)$ and $f \in \pdnr_2(q)$, then $\Phi_{d(e)}(f) \in \pdnr_2(p)$.
\end{itemize}
Compute $c(e,i)$ by computing an index $k$ for a machine $\Phi_k$ that simulates $\Phi_i$, but uses $\Phi_e$ to answer $\Phi_i$'s oracle queries.  Compute $j(e)$ by computing an index $k$ for a machine $\Phi_k$ that behaves as follows when equipped with any oracle $g$.  Given input $\ell$, $\Phi_k(g)(\ell)$ computes an index $i$ such that $\forall h \forall m (\Phi_\ell(h)(\ell)\da \biimp \Phi_i(h)(m)\da)$, then it outputs $g(c(e,i))$.  If $p = \Phi_e(q)$, then
\begin{align*}
\ell \in p' \biimp \Phi_\ell(p)(\ell)\da \biimp \Phi_i(p)(c(e,i))\da \biimp \Phi_{c(e,i)}(q)(c(e,i))\da \biimp c(e,i) \in q',
\end{align*}
so $\Phi_{j(e)}(q')(\ell) = q'(c(e,i)) = p'(\ell)$ as desired.  The function $d$ can be computed by a similar strategy.

Fix an index $e$ for a machine $\Phi_e$ such that for every $h \colon \omega \times \omega \imp \omega$, if $h$ is a $\mbf{\Delta}^0_2$-approximation to a function $p$, then $\Phi_e(h') = p$.  Let $h \colon \omega \times \omega \imp \omega$ be an input to $(\npdnr{(n+1)}_2)'$, and let $p = \lim h$.  Let $\Phi$ be the identity functional, and let $f \in \npdnr{(n+2)}_2(h)$.  Let $\Psi$ be the functional given by $\Psi(f) = \Phi_{d(j^n(e))}(f)$.  We have that $p = \Phi_e(h')$, so $n$ applications of $j$ to $e$ yields an index $j^n(e)$ such that $p^{(n)} = \Phi_{j^n(e)}(h^{(n+1)})$.  Then $\Phi_{d(j^n(e))}(f)$ is $\tdnr_2$ relative to $p^{(n)}$ because $f$ is $\tdnr_2$ relative to $h^{(n+1)}$.  Hence $\Phi_{d(j^n(e))}(f) \in (\npdnr{(n+1)}_2)'(h)$, so $\Phi$ and $\Psi$ witness that $(\npdnr{(n+1)}_2)' \leqsW \npdnr{(n+2)}_2$.

We now prove the lemma by induction on $n$.  For the base case, $\pdnr_2$ and $\npdnr{1}_2$ are the same problem, and $\wkl \equivsW \pdnr_2$ by~\cite{BrattkaHendtlassKreuzer}*{Corollary~5.3}.  That $\wkl \equivsW \pdnr_2$ may also be observed by noting that the uniformities in the classic arguments of Jockusch and Soare~\cite{JockuschSoare} imply that $\wkl \equivW \pdnr_2$ and then by checking that $\wkl$ and $\pdnr_2$ are both cylinders.  Suppose by induction that $\wkl^{(n)} \equivsW \pdnr_2^{(n)} \equivsW \npdnr{(n+1)}_2$.  Then $\wkl^{(n+1)} \equivsW \pdnr_2^{(n+1)}$ by Proposition~\ref{prop-WJumpPres}.  Furthermore:
\begin{align*}
\pdnr_2^{(n+1)} = (\pdnr_2^{(n)})' \equivsW (\npdnr{(n+1)}_2)' \equivsW \npdnr{(n+2)}_2,
\end{align*}
where the first equivalence uses the induction hypothesis and Proposition~\ref{prop-WJumpPres}.
\end{proof}

Ramsey's theorem for pairs and its consequences have received considerable attention in the Weihrauch degrees.  See, for example,~\cites{DoraisDzhafarovHirstMiletiShafer, HirschfeldtJockusch, BrattkaRakotoniaina}.  We define the multi-valued functions relevant to our analysis of the Rival--Sands theorem and its weak version.

Multi-valued functions corresponding to $\rt^2_2$, $\srt^2_2$, and $\coh$ are defined as follows.
\begin{Definition}{\ }
\begin{itemize}
\item $\rt^2_2$ is the following multi-valued function.
\begin{itemize}
\item Input/instance:  A coloring $c \colon [\omega]^2 \imp \{0, 1\}$.

\smallskip

\item Output/solution:  An infinite $H \subseteq \omega$ that is homogeneous for $c$.
\end{itemize}

\medskip

\item $\srt^2_2$ is the following multi-valued function.
\begin{itemize}
\item Input/instance:  A stable coloring $c \colon [\omega]^2 \imp \{0, 1\}$.

\smallskip

\item Output/solution:  An infinite $H \subseteq \omega$ that is homogeneous for $c$.
\end{itemize}

\medskip

\item $\coh$ is the following multi-valued function.
\begin{itemize}
\item Input/instance:  A sequence $\vec{A} = (A_i : i \in \omega)$ of subsets of $\omega$.

\smallskip

\item Output/solution:  A set $C \subseteq \omega$ that is cohesive for $\vec{A}$.
\end{itemize}
\end{itemize}
\end{Definition}

Over $\rca$, the ascending/descending sequence principle $\ads$ decomposes into the conjunction of a stable version $\sads$ and a cohesive version $\cads$, similar to how $\rt^2_2$ decomposes into the conjunction of $\srt^2_2$ and $\coh$~\cite{HirschfeldtShore}*{Proposition~2.7}.  There are choices to be made concerning what one means by \emph{stable} and \emph{sequence} that are inessential when working in $\rca$ or in the computable degrees, but make a difference when working in the Weihrauch degrees.  These choices give rise to a number of multi-valued functions related to $\ads$ that are explored in~\cite{AstorDzhafarovSolomonSuggs}.

Officially, a linear order is \emph{stable} if it satisfies the following definition.
\begin{Definition}
A linear order $(L, <_L)$ is \emph{stable} if every element either has only finitely many $<_L$-predecessors or has only finitely many $<_L$-successors.
\end{Definition}
Thus the countable stable linear orders are those of type $\omega + k$ for a finite linear order $k$, of type $k + \omega^*$ for a finite linear order $k$, and of type $\omega + \omega^*$.  Over $\rca$, $\sads$ is defined to be $\ads$ restricted to stable linear orders, which is equivalent to $\ads$ restricted to linear orders of type $\omega + \omega^*$.  In the Weihrauch degrees, restricting to linear orders of type $\omega + \omega^*$ yields a weaker principle than restricting to all stable linear orders.

When working with linear orders in the Weihrauch degrees, there is an important distinction between an \emph{ascending sequence} (as defined in Definition~\ref{def-ADS}) and an \emph{ascending chain}.  Essentially, the difference is that we can uniformly distinguish ascending sequences from descending sequences, but we cannot uniformly distinguish ascending chains from descending chains.

\begin{Definition}[See~\cite{AstorDzhafarovSolomonSuggs}*{Definition~2.1}]
Let $(L, <_L)$ be a linear order.
\begin{itemize}
\item A set $C \subseteq L$ is an \emph{ascending chain} in $L$ if $\{x \in C : x <_L y\}$ is finite for every $y \in C$.

\smallskip

\item A set $C \subseteq L$ is an \emph{descending chain} in $L$ if $\{x \in C : y <_L x\}$ is finite for every $y \in C$.
\end{itemize}
\end{Definition}

Of the versions of $\ads$ and $\sads$ studied in~\cite{AstorDzhafarovSolomonSuggs}, we mostly consider $\ads$, which is $\leqW$-greatest, and $\sadc$, which is $\leqW$-least.  In Section~\ref{sec-wRSgWeihrauch}, we show that the weak Rival--Sands theorem is above $\ads$ in the computable degrees, but not above $\sadc$ in the Weihrauch degrees.

\begin{Definition}{\ }
\begin{itemize}
\item $\ads$ is the following multi-valued function.
\begin{itemize}
\item Input/instance:  An infinite linear order $L = (L, <_L)$.

\smallskip

\item Output/solution:  An infinite $S \subseteq L$ that is either an ascending sequence in $L$ or a descending sequence in $L$.
\end{itemize}

\medskip

\item $\adc$ (for the \emph{ascending/descending chain principle}) is the following multi-valued function.
\begin{itemize}
\item Input/instance:  An infinite linear order $L = (L, <_L)$.

\smallskip

\item Output/solution:  An infinite $S \subseteq L$ that is either an ascending chain in $L$ or a descending chain in $L$.
\end{itemize}

\medskip

\item $\sadc$ (for the \emph{stable ascending/descending chain principle}) is the following multi-valued function.
\begin{itemize}
\item Input/instance:  An infinite linear order $L = (L, <_L)$ of order-type $\omega + \omega^*$.

\smallskip

\item Output/solution:  An infinite $C \subseteq L$ that is either an ascending chain in $L$ or a descending chain in $L$.
\end{itemize}
\end{itemize}
\end{Definition}
We warn the reader that the definitions of $\ads$, $\adc$, and $\sadc$ in~\cite{AstorDzhafarovSolomonSuggs} require that the domain of the input linear order $L$ is $\omega$, whereas here we find it convenient to allow the domain of $L$ to be any infinite subset of $\omega$.  The difference is exactly the difference between $\leqW$ and $\leqsW$.  Let $\ads\rst_{\dom(L) = \omega}$ denote the restriction of $\ads$ to linear orders with domain $\omega$.  Then $\ads\rst_{\dom(L) = \omega} \equivW \ads$ but $\ads\rst_{\dom(L) = \omega} \ltsW \ads$.  In fact, $\ads$ is the \emph{cylindrification} of $\ads\rst_{\dom(L) = \omega}$, meaning that $\ads \equivsW \id \times \ads\rst_{\dom(L) = \omega}$.  This can be seen by an argument analogous to the proof of Proposition~\ref{prop-wRSgCyl} below.  Similarly, $\adc$ and $\sadc$ are the cylindrifications of their restrictions $\adc\rst_{\dom(L) = \omega}$ and $\sadc\rst_{\dom(L) = \omega}$ to linear orders with domain $\omega$.

We use the following formulation of $\cads$ as a multi-valued function.
\begin{Definition}
$\cads$ (\emph{for the cohesive ascending/descending sequence principle}) is the following multi-valued function.
\begin{itemize}
\item Input/instance:  An infinite linear order $L = (L, <_L)$.

\smallskip

\item Output/solution:  An infinite $H \subseteq L$ such that $(H, <_L)$ is stable.
\end{itemize}
\end{Definition}

The definition of $\cads$ is also sensitive to the precise formulation of \emph{stable}.  For the purposes of this discussion, call a linear order \emph{strictly stable} if it is of order-type $\omega$, $\omega^*$, or $\omega + \omega^*$.  Let $\cadsst$ denote the multi-valued function defined in the same way as $\cads$, except with \emph{strictly stable} in place of \emph{stable}.  Then $\cads \ltW \cadsst$.  The difference in uniformity between $\cads$ and $\cadsst$ is due to the fact that any finite suborder of an infinite linear order can be extended to an infinite stable suborder but not necessarily to an infinite strictly stable suborder.  The following proposition requires no great originality, but we include a proof in order to complete the discussion.

\begin{Proposition}
$\cadsst \nleqW \cads$.  Therefore $\cads \ltW \cadsst$. 
\end{Proposition}

\begin{proof}
Suppose for a contradiction that functionals $\Phi$ and $\Psi$ witness that $\cadsst \leqW \cads$.  Let $L_0$ be the $\cadsst$-instance $L_0 = (\omega, <)$, let $R_0 = \Phi(L_0)$, and write $R_0 = (R_0, <_{R_0})$.  The linear order $R_0$ is a $\cads$-instance, so let $C_0 \subseteq R_0$ be a $\cads$-solution.  Then $\Psi(\la L_0, C_0 \ra)$ is a $\cadsst$-solution to $L_0$.  In particular, $\Psi(\la L_0, C_0 \ra)$ is infinite.  Fix any $\ell \in \Psi(\la L_0, C_0 \ra)$.  Let $\sigma \subseteq L_0$ and $\tau \subseteq C_0$ be initial segments of $L_0$ and $C_0$ long enough to guarantee that $\ell \in \Psi(\la \sigma, \tau \ra)$ and that if $n < |\tau|$ and $\tau(n) = 1$, then $n$ is in the domain of the partially-defined structure computed by $\Phi(\sigma)$.  The string $\sigma$ only encodes information about elements of $L_0$ that are $<$-less than $|\sigma|$.  Let $k$ be the $<$-maximum of $|\sigma|$ and $\ell+1$.  Then any linear order $L_1 = (\omega, <_{L_1})$ that defines $0 <_{L_1} 1 <_{L_1} \cdots <_{L_1} k-1$ is consistent with the information contained in $\sigma$ and therefore satisfies $\sigma \subseteq L_1$.  Thus let $L_1 = (\omega, <_{L_1})$ be the linear order of type $\omega + k$ in which $i <_{L_1} j$ if and only if $i < j < k$, $k \leq i < j$, or $j < k \leq i$.  Then $\sigma \subseteq L_1$.  Notice that $\ell$ is in the $k$-part of $L_1$ and therefore that no strictly stable suborder of $L_1$ contains $\ell$.

$L_1$ is a $\cadsst$-instance, so $R_1 = \Phi(L_1)$ is an infinite linear order.  Write $R_1 = (R_1, <_{R_1})$.  Notice that $(\forall n < |\tau|)(\tau(n) = 1 \imp n \in R_1)$ by the choices of $\sigma$ and $\tau$ and the fact that $\sigma \subseteq L_1$.  Let $S \subseteq R_1$ be an $\ads$-solution to $R_1$ with $n > |\tau|$ for every $n \in S$.  Let $C_1 = S \cup \{n < |\tau| : \tau(n) = 1\} \subseteq R_1$.  Then $\tau \subseteq C_1$.  Furthermore, $(C_1, <_{R_1})$ is stable because it is the union of the ascending or descending sequence $S$ and a finite set.  Thus $C_1$ is a $\cads$-solution to $R_1$.  However, $\ell \in \Psi(\la L_1, C_1 \ra)$ because $\ell \in \Psi(\la \sigma, \tau \ra)$, $\sigma \subseteq L_1$, and $\tau \subseteq C_1$.  Thus $\Psi(\la L_1, C_1 \ra)$ is not a $\cadsst$-solution to $L_1$ because no strictly stable suborder of $L_1$ contains $\ell$.  Therefore $\Phi$ and $\Psi$ do not witness that $\cadsst \leqW \cads$.  So $\cadsst \nleqW \cads$.
\end{proof}

The reason we consider $\cads$ is because $\cads \equivW \coh$.  Indeed, with the formulations of $\coh$ and $\cads$ given here, $\cads\rst_{\dom(L) = \omega} \equivsW \coh$, where $\cads\rst_{\dom(L) = \omega}$ is the restriction of $\cads$ to linear orders with domain $\omega$.  Thus $\cads$ is the cylindrification of both $\cads\rst_{\dom(L) = \omega}$ and $\coh$.  The proofs are those of~\cite{HirschfeldtShore}, which we repeat here for the sake of completeness.  The equivalence between $\cads$ and $\coh$ helps us show that $\coh$ Weihrauch reduces to the weak Rival--Sands theorem in Theorem~\ref{thm-COHredWRSG} below.

\begin{Proposition}[See~\cite{HirschfeldtShore}*{Propositions~2.9 and~4.4}]\label{prop-CADSCOH}
$\cads\rst_{\dom(L) = \omega} \equivsW \coh$.  Therefore $\cads \equivsW \id \times \cads\rst_{\dom(L) = \omega} \equivsW \id \times \coh$.
\end{Proposition}

\begin{proof}
We have that $\cads \equivsW \id \times \cads\rst_{\dom(L) = \omega}$ by an argument analogous to the proof of Proposition~\ref{prop-wRSgCyl} below.  So it suffices to show that $\cads\rst_{\dom(L) = \omega} \equivsW \coh$.

For $\cads\rst_{\dom(L) = \omega} \leqsW \coh$, given a linear order $L = (\omega, <_L)$, apply $\coh$ to the sequence $\vec{A} = (A_i : i \in \omega)$ where $A_i = \{n \in \omega : i <_L n\}$.  Then any $\vec{A}$-cohesive set $C$ is also a $\cads$-solution to $L$.

The proof of~\cite{HirschfeldtShore}*{Proposition~4.4} showing that $\rca + \bst + \cads \vdash \coh$ describes a strong Weihrauch reduction $\coh \leqsW \cads\rst_{\dom(L) = \omega}$.  Let $\vec{A} = (A_i : i \in \omega)$ be a $\coh$-instance.  Define a functional $\Phi(\vec{A})$ computing a linear order $L = (\omega, <_L)$ as follows.  Given $x$ and $y$, define $x <_L y$ if and only if $\la A_i(x) : i \leq x \ra \ltlex \la A_i(y) : i \leq y \ra$, where $\ltlex$ denotes the lexicographic order on $2^{<\omega}$.  Let $C$ be a $\cads$-solution to $L$, and let $\Psi$ be the identity functional.  We claim that $C$ is $\vec{A}$-cohesive and hence that $\Phi$ and $\Psi$ witness that $\coh \leqsW \cads\rst_{\dom(L) = \omega}$.

To see that $C$ is $\vec{A}$-cohesive, fix $n$ and let $F_n = \{\sigma \in 2^{n+1} : (\exists x \in C)(\sigma \subseteq \la A_i(x) : i \leq x \ra)\}$.  Let $\sigma_0 \ltlex \cdots \ltlex \sigma_{k-1}$ list the elements of $F_n$ in $\ltlex$-increasing order.  For each $j < k$, let $x_{\sigma_j}$ be the $<$-least element of $C$ witnessing that $\sigma_j \in F_n$.  Then $x_{\sigma_0} <_L \cdots <_L x_{\sigma_{k-1}}$.  The order $(C, <_L)$ is stable, so in $C$ exactly one interval $(-\infty, x_{\sigma_0}), (x_{\sigma_0}, x_{\sigma_1}), \dots, (x_{\sigma_{k-2}}, x_{\sigma_{k-1}}), (x_{\sigma_{k-1}}, \infty)$ is infinite, where $(-\infty, a)$ and $(a, \infty)$ denote $\{x \in C : x <_L a\}$ and $\{x \in C : a <_L x\}$.  If $(x_{\sigma_j}, x_{\sigma_{j+1}})$ is infinite for some $j < k-1$, then almost every $y \in C$ satisfies $\sigma_j \subseteq \la A_i(y) : i \leq y \ra$.  In particular, $A_n(y) = \sigma_j(n)$ for almost every $y \in C$, so either $C \subseteq^* A_n$ or $C \subseteq^* \ol{A_n}$.  Similarly, if $(-\infty, x_{\sigma_0})$ is infinite then $A_n(y) = \sigma_0(n)$ for almost every $y \in C$; and if $(x_{\sigma_{k-1}}, \infty)$ is infinite, then $A_n(y) = \sigma_{k-1}(n)$ for almost every $y \in C$.  Thus $C$ is $\vec{A}$-cohesive.  
\end{proof}

Finally, we consider multi-valued functions corresponding to the various versions of the infinite pigeonhole principle presented in Definition~\ref{def-RT1SOA}.  In Section~\ref{sec-ADSvsRT1Weihrauch}, we narrow an apparent gap in the literature concerning the Weihrauch degrees of the ascending/descending sequence principle and the infinite pigeonhole principle.  In Section~\ref{sec-wRSgWeihrauch}, we show that the weak Rival--Sands theorem is above the infinite pigeonhole principle in the Weihrauch degrees.

\begin{Definition}\label{def-RT1WD}{\ }
\begin{itemize}
\item For each $k > 0$, $\rt^1_k$ is the following multi-valued function.
\begin{itemize}
\item Input/instance:  A coloring $c \colon \omega \imp \{0, 1, \dots, k-1\}$.

\smallskip

\item Output/solution:  An infinite $H \subseteq \omega$ that is monochromatic for $c$:  $(\forall x, y \in H)(c(x) = c(y))$.
\end{itemize}

\medskip

\item $\rt^1_{<\infty}$ is the following multi-valued function.
\begin{itemize}
\item Input/instance:  A coloring $c \colon \omega \imp \omega$ with finite range.

\smallskip

\item Output/solution:  An infinite $H \subseteq \omega$ that is monochromatic for $c$:  $(\forall x, y \in H)(c(x) = c(y))$.
\end{itemize}
\end{itemize}
\end{Definition}

If $0 < k \leq \ell$, then $\rt^1_k \leqsW \rt^1_\ell \leqsW \rt^1_{<\infty}$.  Furthermore, it is Weihrauch equivalent to specify that solutions to $\rt^1_k$- and $\rt^1_{<\infty}$-instances are the colors of the monochromatic sets instead of the monochromatic sets themselves.  For $k > 0$, let $\crt^1_k$ denote the problem whose instances are colorings $c \colon \omega \imp k$ and whose solutions are the numbers $i < k$ such that $c^{-1}(i)$ is infinite.  Define $\crt^1_{<\infty}$ similarly.    Then $\crt^1_k \equivW \rt^1_k$ for each $k > 0$, and $\crt^1_{<\infty} \equivW \rt^1_{<\infty}$.  However, $\crt^1_k \nequivsW \rt^1_k$ for $k > 2$, and $\crt^1_{<\infty} \nequivsW \rt^1_{<\infty}$.

\section{Ascending/descending sequence principles versus pigeonhole principles in the Weihrauch degrees}\label{sec-ADSvsRT1Weihrauch}

In reverse mathematics, $\rca + \ads \vdash \rt^1_{<\infty}$ by Proposition~\ref{prop-rt22andADS} and the equivalence between $\rt^1_{<\infty}$ and $\bst$.  Indeed, the proof that $\rca + \ads \vdash \bst$ in~\cite{HirschfeldtShore} establishes that $\rca + \ads \vdash \rt^1_{<\infty}$.  In the Weihrauch degrees, the analogous question of whether $\rt^1_{<\infty} \leqW \ads$ has not yet, to our knowledge, been considered in the literature.  We address this gap by showing that $\rt^1_5 \nleqW \ads$ and therefore that $\rt^1_{<\infty} \nleqW \ads$.  On the positive side, we show that $\rt^1_3 \leqsW \adc$.  Whether $\rt^1_4 \leqW \ads$ remains open.

\begin{Theorem}\label{thm-RT13vsADC}
$\rt^1_3 \leqsW \adc$.
\end{Theorem}

\begin{proof}
$\rt^1_3 \equivW \crt^1_3$ and $\adc$ is a cylinder, so it suffices to show that $\crt^1_3 \leqW \adc$.  Let $c$ be a $\crt^1_3$-instance.  Define a functional $\Phi$, where $\Phi(c)$ computes a linear order $L = (\omega, <_L)$ as follows.  The computation of $L$ proceeds in stages, where at stage $s$ the order $<_L$ is determined on $\{0, 1, \dots, s\}$.  Throughout the computation, we maintain three sets $A_s, M_s, D_s \subseteq \{0, 1, \dots, s\}$, with $\max_{<_L}(A_s) <_L \min_{<_L}(M_s)$ and $\max_{<_L}(M_s) <_L \min_{<_L}(D_s)$, where $\min_{<_L}(X)$ and $\max_{<_L}(X)$ denote the minimum and maximum elements of the finite set $X$ with respect to $<_L$.  The sets $A_s$ and $D_s$ are used to build an ascending sequence and a descending sequence in $L$ in order to achieve the following.
\begin{itemize}
\item If only two colors $i < j < 3$ occur in the range of $c$ infinitely often, then $L$ has order-type $\omega + k + \omega^*$ for some finite linear order $k$, with the $\omega$-part of $L$ corresponding to color $i$ and the $\omega^*$-part of $L$ corresponding to color $j$.

\medskip

\item If only one color $i < 3$ occurs in the range of $c$ infinitely often, then $L$ has either order-type $\omega + k$ or order-type $k + \omega^*$ for some finite linear order $k$, with the $\omega$-part or the $\omega^*$-part of $L$ corresponding to color $i$.
\end{itemize}

To monitor the last two colors seen up to $s$ (or the only color seen so far, if $c$ is constant up to $s$), let $t < s$ be greatest such that $c(t) \neq c(s)$, let $\last_s = \{c(t), c(s)\}$ if there is such a $t$, and otherwise let $\last_s = \{c(s)\}$.  We assign the least color of $\last_s$ to $A_s$ and the other color (if it exists) to $D_s$.

At stage $0$, let $A_0 = \{0\}$, $M_0 = \emptyset$, and $D_0 = \emptyset$.  Assign $A_0$ color $c(0)$ and assign $D_0$ no color.  At stage $s+1$, first check if $\last_{s+1} = \last_s$.  If $\last_{s+1} = \last_s$, then color $c(s+1)$ is assigned to either $A_s$ or $D_s$.  If $c(s+1)$ is assigned to $A_s$, then set $A_{s+1} = A_s \cup \{s+1\}$, $M_{s+1} = M_s$, and $D_{s+1} = D_s$.  Extend $<_L$ so that $s+1$ is the $<_L$-maximum element of $A_{s+1}$ and $<_L$-below all elements of $M_{s+1}$ and $D_{s+1}$.  If $c(s+1)$ is assigned to $D_s$, then set $A_{s+1} = A_s$, $M_{s+1} = M_s$, and $D_{s+1} = D_s \cup \{s+1\}$.  Extend $<_L$ so that $s+1$ is the $<_L$-minimum element of $D_{s+1}$ and $<_L$-above all elements of $A_{s+1}$ and $M_{s+1}$.  Assign $A_{s+1}$ the same color as $A_s$, and assign $D_{s+1}$ the same color as $D_s$.  If $\last_{s+1} \neq \last_s$, then set $M_{s+1} = \{0, 1, \dots, s\}$.  If $c(s+1)$ is the least color of $\last_{s+1}$, then set $A_{s+1} = \{s+1\}$, set $D_{s+1} = \emptyset$, extend $<_L$ so that $s+1$ is the $<_L$-minimum element of $\{0, 1, \dots, s+1\}$, assign $A_{s+1}$ color $c(s+1)$, and assign $D_{s+1}$ the other color of $\last_{s+1}$.  If $c(s+1)$ is not the least color of $\last_{s+1}$, then set $A_{s+1} = \emptyset$, set $D_{s+1} = \{s+1\}$, extend $<_L$ so that $s+1$ is the $<_L$-maximum element of $\{0, 1, \dots, s+1\}$, assign $D_{s+1}$ color $c(s+1)$, and assign $A_{s+1}$ the other color of $\last_{s+1}$.  This completes the computation of $L$.

The linear order $L$ is a valid $\adc$-instance, so let $S$ be an $\adc$-solution to $L$.  Define a functional $\Psi(\la c, S \ra)$ by finding the $<$-least element $x_0$ of $S$ and outputting $\Psi(\la c, S \ra) = c(x_0)$.  We show that $c(x_0)$ appears in the range of $c$ infinitely often and therefore that $\Psi(\la c, S \ra)$ is a $\crt^1_3$-solution to $c$.  Thus $\Phi$ and $\Psi$ witness that $\crt^1_3 \leqW \adc$.

If every color $i < 3$ appears in the range of $c$ infinitely often, then $c(x_0)$ appears in the range of $c$ infinitely often.  Suppose that exactly two colors $i < j < 3$ appear in the range of $c$ infinitely often.  Then there is an $s_0$ such that $\last_s = \last_{s_0} = \{i,j\}$ for all $s \geq s_0$.  In this case, each $s \geq s_0$ with $c(s) = i$ is added to $A_s$, and each $s \geq s_0$ with $c(s) = j$ is added to $D_s$.  Thus $L$ is a linear order of type $\omega + k + \omega^*$ with $\omega$-part $A = \bigcup_{s \geq s_0}A_s$, $\omega^*$-part $D = \bigcup_{s \geq s_0}D_s$, and $k$-part $M_{s_0}$.  If $S$ is an ascending chain, then it must be that $S \subseteq A$.  We have that $c(x) = i$ for all $x \in A$.  In particular, $c(x_0) = i$, which occurs in the range of $c$ infinitely often.  If $S$ is a descending chain, then it must be that $S \subseteq D$.  We have that $c(x) = j$ for all $x \in D$.  Thus $c(x_0) = j$, which occurs in the range of $c$ infinitely often.

Finally, suppose that exactly one color $i < 3$ appears in the range of $c$ infinitely often.  Then there is an $s_0$ such that $c(s) = i$ for all $s \geq s_0$ and hence is also such that $\last_s = \last_{s_0}$ for all $s \geq s_0$.  If $i$ is the least color of $\last_{s_0}$, then $s$ is added to $A_s$ for all $s \geq s_0$, and $L$ is a linear order of type $\omega + k$ with $\omega$-part $A = \bigcup_{s \geq s_0} A_s$ and $k$-part $M_{s_0} \cup D_{s_0}$.  It must therefore be that $S \subseteq A$.  We have that $c(x) = i$ for all $x \in A$.  Thus $c(x_0) = i$, which occurs in the range of $c$ infinitely often.  If instead $i$ is not the least color of $\last_{s_0}$, then $s$ is added to $D_s$ for all $s \geq s_0$, and $L$ is a linear order of type $k + \omega^*$ with $\omega^*$-part $D = \bigcup_{s \geq s_0} D_s$ and $k$-part $A_{s_0} \cup M_{s_0}$.  It must therefore be that $S \subseteq D$.  We have that $c(x) = i$ for all $x \in D$.  Thus $c(x_0) = i$, which occurs in the range of $c$ infinitely often.
\end{proof}

\begin{Theorem}\label{thm-RT15vsADS}
$\rt^1_5 \nleqW \ads$.  Therefore $\rt^1_{<\infty} \nleqW \ads$.
\end{Theorem}

\begin{proof}
$\rt^1_5 \equivW \crt^1_5$ and $\ads$ is a cylinder, so it suffices to show that $\crt^1_5 \nleqsW \ads$.  Suppose for a contradiction that $\Phi$ and $\Psi$ witness that $\crt^1_5 \leqsW \ads$.  We compute a coloring $c \colon \omega \imp 5$ such that the $\ads$-instance $\Phi(c)$ has a solution $S$ for which $c^{-1}(\Psi(S))$ is finite, contradicting that $\Phi$ and $\Psi$ witness that $\crt^1_5 \leqsW \ads$.

The computation of $c$ proceeds in stages, where at stage $s+1$ we determine the value of $c(s)$.  Thus we compute a sequence of strings $(c_s : s \in \omega)$, where $c_s \in 5^s$ and $c_s \subseteq c_{s+1}$ for each $s$.  The final coloring $c$ is $c = \bigcup_{s \in \omega}c_s$.

For each $s$, let $L_s = \Phi(c_s) \rst s$ denote the partially-defined structure obtained by running $\Phi(c_s)(n)$ for $s$ steps for each $n < s$.  Write also $L_s = (L_s, <_{L_s})$.  $L_s$ is not necessarily a linear order, but it must be consistent with being a linear order because there are functions $c \colon \omega \imp 5$ extending $c_s$.

For $\sigma \in 2^{<\omega}$, let $\set(\sigma) = \{n < |\sigma| : \sigma(n) = 1\}$ denote the finite set for which $\sigma$ is a characteristic string.  

The computation of $c$ begins in phase~I, and it may or may not eventually progress to phase~II.  The goal of phase~I is to identify $s, m_*  \in \omega$, $k_\asc, k_\dec < 5$, and $\sigma_*, \tau_* \in 2^{<\omega}$ such that
\begin{itemize}
\item $\set(\sigma_*)$ is an ascending sequence in $L_s$ (in the sense of Definition~\ref{def-ADS}) with $k_\asc = \Psi(\sigma_*)\da$;

\smallskip

\item $\set(\tau_*)$ is a descending sequence in $L_s$ with $k_\dec = \Psi(\tau_*)\da$;

\smallskip

\item $m_*$ is both the $<_{L_s}$-maximum element of $\set(\sigma_*)$ and the $<_{L_s}$-minimum element of $\set(\tau_*)$.
\end{itemize}

Once $s$, $m_*$, $k_\asc$, $k_\dec$, $\sigma_*$, and $\tau_*$ are found, the computation enters phase~II and no longer uses colors $k_\asc$ and $k_\dec$.  The point is that, at the end of the construction, if $L = \Phi(c)$ has an infinite ascending sequence above $m_*$, then it has an infinite ascending sequence $S$ with $\sigma_* \subseteq S$ and hence with $\Psi(S) = k_\asc$.  Similarly, if $L$ has an infinite descending sequence below $m_*$, then it has an infinite descending sequence $S$ with $\tau_* \subseteq S$ and hence with $\Psi(S) = k_\dec$.  In both cases, $S$ is as desired because $c^{-1}(k_\asc)$ and $c^{-1}(k_\dec)$ are finite.  

Computation in phase~I proceeds as follows.  We maintain sequences $\vec\sigma = (\la \sigma_\ell, u_\ell, i_\ell \ra : \ell < a)$ and $\vec\tau = (\la \tau_\ell, d_\ell, j_\ell \ra : \ell < b)$ satisfying the following properties at each stage $s$.
\begin{enumerate}
\item\label{it-seq1} For each $\ell < a$, $\set(\sigma_\ell)$ is an ascending sequence in $L_s$, $u_\ell$ is the $<_{L_s}$-maximum element of $\set(\sigma_\ell)$, and $\Psi(\sigma_\ell) = i_\ell$.

\smallskip

\item\label{it-seq2} For each $\ell < b$, $\set(\tau_\ell)$ is an descending sequence in $L_s$, $d_\ell$ is the $<_{L_s}$-minimum element of $\set(\tau_\ell)$, and $\Psi(\tau_\ell) = j_\ell$.

\smallskip

\item\label{it-seq3} For each $\ell_0 < \ell_1 < a$, $u_{\ell_0} >_{L_s} u_{\ell_1}$.

\smallskip

\item\label{it-seq4} For each $\ell_0 < \ell_1 < b$, $d_{\ell_0} <_{L_s} d_{\ell_1}$.
\end{enumerate}

At stage $0$, begin with $c_0 = \emptyset$, $\vec\sigma = \emptyset$, and $\vec\tau = \emptyset$.  At stage $s+1$, let $c_{s+1}(s)$ be the least $i < 5$ that is neither $i_{a-1}$ (if $a > 0$) nor $j_{b-1}$ (if $b > 0$).  Next, search for an $\eta \in 2^{<s}$ such that $\Psi(\eta)\da$ and either
\begin{enumerate}[\indent(i)]
\item\label{it-FindAsc1} $\set(\eta)$ is an ascending sequence in $L_s$ with $<_{L_s}$-maximum element $u$, and $u <_{L_s} u_{a-1}$ if $a > 0$; or

\medskip

\item\label{it-FindDec1} $\set(\eta)$ is an descending sequence in $L_s$ with $<_{L_s}$-minimum element $d$, and $d >_{L_s} d_{b-1}$ if $b > 0$.
\end{enumerate}
If there is such an $\eta$, let $\eta$ be the first one found.  If $\eta$ satisfies~(\ref{it-FindAsc1}), let $\la \sigma_a, u_a, i_a \ra = \la \eta, u, \Psi(\eta)\ra$ and append this element to $\vec\sigma$.  If $\eta$ satisfies~(\ref{it-FindDec1}), let $\la \tau_b, d_b, j_b \ra = \la \eta, d, \Psi(\eta)\ra$ and append this element to $\vec\tau$.  If there is no such $\eta$, then do not update $\vec\sigma$ or $\vec\tau$.

Next, search for an $\theta \in 2^{<s}$ such that $\Psi(\theta)\da$ and either
\begin{enumerate}[\indent\indent(a)]
\item\label{it-FindDec2} $\set(\theta) \subseteq \{u_0, \dots, u_{a-1}\}$ is a descending sequence in $L_s$ or

\smallskip

\item\label{it-FindAsc2} $\set(\theta) \subseteq \{d_0, \dots, d_{b-1}\}$ is an ascending sequence in $L_s$.
\end{enumerate}
If there is such a $\theta$, let $\theta$ be the first one found.  If $\theta$ satisfies~(\ref{it-FindDec2}), let $u_\ell$ be the $<_{L_s}$-minimum element of $\set(\theta)$, which is also the $<_{L_s}$-maximum element of $\sigma_\ell$.  Set $\sigma_* = \sigma_\ell$, $\tau_* = \theta$, $m_* = u_\ell$, $k_\asc = i_\ell$, and $k_\dec = \Psi(\theta)$.  If $\theta$ satisfies~(\ref{it-FindAsc2}), let $d_\ell$ be the $<_{L_s}$-maximum element of $\set(\theta)$, which is also the $<_{L_s}$-minimum element of $\tau_\ell$.  Set $\sigma_* = \theta$, $\tau_* = \tau_\ell$, $m_* = d_\ell$, $k_\asc = \Psi(\theta)$, and $k_\dec = j_\ell$.  Go to stage $s+2$ and begin phase~II.  If there is no such $\theta$, go to stage $s+2$ and remain in phase~I.

The phase~II strategy is to reset $\vec\sigma$ and $\vec\tau$ to the $\sigma_*$, $\tau_*$, $m_*$, $k_\asc$ and $k_\dec$ found at the end of phase~I and then rerun a portion of the phase~I strategy.  Upon beginning phase~II, reset $\vec\sigma$ and $\vec\tau$ to $\vec\sigma = \la \sigma_0, u_0, i_0 \ra = \la \sigma_*, m_*, k_\asc \ra$ and $\vec\tau = \la \tau_0, d_0, j_0 \ra = \la \tau_*, m_*, k_\dec \ra$.  Throughout phase~II, $\vec\sigma$ and $\vec\tau$ satisfy the same items~(\ref{it-seq1})--(\ref{it-seq4}) from phase~I.  Computation in phase~II proceeds as follows.  At stage $s+1$, let $c_{s+1}(s)$ be the least $i < 5$ not in $\{k_\asc, k_\dec, i_{a-1}, j_{b-1}\}$.  Next, as in phase~I, search for an $\eta \in 2^{<s}$ with $\Psi(\eta)\da$ that satisfies either~(\ref{it-FindAsc1}) or~(\ref{it-FindDec1}).  If such an $\eta$ is found, then update either $\vec\sigma$ or $\vec\tau$ as in phase~I and go to stage $s+2$.  If no such $\eta$ is found, go to stage $s+2$ without updating $\vec\sigma$ or $\vec\tau$.  This completes the computation.

Let $L = \Phi(c)$ and write $L = (L, <_L)$.  We find an $\ads$-solution $S$ to $L$ such that $c^{-1}(\Psi(S))$ is finite, contradicting that $\Phi$ and $\Psi$ witness that $\crt^1_5 \leqsW \ads$.

First, suppose that the computation of $c$ never leaves phase~I.  Then there must be a stage after which no further elements are appended to either $\vec\sigma$ or $\vec\tau$.  This is because if, say, elements are appended to $\vec\sigma$ infinitely often, then $u_0 >_L u_1 >_L u_2 >_L \cdots$, which means that there is an infinite descending sequence $D \subseteq \{u_\ell : \ell \in \omega\}$.  This $D$ is an $\ads$-solution to $L$, so $\Psi(D)\da$.  Let $\theta \subseteq D$ be long enough so that $\Psi(\theta)\da$.  This $\theta$ eventually satisfies item~(\ref{it-FindDec2}) of phase~I, and the construction eventually finds $\theta$.  Thus the computation of $c$ eventually enters phase~II, contradicting the assumption that it never leaves phase~I.  So let $s_0$ be a stage after which no further elements are appended to $\vec\sigma$ or $\vec\tau$.  Then $a$, $b$, $i_{a-1}$ (if $a > 0$), and $j_{b-1}$ (if $b > 0$) do not change after stage $s_0$, and for every $s > s_0$, $c(s)$ is the least $i < 5$ that is neither $i_{a-1}$ (if $a > 0$) nor $j_{b-1}$ (if $b > 0$).  Let $A$ be an $\ads$-solution to $L$, and assume that $A$ is ascending (the descending case is symmetric).  If $a = 0$ or if $x <_L u_{a-1}$ for all $x \in A$, then let $\eta \subseteq A$ be long enough so that $\Psi(\eta)\da$.  This $\eta$ eventually satisfies item~(\ref{it-FindAsc1}) of phase~I, so the computation adds an element to $\vec\sigma$ at some stage after $s_0$, which is a contradiction.  Therefore it must be that $a > 0$ and that $x \geq_L u_{a-1}$ for some $x \in A$.  As $A$ is ascending, this means that $x >_L u_{a-1}$ for almost every $x \in A$.  Let $S = \set(\sigma_{a-1}) \cup \{x \in A : (x > u_{a-1}) \andd (x >_L u_{a-1})\}$.  Then $S$ is an infinite ascending sequence in $L$.  However, $\sigma_{a-1} \subseteq S$, so $\Psi(S) = i_{a-1}$.  We have that $c(s) \neq i_{a-1}$ for all $s > s_0$, so $S$ is as desired.

Now, suppose that the computation of $c$ eventually enters phase~II at some stage $s_0$.  Then $c(s)$ is neither $k_\asc$ nor $k_\dec$ for all $s > s_0$.  Recall that $\vec\sigma$ and $\vec\tau$ are reset at the beginning of phase~II.  Suppose that elements are appended to $\vec\sigma$ infinitely often in phase~II.  Then $m_* = u_0 >_L u_1 >_L u_2 >_L \cdots$, so there is an infinite descending sequence $D \subseteq \{u_\ell : \ell \in \omega\}$.  Recall that $\set(\tau_*)$ is a descending sequence with $\leq_L$-minimum element $m_*$ and $\Psi(\tau_*) = k_\dec$.  Let $S = \set(\tau_*) \cup \{x \in D : (x > m_*) \andd (x <_L m_*)\}$.  Then $S$ is an infinite descending sequence with $\tau_* \subseteq S$.  Therefore $\Psi(S) = k_\dec$.  However, $c(s) \neq k_\dec$ for all $s > s_0$, so $S$ is as desired.  If instead elements are appended to $\vec\tau$ infinitely often in phase~II, then a symmetric argument shows that there is an infinite ascending sequence $S$ with $\sigma_* \subseteq S$ and therefore with $\Psi(S) = k_\asc$.

Finally, suppose that there is a stage $s_1 > s_0$ after which no further elements are appended to either $\vec\sigma$ or $\vec\tau$.  We argue as in the case in which the computation of $c$ never leaves phase~I.  Notice that $a$, $b$, $i_{a-1}$, and $j_{b-1}$ do not change after stage $s_1$, and for every $s > s_1$, $c(s)$ is the least $i < 5$ that is not in $\{k_\asc, k_\dec, i_{a-1}, j_{b-1}\}$.  Let $A$ be an $\ads$-solution to $L$, and assume that $A$ is ascending (the descending case is symmetric).  If $x <_L u_{a-1}$ for all $x \in A$, then the computation must append an element to $\vec\sigma$ at some stage after $s_1$, which is a contradiction.  Otherwise, $x >_L u_{a-1}$ for almost every $x \in A$.  Let $S = \set(\sigma_{a-1}) \cup \{x \in A : (x > u_{a-1}) \andd (x >_L u_{a-1})\}$.  Then $S$ is an infinite ascending sequence in $L$ with $\Psi(S) = i_{a-1}$, but $c(s) \neq i_{a-1}$ for all $s > s_1$.  Thus $S$ is as desired.
\end{proof}

\begin{Question}
Does $\rt^1_4 \leqW \ads$ hold?
\end{Question}

\section{The Rival--Sands theorem in the Weihrauch degrees}\label{sec-WRS}

In this section, we show that the Rival--Sands theorem is Weihrauch equivalent to $\wkl''$ by showing that it is Weihrauch equivalent to $\npdnr{3}_2$.  To our knowledge, the Rival--Sands theorem gives the first example of an ordinary mathematical theorem that is Weihrauch equivalent to $\wkl''$.

Multi-valued functions corresponding to the Rival--Sands theorem and its refined version are defined as follows.
\begin{Definition}{\ }
\begin{itemize}
\item $\rsg$ is the following multi-valued function.
\begin{itemize}
\item Input/instance:  An infinite graph $G = (V,E)$.

\smallskip

\item Output/solution:  An infinite $H \subseteq V$ such that, for all $v \in V$, either $|H \cap N(v)| = \omega$ or $|H \cap N(v)| \leq 1$.
\end{itemize}

\medskip

\item $\rsgr$ is the following multi-valued function.
\begin{itemize}
\item Input:  An infinite graph $G = (V,E)$.

\smallskip

\item Output:  An infinite $H \subseteq V$ such that, for all $v \in V$, either $|H \cap N(v)| = \omega$ or $|H \cap N(v)| \leq 1$; and additionally, for all $v \in H$, either $|H \cap N(v)| = \omega$ or $|H \cap N(v)| = 0$.
\end{itemize}
\end{itemize}
\end{Definition}

Clearly $\rsg \leqsW \rsgr$ because for every infinite graph $G$, every $\rsgr$-solution to $G$ is also an $\rsg$-solution to $G$.

The proof that $\rsgr \leqsW \npdnr{3}_2$ is a refinement of the proof of Theorem~\ref{thm-RSgInACA}, which we outline before giving the full details.  Let $G = (V, E)$ be an infinite graph, and let $F = \{x \in V : \text{$N(x)$ is finite}\}$.  Let $f$ be $\tdnr_2$ relative to $G''$.  We need to uniformly compute an $\rsgr$-solution $H$ to $G$ from $f$.  To do this, we begin computing $H$ under the assumption that $F$ is infinite.  The function $f$ uniformly computes $G''$, $G'$, and $G$, so we may use these sets in the computation of $H$.  If $F$ really is infinite, then $G'$ computes an infinite $F_0 \subseteq F$.  Using $f$, we compute an infinite set $C \subseteq F_0$ that is cohesive for the sequence $(N(x) : x \in V)$ together with the function $p \colon V \imp \{0,1\}$ given by
\begin{align*}
p(x) = 
\begin{cases}
0 & \text{if $C \subseteq^* \ol{N(x)}$}\\
1 & \text{if $C \subseteq^* N(x)$}
\end{cases}
\end{align*}
(in fact, we need only the function $p$) using a strategy for producing cohesive sets similar to those of~\cite{JockuschStephan}.  The computation of $H$ then proceeds in the same manner as in the proof of Theorem~\ref{thm-RSgInACA}.  If $F$ is actually finite, then eventually $G''$ will detect this.  At the point when $G''$ detects that $F$ is finite, we will have committed to a finite set $\{x_0, \dots, x_{n-1}\} \subseteq F$ being in $H$.  Using $G''$, we may compute a bound $s$ such that $H = \{x_0, \dots, x_{n-1}\} \cup \{y \in V : y > s\}$ is an $\rsgr$-solution to $G$.

We introduce a bit of notation to help describe the proof.

\begin{Definition}
For $\vec{A} = (A_i : i \in \omega)$ a sequence of subsets of $\omega$ and $\sigma \in 2^{<\omega}$, let $A^\sigma \subseteq \omega$ denote the set
\begin{align*}
A^\sigma = \bigcap_{\substack{i < |\sigma|\\\sigma(i) = 0}}\ol{A_i} \cap \bigcap_{\substack{i < |\sigma|\\\sigma(i) = 1}}A_i,
\end{align*}
where also $A^\emptyset = \omega$.
\end{Definition}

\begin{Lemma}\label{lem-RSGRred3DNR2}
$\rsgr \leqsW \npdnr{3}_2$.
\end{Lemma}

\begin{proof}
The proof takes advantage of several well-known uniformities concerning enumerations, Turing jumps, and $\tdnr_2$ functions that we list and briefly sketch for the reader's reference.
\begin{enumerate}
\item\label{it-ure} There is a recursive function $e \mapsto i$ such that, for every $g \in \omega^\omega$, if $\ran(\Phi_e(g))$ is infinite, then $\Phi_i(g)$ computes the characteristic function of an infinite subset of $\ran(\Phi_e(g))$.  If $\ran(\Phi_e(g))$ is finite, then $\Phi_i(g)$ is partial but still satisfies $\forall n [\Phi_i(g)(n) = 1 \imp n \in \ran(\Phi_e(g))]$.

\medskip

\item\label{it-ujump} There is an index $e$ such that $\Phi_e(g') = g$ for every $g \in \omega^\omega$.

\medskip

\item\label{it-udnrID} There is an index $e$ such that $\Phi_e(f) = g$ whenever $f,g \in \omega^\omega$ are such that $f$ is $\tdnr_2$ relative to $g$.

\medskip

\item\label{it-uDNRWKL} There is a recursive function $e \mapsto i$ such that whenever $f, g \in \omega^\omega$ are such that $f$ is $\tdnr_2$ relative to $g$, we have that $\Phi_i(f) \colon \omega \imp \{0, 1\}$ is total, and also if $\Phi_e(g)$ computes an infinite subtree of $2^{<\omega}$, then $\Phi_i(f)$ computes an infinite path through $\Phi_e(g)$.
\end{enumerate}
For~(\ref{it-ure}), observe that typical proofs that every infinite r.e.\ set contains an infinite recursive subset (see~\cite{SoareBook}*{Section~II.1}) relativize and are uniform.  For~(\ref{it-ujump}), observe that typical proofs that $g'$ computes $g$ are uniform (see~\cite{SoareBook}*{Section~III.2}).  For~(\ref{it-udnrID}), $\Phi_e(f)(n)$ generates indices $i_0, i_1, \dots$ such that for every $g \in \omega^\omega$ and every $y \in \omega$, $\Phi_{i_y}(g)(i_y) = 0$ if $g(n) = y$ and $\Phi_{i_y}(g)(i_y) = 1$ if $g(n) \neq y$.  Then $\Phi_e(f)(n)$ searches for a $y$ with $f(i_y) = 1$ and outputs the first $y$ found.  If $f$ is $\tdnr_2$ relative to $g$, then $\Phi_e(f)(n) = g(n)$.  For~(\ref{it-uDNRWKL}), observe that typical proofs that $\tdnr_2$ functions compute paths through infinite recursive subtrees of $2^{<\omega}$ (see~\cite{NiesBook}*{Theorem~4.3.2}) relativize and are uniform.  

Let $G = (V,E)$ be an $\rsgr$-instance.  Let $\Phi$ be the identity functional, and let $f \in \npdnr{3}_2(G)$.  That is, let $f$ be $\tdnr_2$ relative to $G''$.  Define a functional $\Psi(f)$ computing a set $H$ as follows.  By items~(\ref{it-ujump}) and~(\ref{it-udnrID}) above, $f$ uniformly computes $G''$, which uniformly computes $G'$, which uniformly computes $G$, so we may assume that $\Psi(f)$ has access to all these sets.  

Let $\vec{A} = (A_i : i \in \omega)$ be the sequence of sets where $A_x = N(x)$ if $x \in V$ and $A_x = \emptyset$ if $x \notin V$.  Let $F$ denote the set $\{x \in V : \text{$N(x)$ is finite}\}$.  The sequence $\vec{A}$ is uniformly computable from $G$, and the set $F$ is uniformly computable from $G''$.  Ahead of time, fix indices $e_0$, $e_1$, and $e_2$ for Turing functionals with the following properties.
\begin{itemize}
\item If $F$ is infinite, then $\Phi_{e_0}(G')$ is total and computes an infinite independent set $F_0 \subseteq F$.  If $F$ is finite, then $\{x : \Phi_{e_0}(G')(x) = 1\}$ is a finite independent subset of $F$.

\medskip

\item If $F$ is infinite, then $\Phi_{e_1}(G'')$ computes an infinite tree $T \subseteq 2^{<\omega}$ where $[T]$ consists of exactly all $p \in 2^\omega$ such that $F_0 \cap A^{p \rst n}$ is infinite for every $n$.

\medskip

\item The function $\Phi_{e_2}(f)$ is total, and if $F$ is infinite, then $\Phi_{e_2}(f)$ computes an infinite path through $T$.
\end{itemize}
For $e_0$, first fix an index $e$ such that $\Phi_e(G')$ enumerates an independent subset of $F$ in the following way.  On input $0$, $\Phi_e(G')(0)$ uses $G'$ to search for the first pair $\la x, s \ra \in V \times \omega$ where $(\forall y > s)(y \notin N(x))$ and then outputs $x$.  Likewise, $\Phi_e(G')(n+1)$ uses $G'$ to search for the first pair $\la x, s \ra \in V \times \omega$ where $x$ is not adjacent to any of $\Phi_e(G')(0)$, $\Phi_e(G')(1), \dots, \Phi_e(G')(n)$ and $(\forall y > s)(y \notin N(x))$, and then $\Phi_e(G')(n+1)$ outputs $x$.  If $F$ is infinite, then such $x$ are always found, and $\Phi_e(G')$ enumerates an infinite independent subset of $F$.  Otherwise $\Phi_e(G')$ is partial, but the elements that $\Phi_e(G')$ enumerates form a finite independent subset of $F$.  Index $e_0$ may be obtained from $e$ via item~(\ref{it-ure}) above.  Let $e_1$ be an index such that $\Phi_{e_1}(G'')$ computes the set of all $\sigma \in 2^{<\omega}$ such that $(\exists n > |\sigma|)(n \in \Phi_{e_0}(G') \cap A^\sigma)$.  Index $e_2$ may be obtained from index $e_1$ via item~(\ref{it-uDNRWKL}) above.

To compute $H$, follow the strategy from Theorem~\ref{thm-RSgInACA} while simultaneously using $G''$ to check if $F$ is finite.  Notice that by item~(\ref{it-uDNRWKL}), $\Phi_{e_2}(f)$ is total regardless of whether or not $F$ is infinite.  Let $p$ denote the function computed by $\Phi_{e_2}(f)$.  Suppose inductively that $x_0 < x_1 < \cdots < x_{n-1}$ have already been added to $H$ and that each $x_i$ is in $\Phi_{e_0}(G')$.  Note that the set $Y = \bigcup_{i < n}N(x_i)$ is finite because each $x_i$ is in $F$.  Compute $Y$ and its maximum using $G'$.  Use $p$ and $G''$ to simultaneously search for either an $x > x_{n-1}$ such that
\begin{align*}
\Phi_{e_0}(G')(x) = 1 \andd (\forall y \in Y)[x \in N(y) \biimp (p(y) = 1 \andd y \notin F)]\label{fmla-findx}\tag{$a$}
\end{align*}
or an $s$ that is large enough to bound $F$ and its neighbors:
\begin{align*}
(\forall y \in V)[(\exists b \forall z(z \in N(y) \imp z < b)) \imp (y < s \andd \forall z(z \in N(y) \imp z < s))].\label{fmla-finds}\tag{$b$}
\end{align*}
Prenexing~(\ref{fmla-finds}) yields a $\Pi^0_2$ formula relative to $G$, so $G''$ can check if a given $s$ satisfies~(\ref{fmla-finds}).  If $x$ is found, let $x_n = x$ and continue computing $H$ according to the foregoing procedure.  If $s$ is found, abandon this procedure and instead add $\{y \in V : y > s\}$ to the already-computed $\{x_0, x_1, \dots, x_{n-1}\}$ to compute the set $H = \{x_0, x_1, \dots, x_{n-1}\} \cup \{y \in V : y > s\}$.

We now verify that $H$ is an $\rsgr$-solution to $G$.  First suppose that $F$ is infinite.  The verification in this case is very similar to the analogous case in the proof of Theorem~\ref{thm-RSgInACA}.  In this case, each step of the computation of $H$ always finds an $x$ as in~(\ref{fmla-findx}), and there is no $s$ as in~(\ref{fmla-finds}).  $F$ is infinite, so $\Phi_{e_0}(G')$ is total, $\Phi_{e_0}(G')$ computes an infinite independent $F_0 \subseteq F$, and $F_0 \cap A^{p \rst n}$ is infinite for every $n$.  Notice that in this case, the `$y \notin F$' clause in~(\ref{fmla-findx}) is redundant because $p(y) = 1 \imp y \in V \setminus F$ under the current assumptions.  The `$y \notin F$' clause is useful in the next case.  Thus for every finite $Y \subseteq V$, the set $F_0 \cap A^{p \rst (\max{Y}+1)}$ is infinite, and therefore there are infinitely many $x \in F_0$ such that $(\forall y \in Y)(x \in N(y) \biimp p(y) = 1)$.  Thus an $x$ as in~(\ref{fmla-findx}) always exists, so the elements of $H$ are computed by finding such elements.

Consider a $v \in V$.  If $H \cap N(v) \neq \emptyset$, then let $m$ be least such that $x_m \in N(v)$ (and hence also least such that $v \in N(x_m)$).  If $p(v) = 0$, then every $x_n$ with $n > m$ is chosen from $\ol{N(v)}$, so $|H \cap N(v)| = 1$.  If $p(v) = 1$, then every $x_n$ with $n > m$ is chosen from $N(v)$, so $|H \cap N(v)| = \omega$.  Thus for every $v \in V$, either $|H \cap N(v)| = \omega$ or $|H \cap N(v)| \leq 1$.  Furthermore, $H \subseteq F_0$ and $F_0$ is an independent set, so if $v \in H$, then $|H \cap N(v)| = 0$.  Thus $H$ is an $\rsgr$-solution to $G$.

Now suppose that $F$ is finite.  In this case there are only finitely many $x$ with $\Phi_{e_0}(G')(x)=1$, and every sufficiently large $s$ witnesses~(\ref{fmla-finds}).  Thus the computation of $H$ reaches a point where it has computed some finite independent set $\{x_0, x_1, \dots, x_{n-1}\} \subseteq F$, has found an $s$ as in~(\ref{fmla-finds}), and has switched to computing the set $H = \{x_0, x_1, \dots, x_{n-1}\} \cup \{y \in V : y > s\}$.  This $H$ is an $\rsgr$-solution to $G$.  Consider a $v \in V$.  If $v \notin F$, then $N(v)$ is infinite and $H$ contains almost every element of $V$, so $|H \cap N(v)| = \omega$.  Suppose instead that $v \in F$.  Then $\{y \in V : y > s\} \cap N(v) = \emptyset$ because the elements of $F$ have no neighbors greater than $s$.  If $\{x_0, x_1, \dots, x_{n-1}\} \cap N(v) \neq \emptyset$, let $m$ be least such that $x_m \in N(v)$.  Then also $v \in N(x_m)$, so the remaining $x_i$ for $m < i < n$ chosen according to~(\ref{fmla-findx}) are chosen from $\ol{N(v)}$ because $v \in F$.  Thus $|\{x_0, x_1, \dots, x_{n-1}\} \cap N(v)| \leq 1$, so $|H \cap N(v)| \leq 1$.  Lastly, suppose that $v \in H$.  If $v > s$, then $|H \cap N(v)| = \omega$ because $v$ has infinitely many neighbors.  If $v \in \{x_0, x_1, \dots, x_{n-1}\}$, then $\{x_0, x_1, \dots, x_{n-1}\}  \cap N(v) = \emptyset$ because $\{x_0, x_1, \dots, x_{n-1}\}$ is independent, and $\{y \in V : y > s\} \cap N(v) = \emptyset$ because $\{x_0, x_1, \dots, x_{n-1}\} \subseteq F$ and the elements of $F$ have no neighbors greater than $s$.  Thus $|H \cap N(v)| = 0$.  So $H$ is an $\rsgr$-solution to $G$.
\end{proof}

The remainder of this section is dedicated to showing that $\npdnr{3}_2 \leqW \rsg$.  We then observe that $\leqW$ can be improved to $\leqsW$.  Indeed, it is straightforward to show that $\rsg$ is a cylinder by an argument analogous to that of Proposition~\ref{prop-wRSgCyl}.  The proof that $\npdnr{3}_2 \leqsW \rsg$ makes use of several intermediate $\coh$-inspired problems interpolating between $\npdnr{3}_2$ and $\rsg$.

\begin{Definition}{\ }
\begin{itemize}

\item $\ocoh$ (for \emph{omniscient $\coh$}) is the following multi-valued function.
\begin{itemize}
\item Input/instance:  A sequence $\vec{A} = (A_i : i \in \omega)$ of subsets of $\omega$.

\smallskip

\item Output/solution:  A function $p \in 2^\omega$ such that $A^{p \rst n}$ is infinite for every $n \in \omega$.
\end{itemize}

\medskip

\item $\ocohd$ (for \emph{omniscient $\coh$ in $\mbf{\Delta}^0_2$}) is the following multi-valued function.
\begin{itemize}
\item Input/instance:  A pair $\la f, \vec{A} \ra$, where $f$ is a $\mbf{\Delta}^0_2$-approximation to an infinite set $Z \subseteq \omega$, and $\vec{A} = (A_i : i \in \omega)$ is a sequence of subsets of $\omega$.

\smallskip

\item Output/solution:  A function $p \in 2^\omega$ such that $Z \cap A^{p \rst n}$ is infinite for every $n \in \omega$.
\end{itemize}
\end{itemize}
\end{Definition}

The intuition behind the name \emph{omniscient} $\coh$ is that an $\ocoh$-solution $p$ to an $\ocoh$-instance $\vec{A}$ can be used to compute an $\vec{A}$-cohesive set $C$ where, for each $i$, $C \subseteq^* \ol{A_i}$ if $p(i) = 0$ and $C \subseteq^* A_i$ if $p(i) = 1$.  Thus the function $p$ is omniscient regarding the $\vec{A}$-cohesiveness of $C$ because, for each $i$, $p$ knows whether $C \subseteq^* \ol{A_i}$ or $C \subseteq^* A_i$.

\begin{Definition}{\ }
\begin{itemize}
\item $\infone$ (for \emph{infinity or one}) is the following multi-valued function.
\begin{itemize}
\item Input/instance:  A sequence $\vec{A} = (A_i : i \in \omega)$ of subsets of $\omega$ such that each $x \in \omega$ is in only finitely many of the $A_i$:  $\forall x \exists n \forall i (i > n \imp x \notin A_i)$.

\smallskip

\item Output/solution:  An infinite $D \subseteq \omega$ such that for every $i \in \omega$, either $|D \cap A_i| = \omega$ or $|D \cap A_i| \leq 1$.
\end{itemize}

\medskip

\item $\infoned$ (for \emph{infinity or one in $\mbf{\Delta}^0_2$}) is the following multi-valued function.
\begin{itemize}
\item Input/instance:  A pair $\la f, \vec{A} \ra$, where $f$ is a $\mbf{\Delta}^0_2$-approximation to an infinite set $Z \subseteq \omega$, and $\vec{A} = (A_i : i \in \omega)$ is a sequence of subsets of $\omega$ such that each $x \in Z$ is in only finitely many of the $A_i$:  $(\forall x \in Z)(\exists n) (\forall i)(i > n \imp x \notin A_i)$.

\smallskip

\item Output/solution:  An infinite $D \subseteq Z$ such that for every $i \in \omega$, either $|D \cap A_i| = \omega$ or $|D \cap A_i| \leq 1$.
\end{itemize}

\medskip

\item $\infonesd$ (for \emph{infinity or one almost in $\mbf{\Delta}^0_2$}) is an alternate version of $\infoned$ where the output $D$ is only required to be an infinite $D \subseteq^* Z$, rather than an infinite $D \subseteq Z$.
\end{itemize}
\end{Definition}

We warn the reader that the `$*$' in the name $\infonesd$ refers to the almost inclusion $D \subseteq^* Z$ and is unrelated to the finite parallelization operation and the compositional product operation in the Weihrauch degrees.

Notice how the dichotomy between \emph{infinitely many} and \emph{at most one} in solutions to the $\infone$ problem and its variants reflects the analogous dichotomy in solutions to $\rsg$.

The strategy is to prove a sequence of lemmas showing that
\begin{itemize}
\item $\infoned \equivW \infonesd \equivW \ocohd$,

\smallskip

\item $\npdnr{3}_2 \leqsW \ocohd$, and

\smallskip

\item $\infonesd \leqsW \rsg$.
\end{itemize}
It then follows that $\npdnr{3}_2 \leqW \rsg$.  Out of independent interest, and for the sake of completeness, we first show that $\infone \equivW \ocoh$.

\begin{Proposition}\label{prop-InfOneEquivOCOH}
$\infone \equivW \ocoh$.
\end{Proposition}

\begin{proof}
$\infone \leqW \ocoh$:  Let $\vec{A} = (A_i : i \in \omega)$ be an $\infone$-instance.  Let $\Phi$ be the functional given by $\Phi(\vec{A}) = \vec{B}$, where
\begin{align*}
B_{2i} &= A_i\\
B_{2\la x, i \ra + 1} &= \{k \in \omega : (\forall j \in [i,k])(x \notin A_j)\}.
\end{align*}
The sequence $\vec{B}$ is a valid $\ocoh$-instance, so let $p \in \ocoh(\vec{B})$.  Define a functional $\Psi(\la \vec{A}, p \ra)$ computing a set $D$ as follows.  Using $\Phi$, we may again compute $\Phi(\vec{A}) = \vec{B}$.  Decompose $p$ as $p = q \oplus r$.  The sequence $\vec{A}$ is assumed to be a valid $\infone$-instance, so each $x \in \omega$ is in only finitely many of the sets $A_i$.  Using $r$, we can compute a bound $b \colon \omega \imp \omega$ such that $\forall x (i \geq b(x) \imp x \notin A_i)$.  To do this, observe the following for given $x,i \in \omega$.
\begin{itemize}
\item If $(\exists j \geq i)(x \in A_j)$, then $B_{2\la x,i \ra + 1}$ is finite and $r(\la x,i \ra) = 0$.

\smallskip

\item If $(\forall j \geq i)(x \notin A_j)$, then $B_{2\la x,i \ra + 1}$ is co-finite and $r(\la x,i \ra) = 1$.
\end{itemize}
Thus we may compute $b$ by letting $b(x)$ be the least $i$ such that $r(\la x,i \ra) = 1$.

To compute $D$, suppose inductively that $x_0 < x_1 < \cdots < x_{n-1}$ have already been added to $D$.  Let $m = n + 1 + \max\{b(x_i) : i < n\}$.  The set $A^{q \rst m}$ is infinite because $q \in \ocoh(\vec{A})$ (which is because $q \oplus r \in \ocoh(\vec{B})$).  So choose $x_n$ to be the least member of $A^{q \rst m}\setminus \{x_i : i < n\}$, and add $x_n$ to $D$.

To see that $D \in \infone(\vec{A})$, consider an $i \in \omega$.  If $q(i) = 1$, then $|D \cap A_i| = \omega$ because in this case eventually every $x_n$ is chosen from $A_i$.  Suppose instead that $q(i) = 0$ and that $x_n \in D \cap A_i$ for some $n$.  When choosing $x_k$ for a $k > n$, we use an $m > b(x_n)$ to make the choice.  Furthermore, $b(x_n) > i$ because $x_n \in A_i$.  Thus $m > b(x_n) > i$, so $x_k$ is chosen from $\ol{A_i}$.  This means that if $q(i) = 0$, then at most one $x_n$ can be in $D \cap A_i$.  So $|D \cap A_i| \leq 1$.  Therefore $D \in \infone(\vec{A})$.  In fact, $D$ is also cohesive for $\vec{A}$.  Thus $\Phi$ and $\Psi$ witness that $\infone \leqW \ocoh$.

$\ocoh \leqW \infone$:  For this argument, we make explicit the bijective encoding of binary strings as natural numbers.  For $\sigma \in 2^{<\omega}$, let $\code{\sigma} \in \omega$ denote the code for $\sigma$.  Let $\vec{A} = (A_i : i \in \omega)$ be an $\ocoh$-instance.  Let $\Phi$ be the functional $\Phi(\vec{A}) = \vec{B}$, where $\vec{B} = (B_i : i \in \omega)$ is the sequence given by
\begin{align*}
B_{\code{\sigma}} = A^\sigma \setminus \{x : x \leq \code{\sigma}\}.
\end{align*}
Each $x$ is in only finitely many of the $B_i$, so $\vec{B}$ is a valid input to $\infone$.  Let $D \in \infone(\vec{B})$.  Define a functional $\Psi(\la \vec{A}, D \ra)$ computing a function $p$ as follows.  Using $\Phi$, we may again compute $\Phi(\vec{A}) = \vec{B}$.  The goal is to ensure that $B_{\code{p \rst n}}$ is infinite for every $n$ and therefore that $A^{p \rst n}$ is infinite for every $n$.  To determine $p(0)$, observe that, for the strings $0$ and $1$, $B_{\code{0}} =^* A^0 = \ol{A_0}$ and $B_{\code{1}} =^* A^1 = A_0$.  Thus $B_{\code{0}} \cup B_{\code{1}} =^* \omega$, and therefore either $|D \cap B_{\code{0}}| = \omega$ or $|D \cap B_{\code{1}}| = \omega$.  Search $D \cap B_{\code{0}}$ and $D \cap B_{\code{1}}$ until either discovering that $|D \cap B_{\code{0}}| \geq 2$ or discovering that $|D \cap B_{\code{1}}| \geq 2$.  If we first discover that $|D \cap B_{\code{0}}| \geq 2$, then in fact $|D \cap B_{\code{0}}| = \omega$ because $D \in \infone(\vec{B})$.  So output $p(0) = 0$.  Likewise, if we first discover that $|D \cap B_{\code{1}}| \geq 2$, then $|D \cap B_{\code{1}}| = \omega$, so output $p(0) = 1$.  Now compute the rest of $p$ in the same way.  Suppose inductively that we have computed $p \rst n$ and that $|D \cap B_{\code{p \rst n}}| = \omega$.  As
\begin{align*}
B_{\code{p \rst n}} =^* A^{p \rst n} = A^{(p \rst n)^\smf 0} \cup A^{(p \rst n)^\smf 1} =^* B_{\code{(p \rst n)^\smf 0}} \cup B_{\code{(p \rst n)^\smf 1}},
\end{align*}
it must be that either $|D \cap B_{\code{(p \rst n)^\smf 0}}| = \omega$ or $|D \cap B_{\code{(p \rst n)^\smf 1}}| = \omega$.  As before, search $D \cap B_{\code{(p \rst n)^\smf 0}}$ and $D \cap B_{\code{(p \rst n)^\smf 1}}$ until discovering that $|D \cap B_{\code{(p \rst n)^\smf 0}}| \geq 2$ or that $|D \cap B_{\code{(p \rst n)^\smf 1}}| \geq 2$.  If we first discover that $|D \cap B_{\code{(p \rst n)^\smf 0}}| \geq 2$, then $|D \cap B_{\code{(p \rst n)^\smf 0}}| = \omega$, so output $p(n) = 0$.  If we first discover that $|D \cap B_{\code{(p \rst n)^\smf 1}}| \geq 2$, then $|D \cap B_{\code{(p \rst n)^\smf 1}}| = \omega$, so output $p(n) = 1$.  In the end, $D \cap B_{\code{p \rst n}}$ is infinite for every $n$, so $B_{\code{p \rst n}}$ is infinite for every $n$, so $A^{p \rst n}$ is infinite for every $n$.  Thus $p \in \ocoh(\vec{A})$, so $\Phi$ and $\Psi$ witness that $\ocoh \leqW \infone$.
\end{proof}

\begin{Lemma}\label{lem-INFONEDequivOCOHD}
$\infoned \equivW \infonesd \equivW \ocohd$.
\end{Lemma}

\begin{proof}
We show that
\begin{align*}
\ocohd \leqW \infonesd \leqsW \infoned \leqW \ocohd.
\end{align*}
Notice that the middle $\leqsW$-reduction is trivial.

$\infoned \leqW \ocohd$:  Let $\la f, \vec{A} \ra$ be an $\infoned$-instance.  The proof is the same as that of the $\infone \leqW \ocoh$ direction of Proposition~\ref{prop-InfOneEquivOCOH}, except now we must also use $\ocohd$ to compute the set $Z$ approximated by $f$.  Let $\Phi$ be the functional given by $\Phi(\la f, \vec{A} \ra) = \la f, \vec{B} \ra$, where
\begin{align*}
B_{3i} &= A_i\\
B_{3\la x, i \ra + 1} &= \{k \in \omega : (\forall j \in [i,k])(x \notin A_j)\}\\
B_{3x + 2} &= \{s \in \omega : f(x,s)=1\}.
\end{align*}
The pair $\la f, \vec{B} \ra$ is a valid $\ocohd$-instance, so let $p \in \ocohd(\la f, \vec{B} \ra)$.  Define a functional $\Psi(\la \la f, \vec{A} \ra, p \ra)$ computing a set $D$ as follows.  Using $\Phi$, we may again compute $\Phi(\la f, \vec{A} \ra) = \la f, \vec{B} \ra$.  Decompose $p$ into functions $q,r,h \colon \omega \imp \omega$ by letting
\begin{align*}
q(n) = p(3n) &&
r(n) = p(3n+1) &&
h(n) = p(3n+2) &&
\end{align*}
for each $n \in \omega$.  As $\la f, \vec{A} \ra$ is a valid $\infoned$-instance, the function $f$ is a $\mbf{\Delta}^0_2$-approximation to an infinite set $Z$.  We have that $h$ is the characteristic function of $Z$.  To see this, observe the following for a given $x \in \omega$.
\begin{itemize}
\item If $\lim_s f(x,s) = 0$, then $B_{3x+2}$ is finite, so $h(x) = 0$.

\smallskip

\item If $\lim_s f(x,s) = 1$, then $B_{3x+2}$ is co-finite, so $h(x)=1$.
\end{itemize}
Now proceed as in the proof of Proposition~\ref{prop-InfOneEquivOCOH}.  By assumption, each $x \in Z$ is in at most finitely many of the sets $A_i$, so the function $r$ computes a bound $b \colon Z \imp \omega$ such that $(\forall x \in Z)(i \geq b(x) \imp x \notin A_i)$ as before.  To compute $D$, suppose inductively that the numbers $x_0 < x_1 < \cdots < x_{n-1}$ in $Z$ have already been added to $D$.  Let $m = n + 1 + \max\{b(x_i) : i < n\}$.  The set $Z \cap A^{q \rst m}$ is infinite, so choose $x_n$ to be the least member of $(Z \cap A^{q \rst m})\setminus \{x_i : i < n\}$, and add $x_n$ to $D$.  The result is that $D \in \infoned(\la f, \vec{A} \ra)$, so $\Phi$ and $\Psi$ witness that $\infoned \leqW \ocohd$.

$\ocohd \leqW \infonesd$:  The proof is the same as that of the $\ocoh \leqW \infone$ direction of Proposition~\ref{prop-InfOneEquivOCOH}.  Again, let $\code{\sigma} \in \omega$ denote the number coding the string $\sigma \in 2^{<\omega}$.  Let $\la f, \vec{A} \ra$ be an $\ocohd$-instance.  Let $\Phi$ be the functional $\Phi(\la f, \vec{A} \ra) = \la f, \vec{B} \ra$, where $\vec{B} = (B_i : i \in \omega)$ is the sequence given by $B_{\code{\sigma}} = A^\sigma \setminus \{x : x \leq \code{\sigma}\}$.  Then $\la f, \vec{B} \ra$ is a valid $\infonesd$-instance, so let $D \in \infonesd(\la f, \vec{B} \ra)$.  Let $\Psi(\la \la f, \vec{A} \ra, D \ra)$ be the functional computing a function $p$ by the same procedure as in the proof of Proposition~\ref{prop-InfOneEquivOCOH}.  As $D \in \infonesd(\la f, \vec{B} \ra)$, we have that $D \subseteq^* Z$, where $Z$ is the set approximated by $f$.  The function $p$ has the property that $D \cap B_{\code{p \rst n}}$ is infinite for every $n$, hence also $Z \cap B_{\code{p \rst n}}$ is infinite for every $n$ because $D \subseteq^* Z$.  Therefore $Z \cap A^{p \rst n}$ is infinite for every $n$, so $p \in \ocohd(\la f, \vec{A} \ra)$.  Thus $\Phi$ and $\Psi$ witness that $\ocohd \leqW \infonesd$.
\end{proof}

\begin{Lemma}\label{lem-3DNR2redOCOHD}
$\npdnr{3}_2 \leqsW \ocohd$.
\end{Lemma}

\begin{proof}
Let $p$ be an input to $\npdnr{3}_2$.  Define a functional $\Phi(p) = \la f, \vec{A} \ra$ as follows.  The function $f$ is a $\mbf{\Delta}^0_2$-approximation to the set
\begin{align*}
Z = \Bigl\{\sigma \in 2^{<\omega} : \bigl(\forall e < |\sigma|\bigr)\bigl[(\Phi_{e,|\sigma|}(p')(e)\ua \imp \sigma(e) = 0) \andd (\Phi_{e,|\sigma|}(p')(e)\da \imp \sigma(e) = 1)\bigr]\Bigr\}
\end{align*}
of $p'$-approximations to $p''$.  To compute $f$, for each $s \in \omega$, let $\tau_s \in 2^s$ denote the $p$-approximation to $p'$ where
\begin{align*}
\bigl(\forall e < s\bigr)\bigl[(\Phi_{e,s}(p)(e)\ua \imp \tau_s(e) = 0) \andd (\Phi_{e,s}(p)(e)\da \imp \tau_s(e) = 1)\bigr].
\end{align*}
Now, for $\sigma \in 2^{<\omega}$ and $s \in \omega$, define $f(\sigma, s) = 1$ if
\begin{align*}
\bigl(\forall e < |\sigma|\bigr)\bigl[(\Phi_{e,\min\{|\sigma|,s\}}(\tau_s)(e)\ua \imp \sigma(e) = 0) \andd (\Phi_{e,\min\{|\sigma|,s\}}(\tau_s)(e)\da \imp \sigma(e) = 1)\bigr],
\end{align*}
and define $f(\sigma, s) = 0$ otherwise.  Define $\vec{A} = (A_i : i \in \omega)$ by $A_e = \{\sigma \in 2^{<\omega} : \Phi_{e,|\sigma|}(\sigma)(e)\da = 0\}$.

The pair $\la f, \vec{A} \ra$ is a valid $\ocohd$-instance, so let $g \in \ocohd(\la f, \vec{A} \ra)$.  Let $\Psi$ be the identity functional.  We show that $g$ is $\tdnr_2$ relative to $p''$.  Suppose that $\Phi_e(p'')(e)\da = 0$, and let $\alpha \subseteq p''$ be long enough so that $\Phi_{e,|\alpha|}(\alpha)(e)\da = 0$.  All but finitely many strings in $Z$ extend $\alpha$, so $Z \subseteq^* A_e$.  It must therefore be that $g(e) = 1$ because $Z \cap \ol{A_e}$ is finite.  Now suppose instead that $\Phi_e(p'')(e)\da = n$ for some $n > 0$.  Reasoning as before, we conclude that all but finitely many of the $\sigma \in Z$ satisfy $\Phi_{e,|\sigma|}(\sigma)(e)\da = n$.  Therefore $Z \cap A_e$ is finite, so $g(e) = 0$.  We have shown that
\begin{align*}
\forall e\bigl[(\Phi_e(p'')(e)\da = 0 \imp g(e) = 1) \andd (\Phi_e(p'')(e)\da > 0 \imp g(e) = 0)\bigr].
\end{align*}
Therefore $g \in \npdnr{3}_2(p)$, so $\Phi$ and $\Psi$ witness that $\npdnr{3}_2 \leqsW \ocohd$.
\end{proof}

\begin{Lemma}\label{lem-INFONESDredRSG}
$\infonesd \leqsW \rsg$.
\end{Lemma}

\begin{proof}
Ahead of time, recursively partition $\omega$ into three sets $\omega = W \cup X \cup Y$, and write $W = \{w_i : i \in \omega\}$, $X = \{x_i : i \in \omega\}$, and $Y = \{y_i : i \in \omega\}$.  For example, we may take $w_i = 3i$, $x_i = 3i+1$, and $y_i = 3i+2$ for each $i$.  Let $\la f, \vec{A} \ra$ be an $\infonesd$-instance.  Define a functional $\Phi(\la f, \vec{A} \ra)$ computing the graph $G = (V,E)$ where $V = \omega = W \cup X \cup Y$ and 
\begin{align*}
E = &\{(w_i, y_j) \in W \times Y : j < i\}\\
&\cup \{(w_i, x_n) \in W \times X : i < n\}\\
&\cup \{(x_n, y_i) \in X \times Y : n \in A_i\}\\
&\cup \{(x_n, x_s) \in X \times X : (n < s) \andd (f(n,s) = 0)\}.
\end{align*}
The graph $G$ is a valid $\rsg$-instance, so let $H$ be an $\rsg$-solution to $G$.  Let $\Psi$ be the functional $\Psi(H) = D$, where $D = \{n \in \omega : x_n \in H\}$.  We show that $D \in \infonesd(\la f, \vec{A} \ra)$.

Let $Z = \lim f$.  We show that if $n \in Z$, then $x_n$ has only finitely many neighbors.  Clearly $x_n$ has only finitely many neighbors in $W$.  Among the vertices of $Y$, $(x_n, y_i) \in E$ if and only if $n \in A_i$.  The pair $\la f, \vec{A} \ra$ is assumed to be a valid $\infonesd$-instance, and $n \in Z$, which means that $n$ is in only finitely many of the $A_i$.  Thus $x_n$ has only finitely many neighbors in $Y$.  Among the vertices of $X$, we know that if $s$ is sufficiently large, then $(x_n, x_s) \notin E$.  This is because $\lim_s f(n,s) = 1$ (as $n \in Z$), but for $s > n$, $(x_n, x_s) \in E$ if and only if $f(n,s) = 0$.  Thus $x_n$ has only finitely many neighbors in $X$.  Thus if $n \in Z$, then $x_n$ has only finitely many neighbors.

Now we show that $H \cap X$ is infinite and therefore that $D$ is infinite.  To do this, we first show that $|H \cap W| \leq 1$.  Suppose for a contradiction that $i < j$ are such that $w_i, w_j \in H$.  $Z$ is infinite, so let $n > j$ be such that $n \in Z$.  Then $x_n$ has at least the two neighbors $w_i$ and $w_j$ in $H$.  However, $x_n$ cannot have infinitely many neighbors in $H$ because $x_n$ has only finitely many neighbors in $V$.  This is a contradiction, so $|H \cap W| \leq 1$.  Therefore $H \subseteq^* X \cup Y$.  If $|H \cap Y| \leq 2$, then $H \subseteq^* X$, so $H \cap X$ is infinite.  Suppose instead that $|H \cap Y| \geq 2$ and that $i < j$ are such that $y_i, y_j \in H$.  Let $k > j$.  Then $w_k$ has at least the two neighbors $y_i$ and $y_j$ in $H$, so $w_k$ must have infinitely many neighbors in $H$.  Vertex $w_k$ has only finitely many neighbors in $Y$ and no neighbors in $W$, so $w_k$ must have infinitely many neighbors in $H \cap X$.  Therefore $H \cap X$ is infinite in this case as well.

Now we show that $|H \cap \{x_n : n \notin Z\}| \leq 1$ and therefore that $D \subseteq^* Z$.  Suppose for a contradiction that $m < n$ are such that $x_m, x_n \in H$, but also $m, n \notin Z$.  Let $s_0 > n$ be large enough so that $f(m,s) = f(n,s) = 0$ for all $s > s_0$, and, as $Z$ is infinite, let $s > s_0$ be such that $s \in Z$.  Then $x_s$ is adjacent to both $x_m$ and $x_n$, so $x_s$ has at least two neighbors in $H$.  However, $x_s$ cannot have infinitely many neighbors in $H$ because $s \in Z$, which means that $x_s$ has only finitely many neighbors all together.  This is a contradiction, so $|H \cap \{x_n : n \notin Z\}| \leq 1$.

We have shown that $D$ is infinite and that $D \subseteq^* Z$.  To finish the proof, we show that, for every $i \in \omega$, either $|D \cap A_i| = \omega$ or $|D \cap A_i| \leq 1$.  Consider an $i \in \omega$, suppose that $|D \cap A_i| \geq 2$, and let $m < n$ be such that $m,n \in D \cap A_i$.  Then $x_m, x_n \in H$ and $y_i$ is adjacent to both $x_m$ and $x_n$.  Therefore $y_i$ has infinitely many neighbors in $H$, almost all of which are in $X$ because $|H \cap W| \leq 1$ and $y_i$ has no neighbors in $Y$.  That is, $N(y_i) \cap X \cap H$ is infinite.  This means that $D \cap A_i$ is infinite, as desired.  Therefore $D \in \infonesd(\la f, \vec{A} \ra)$, so $\Phi$ and $\Psi$ witness that $\infonesd \leqsW \rsg$.
\end{proof}

\begin{Theorem}\label{thm-RSGequivWKLjj}
The following multi-valued functions are pairwise Weihrauch equivalent:  $\wkl''$, $\rsg$, $\rsgr$, $\npdnr{3}_2$, $\ocohd$, $\infoned$, and $\infonesd$.  In particular,
\begin{align*}
\wkl'' \equivW \rsg \equivW \rsgr.
\end{align*}
\end{Theorem}

\begin{proof}
We have that
\begin{align*}
\rsg &\leqsW \rsgr\\
&\leqsW \npdnr{3}_2 & &\text{by Lemma~\ref{lem-RSGRred3DNR2}}\\
&\leqsW \ocohd & &\text{by Lemma~\ref{lem-3DNR2redOCOHD}}\\
&\equivW \infoned \equivW \infonesd & & \text{by Lemma~\ref{lem-INFONEDequivOCOHD}}\\
&\leqsW \rsg & &\text{by Lemma~\ref{lem-INFONESDredRSG}}
\end{align*}
and that $\wkl'' \equivsW \npdnr{3}_2$ by Lemma~\ref{lem-WKLequivDNR2}.
\end{proof}

The Weihrauch equivalences of Theorem~\ref{thm-RSGequivWKLjj} can be improved to strong Weihrauch equivalences by noticing that the multi-valued functions involved are all cylinders.  It fact, it suffices to show that $\ocohd$, $\infoned$, and $\infonesd$ are cylinders, from which it follows that the Weihrauch equivalences of Lemma~\ref{lem-INFONEDequivOCOHD} can be improved to strong Weihrauch equivalences.

\begin{Lemma}\label{lem-CylinderHelp}
\begin{sloppypar}
The following multi-valued functions are cylinders:  $\ocoh$, $\ocohd$, $\infoned$, and $\infonesd$.
\end{sloppypar}
\end{Lemma}

\begin{proof}
For $\ocoh$, convert a given $g \in \omega^\omega$ and $\ocoh$-instance $\vec{A}$ into the $\ocoh$-instance $\vec{B}$ where
\begin{align*}
B_{2i} &= A_i\\
B_{2\la m, n \ra + 1} &=
\begin{cases}
\omega & \text{if $g(m) = n$}\\
\emptyset & \text{if $g(m) \neq n$.}
\end{cases}
\end{align*}
Let $p$ be an $\ocoh$-solution to $\vec{B}$, and decompose $p$ as $p = q \oplus r$.  Then $q$ is an $\ocoh$-solution to $\vec{A}$, and $g$ can be uniformly computed from $r$ because for each $m$, $g(m)$ is the unique $n$ such that $r(\la m, n \ra) = 1$.  A similar argument shows that $\ocohd$ is a cylinder.

\begin{sloppypar}
For $\infoned$, convert a given $g \in \omega^\omega$ and $\infoned$-instance $\la f, \vec{A} \ra$ into the $\infoned$-instance $\la h, \vec{B} \ra$, where $h \colon \omega^{<\omega} \times \omega \imp 2$ given by
\end{sloppypar}
\begin{align*}
h(\sigma, s) =
\begin{cases}
f(|\sigma|, s) & \text{if $\sigma \subseteq g$}\\
0 & \text{if $\sigma \nsubseteq g$}
\end{cases}
\end{align*}
is a $\mbf{\Delta}^0_2$-approximation to $\{g \rst n : n \in \lim f\}$, and $B_i = \{g \rst n : n \in A_i\}$.  Let $D \in \infoned(\la h, \vec{B} \ra)$.  $D$ uniformly computes $g$ because $D$ is an infinite subset of $\{g \rst n : n \in \omega\}$.  $D$ therefore also uniformly computes the set $\{n : g \rst n \in D\}$, which is an $\infoned$-solution to $\la f, \vec{A} \ra$.

Showing that $\infonesd$ is a cylinder is slightly more complicated.  Convert a given $g \in \omega^\omega$ and $\infonesd$-instance $\la f, \vec{A} \ra$ into the $\infonesd$-instance $\la h, \vec{B} \ra$, where $h$ is the $\mbf{\Delta}^0_2$-approximation to $\{g \rst n : n \in \lim f\}$ as above, and $\vec{B}$ is given by
\begin{align*}
B_{2i} &= \{g \rst n : n \in A_i\}\\
B_{2\sigma + 1} &= \{\tau : \sigma \subseteq \tau\}.
\end{align*}
Each $\tau$ is in at most finitely many sets of the form $B_{2i}$ because $\la f, \vec{A} \ra$ is an $\infonesd$-instance, and each $\tau$ is in at most finitely many sets of the form $B_{2\sigma +1}$ because $\tau \in B_{2\sigma+1}$ if and only if $\sigma \subseteq \tau$.  Thus $\la h, \vec{B} \ra$ is indeed a $\infonesd$-instance.  Let $D \in \infonesd(\la h, \vec{B} \ra)$.  If $\sigma$ and $\tau$ are two elements of $D$ with $\sigma \subsetneq \tau$, then $\sigma \subseteq g$.  To see this, suppose that $\sigma, \tau \in D$ and $\sigma \subsetneq \tau$.  Then $|D \cap B_{2\sigma+1}| \geq 2$, so $|D \cap B_{2\sigma+1}| = \omega$.  However, $D \subseteq^* \{g \rst n : n \in \omega\}$, so if $\sigma \nsubseteq g$, then $D \cap B_{2\sigma+1}$ would be finite.  Thus it must be that $\sigma \subseteq g$.  $D$ therefore uniformly computes $g$ by the following procedure.  Given $n$, search for $\sigma, \tau \in D$ with $|\sigma| > n$ and $\sigma \subsetneq \tau$, and output $\sigma(n)$.  Since $D$ uniformly computes $g$, $D$ therefore also uniformly computes the set $\{n : g \rst n \in D\}$,  which is an $\infonesd$-solution to $\la f, \vec{A} \ra$ as in the $\infoned$ case.
\end{proof}

\begin{Corollary}\label{cor-SWequiv}
The following multi-valued functions are pairwise strongly Weihrauch equivalent:  $\wkl''$, $\rsg$, $\rsgr$, $\npdnr{3}_2$, $\ocohd$, $\infoned$, and $\infonesd$.  In particular,
\begin{align*}
\wkl'' \equivsW \rsg \equivsW \rsgr.
\end{align*}
\end{Corollary}

\begin{proof}
In the proof of Theorem~\ref{thm-RSGequivWKLjj}, the Weihrauch equivalences of Lemma~\ref{lem-INFONEDequivOCOHD} can be improved to strong Weihrauch equivalences because $\ocohd$, $\infoned$, and $\infonesd$ are cylinders by Lemma~\ref{lem-CylinderHelp}.
\end{proof}

We end this section with a number of remarks and observations about its results.

First, it follows from Corollary~\ref{cor-SWequiv} that $\wkl''$, $\npdnr{3}_2$, $\rsg$, and $\rsgr$ are all cylinders because they are all strongly Weihrauch equivalent to $\ocohd$.  It is well-known, and easy to show directly, that $\wkl$ and its jumps are cylinders (see~\cites{BrattkaGherardi, BrattkaGherardiMarcone}).  A proof that each $\npdnr{n}_2$ is a cylinder is also implicit in the proof of Lemma~\ref{lem-RSGRred3DNR2}.  Direct proofs that $\rsg$ and $\rsgr$ are cylinders can be given using the method of Proposition~\ref{prop-wRSgCyl}.  On the other hand, the multi-valued function $\infone$ is not a cylinder, which can be seen by observing that every pair of $\infone$ instances has a common solution.

Second, it is possible to short-circuit the proof that $\npdnr{3}_2 \leqW \rsg$ by considering the multi-valued function $\mc{F}$ which is defined in the same way as $\infonesd$, except that the output is restricted to sets $D$ with $|D \setminus Z| \leq 1$ instead of with $D \subseteq^* Z$.  The proof of Lemma~\ref{lem-INFONESDredRSG} already shows that $\mc{F} \leqsW \rsg$, and one may show that $\npdnr{3}_2 \leqW \mc{F}$ by a proof similar to that of Lemma~\ref{lem-3DNR2redOCOHD}.  However, we find that $\infoned$ and $\infonesd$ are more natural problems than $\mc{F}$ and also that the equivalence with $\ocohd$ is worth establishing.  We hope that the equivalence among $\wkl''$, $\ocohd$, $\infoned$, and $\infonesd$ will be helpful in future research.  Similarly, using the techniques of this section, one may also show that $\ocoh \equivsW \wkl'$ and therefore that $\wkl' \equivW \infone \equivW \ocoh$.

Third, one may wish to consider the versions of $\rsg$ and $\rsgr$ obtained by restricting the inputs to graphs with $V = \omega$.  Denote these restrictions by $\rsg \rst_{V = \omega}$ and $\rsgr \rst_{V = \omega}$.  It is straightforward to show that $\rsg \rst_{V = \omega} \equivW \rsg$, that $\rsgr \rst_{V = \omega} \equivW \rsgr$, and therefore that all four multi-valued functions are pairwise Weihrauch equivalent.  Notice, however, that the strong Weihrauch reduction of Lemma~\ref{lem-INFONESDredRSG} builds a graph with $V = \omega$ and thus shows that $\infonesd \leqsW \rsg \rst_{V = \omega}$.  Therefore $\rsg \leqsW \rsg \rst_{V = \omega} \leqsW \rsgr \rst_{V = \omega} \leqsW \rsgr \equivsW \rsg$, hence all four multi-valued functions are in fact strongly Weihrauch equivalent.

Fourth, for a multi-valued function $\mc{F} \colon \wsubseteq \omega^\omega \rra \omega^\omega$, the \emph{parallelization} of $\mc{F}$ is the multi-valued function corresponding to the problem of solving countably many $\mc{F}$-instances simultaneously in parallel.
\begin{Definition}
Let $\mc{F} \colon \wsubseteq \omega^\omega \rra \omega^\omega$ be a multi-valued function.  The \emph{parallelization} $\wh{\mc{F}}$ of $\mc{F}$ is the following multi-valued function.
\begin{itemize}
\item Input/instance:  A sequence $(f_i : i \in \omega)$ with $f_i \in \dom(\mc{F})$ for each $i \in \omega$.

\smallskip

\item Output/solution:  A sequence $(g_i : i \in \omega)$ with $g_i \in \mc{F}(f_i)$ for each $i \in \omega$.
\end{itemize}
\end{Definition}

We have that $\wh{\rt^2_2} \equivW \wkl''$ by \cite{BrattkaRakotoniaina}*{Corollary~4.18}, so $\rsg$ and $\rsgr$ are Weihrauch equivalent to $\wh{\rt^2_2}$ by Theorem~\ref{thm-RSGequivWKLjj}.  Hence the relationship between the computational strength of $\rsg$ and $\rt^2_2$ can be summarized thusly.  The computational strength required to solve one $\rsg$-instance exactly corresponds to the strength required to solve countably many $\rt^2_2$-instances simultaneously in parallel.

Fifth, let $\rsg \rst_{\mathrm{LF}}$ and $\rsgr \rst_{\mathrm{LF}}$ denote the restrictions of $\rsg$ and $\rsgr$ to infinite locally finite graphs.  Then $\rsg \rst_{\mathrm{LF}} \equivsW \rsgr \rst_{\mathrm{LF}} \equivsW \lim$.  In fact, the proof that $\rsg \rst_{\mathrm{LF}}$ is equivalent to $\aca$ over $\rca$ essentially establishes a Weihrauch equivalence between $\rsg \rst_{\mathrm{LF}}$ and $\lim$.  This gives precise meaning to the introduction's claim that the reversal of $\rsg$ in Theorem~\ref{thm-ACA_RS_eq} does not use the full computational strength of $\rsg$.  The reversal considers only locally finite graphs, thereby establishing that $\rca + \rsg \rst_{\mathrm{LF}} \vdash \aca$.  However, $\rsg \rst_{\mathrm{LF}} \lstc \rsg$ because $\rsg \rst_{\mathrm{LF}} \equivsW \lim \lstc \wkl'' \equivsW \rsg$.

\begin{Proposition}\label{prop-RSgLF}
$\rsg \rst_{\mathrm{LF}} \equivsW \rsgr \rst_{\mathrm{LF}} \equivsW \lim$.
\end{Proposition}

\begin{proof}
It suffices to show that $\rsgr \rst_{\mathrm{LF}} \leqW \lim$ and that $\lim \leqW \rsg \rst_{\mathrm{LF}}$ because all three problems are cylinders (which can be shown for $\rsg \rst_{\mathrm{LF}}$ and $\rsgr \rst_{\mathrm{LF}}$ by the method of Proposition~\ref{prop-wRSgCyl}).  

For $\rsgr \rst_{\mathrm{LF}} \leqW \lim$, let $G = (V, E)$ be an infinite locally finite graph.  Use $\lim$ to compute the function $b \colon V \imp \Pf(V)$ where $b(x) = N(x)$ for all $x \in V$.  Then compute an $\rsg$-solution to $G$ by the same procedure as in the proof of Proposition~\ref{prop-RSgHighRec}.

For $\lim \leqW \rsg \rst_{\mathrm{LF}}$, we use that $\lim$ is strongly Weihrauch equivalent to the Turing jump function $\mathsf{J} \colon \omega^\omega \imp \omega^\omega$ given by $\mathsf{J}(p) = p'$ (see, for example, \cite{BrattkaGherardiPaulySurvey}*{Theorem~6.7}).  Given $p$, uniformly compute an injection $f \colon \omega \imp \omega$ with $\ran(f) = p'$.  Then proceed exactly as in the proof of the reversal in Theorem~\ref{thm-ACA_RS_eq}.
\end{proof}
If we insist on the restrictions of $\rsg$ and $\rsgr$ to locally finite graphs with $V = \omega$, then the resulting problems are not cylinders (because every pair of instances has a common solution).  In this case, we obtain $\equivW$ in place of $\equivsW$ in the above proposition.

Sixth, we may conclude from Corollary~\ref{cor-SWequiv} that $\rsg$ and $\rsgr$ have \emph{universal instances}.  $\rt^2_2$, on the other hand, does not have universal instances~\cite{MiletiThesis}*{Corollary~5.4.8}.

\begin{Definition}
Let $\mc{F} \colon \wsubseteq \omega^\omega \rra \omega^\omega$ be a multi-valued function, thought of as a mathematical problem.  A \emph{universal instance} of $\mc{F}$ is a recursive $\mc{F}$-instance $f_*$ with the property that if $g_*$ is any $\mc{F}$-solution to $f_*$ and $f$ is any recursive $\mc{F}$-instance, then there is an $\mc{F}$-solution $g$ to $f$ with $g \leqT g_*$.
\end{Definition}

\begin{Corollary}
$\rsg$ and $\rsgr$ have universal instances.
\end{Corollary}

\begin{proof}
We show that $\rsg$ has a universal instance.  The argument for $\rsgr$ is similar.

By Corollary~\ref{cor-SWequiv}, let $\Phi$ and $\Psi$ be functionals witnessing that $\npdnr{3}_2 \leqsW \rsg$, and let $\wh{\Phi}$ and $\wh{\Psi}$ be functionals witnessing that $\rsg \leqsW \npdnr{3}_2$.  Let $0$ denote the the function with constant value $0$, and let $G_* = \Phi(0)$.  The function $0$ is a valid $\npdnr{3}_2$-instance, so $G_*$ is a valid $\rsg$-instance.  We show that $G_*$ is universal.  First, $G_*$ is recursive because $0$ is recursive.  Now, let $H_*$ be any $\rsg$-solution to $G_*$, and let $G$ be any recursive $\rsg$-instance.  Let $p = \wh{\Phi}(G)$, which is recursive and therefore satisfies $p'' \equivT 0''$.  Let $f_* = \Psi(H_*)$, which is $\tdnr_2$ relative to $0''$ and therefore computes a function $f$ that is $\tdnr_2$ relative to $p''$.  This $f$ is a $\npdnr{3}_2$-solution to $p$, so $\wh{\Psi}(f)$ is an $\rsg$-solution to $G$.  We have that $\wh{\Psi}(f) \leqT f \leqT f_* \leqT H_*$, so $H_*$ computes an $\rsg$-solution to $G$.
\end{proof}
More generally, one may show that if $\mc{F}$ and $\mc{G}$ are Weihrauch equivalent problems and $\mc{F}$ has a universal instance, then so does $\mc{G}$.  The Weihrauch equivalence between $\rsg$ and $\wkl''$ also yields the following degree-theoretic corollary.

\begin{Corollary}
Let $\mbf{a}$ be any Turing degree.  A Turing degree $\mbf{b}$ computes solutions to every $\rsg$-instance $G \leqT \mbf{a}$ if and only if $\mbf{b}$ is $\pa$ relative to $\mbf{a}''$.  The statement also holds with $\rsgr$ in place of $\rsg$.
\end{Corollary}

Lastly, Figure~\ref{fig-RSg} depicts the position of the Rival--Sands theorem and a selection of related principles in the Weihrauch degrees.

\begin{figure}[h]
\begin{tikzpicture}
\node (whrt22) at (-2,4) {$\wh{\rt^2_2}$};
\node (wkl2) at (0,4) {$\wkl''$};
\node (rsg) at (2,4) {$\rsg$};
\node (rsgr) at (4,4) {$\rsgr$};
\node (ocohd) at (6.5,4) {$\ocohd$};
\node (infoned) at (10,4) {$\infoned$};

\node (wkl1) at (0,2) {$\wkl'$};
\node (ocoh) at (2.5,2) {$\ocoh$};
\node (infone) at (5,2) {$\infone$};

\node (lim) at (0,0) {$\lim$};
\node (rsgLF) at (2.5,0) {$\rsg \rst_{\mathrm{LF}}$};
\node (rsgrLF) at (5,0) {$\rsgr \rst_{\mathrm{LF}}$};

\node(rt22) at (-3, -1) {$\rt^2_2$};
\node (wkl) at (0, -2) {$\wkl$};

\draw[<->] (whrt22) -- (wkl2);
\draw[<->] (wkl2) -- (rsg);
\draw[<->] (rsg) -- (rsgr);
\draw[<->] (rsgr) -- (ocohd);
\draw[<->] (ocohd) -- (infoned);

\draw[->] (wkl2) -- (wkl1);

\draw[<->] (wkl1) -- (ocoh);
\draw[<->] (ocoh) -- (infone);

\draw[->] (wkl1) -- (lim);
\draw[<->] (lim) -- (rsgLF);
\draw[<->] (rsgLF) -- (rsgrLF);

\draw[->] (lim) -- (wkl);
\draw[->] (whrt22) -- (rt22);
\end{tikzpicture}
\caption{The Rival--Sands theorem and selected related principles in the Weihrauch degrees.  An arrow indicates that the target principle Weihrauch reduces to the source principle.  No further arrows may be added, except those that may be inferred by following the arrows drawn.  The single arrows do not reverse, which may be seen by considering the Turing degrees of solutions to recursive instances of the relevant problems.  That $\rt^2_2$ is incomparable with each of $\wkl$, $\lim$, and $\wkl'$ in the Weihrauch degrees is explained in~\cite{BrattkaRakotoniaina}.}
\label{fig-RSg}
\end{figure}
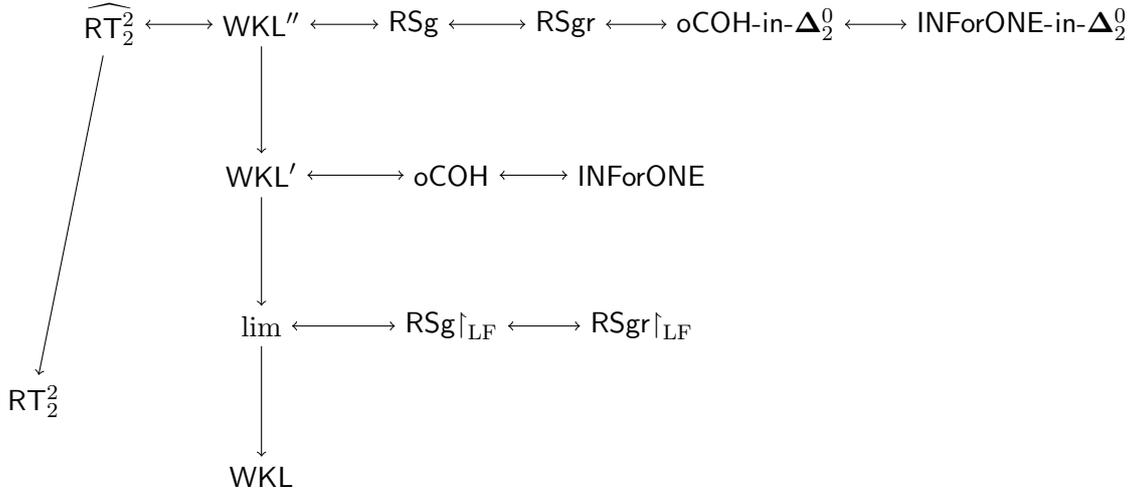

\section{The weak Rival--Sands theorem in the Weihrauch degrees and in the computable degrees}\label{sec-wRSgWeihrauch}

In this section, we compare the Weihrauch degree and the computable degree of the weak Rival--Sands theorem to those of $\rt^2_2$, its consequences, and other familiar benchmarks.  The general theme is that although $\wrsg$ and $\rt^2_2$ are equivalent over $\rca$, $\wrsg$ is much weaker than $\rt^2_2$ in the Weihrauch degrees and in the computable degrees.  This justifies our description of $\wrsg$ as a weaker formalization of $\rt^2_2$.

Multi-valued functions corresponding to the weak Rival--Sands theorem and its refined version are defined as follows.
\begin{Definition}{\ }
\begin{itemize}
\item $\wrsg$ is the following multi-valued function.
\begin{itemize}
\item Input/instance:  An infinite graph $G = (V,E)$.

\smallskip

\item Output/solution:  An infinite $H \subseteq V$ such that, for all $v \in H$, either $|H \cap N(v)| = \omega$ or $|H \cap N(v)| \leq 1$.
\end{itemize}

\medskip

\item $\wrsgr$ is the following multi-valued function.
\begin{itemize}
\item Input/instance:  An infinite graph $G = (V,E)$.

\smallskip

\item Output/solution:  An infinite $H \subseteq V$ such that for all $v \in H$, either $|H \cap N(v)| = \omega$ or $|H \cap N(v)| = 0$.
\end{itemize}
\end{itemize}
\end{Definition}

Clearly $\wrsg \leqsW \wrsgr$ because given a graph $G$, every $\wrsgr$-solution to $G$ is also a $\wrsg$-solution to $G$.  We do not know if this reduction reverses.  That is, we do not know if $\wrsgr \leqW \wrsg$.  We do however show that $\wrsgr \leqc \wrsg$ in Proposition~\ref{prop-WRSGRredWRSG} below.

One may also wish to consider the restrictions $\wrsg \rst_{V = \omega}$ and $\wrsgr \rst_{V = \omega}$ of $\wrsg$ and $\wrsgr$ to input graphs with $V = \omega$.  In the case of $\rsg$, the restriction makes no difference.  In the case of $\wrsg$, the difference between $\wrsg \rst_{V = \omega}$ and $\wrsg$ reflects the difference between $\leqsW$ and $\leqW$.  The problem $\wrsg\rst_{V = \omega}$ is not a cylinder, and $\wrsg$ is its cylindrification.  Thus $\wrsg\rst_{V = \omega} \equivW \wrsg$, but $\wrsg\rst_{V = \omega} \ltsW \wrsg$.  The same comments apply to $\wrsgr$.

\begin{Proposition}\label{prop-wRSgCyl}{\ }
\begin{enumerate}
\item\label{it-notCyl} $\wrsg\rst_{V = \omega}$ and $\wrsgr\rst_{V = \omega}$ are not cylinders.

\smallskip

\item\label{it-cylify} $\wrsg \equivsW \id \times \wrsg\rst_{V = \omega}$ and $\wrsgr \equivsW \id \times \wrsgr\rst_{V = \omega}$.
\end{enumerate}
\end{Proposition}

\begin{proof}
We prove both items for $\wrsg$.  The proofs for $\wrsgr$ are analogous.

For item~(\ref{it-notCyl}), it suffices to observe that every pair of $\wrsg\rst_{V = \omega}$-instances has a common solution.  It follows that $\id \nleqsW \wrsg\rst_{V = \omega}$, so $\wrsg\rst_{V = \omega}$ is not a cylinder.  Let $G_0 = (\omega, E_0)$ and $G_1 = (\omega, E_1)$ be two $\wrsg\rst_{V = \omega}$-instances.  Let $H_0$ be an infinite homogeneous set for $G_0$ (i.e., either an infinite clique or an infinite independent set).  Let $G_1 \rst H_0 = (H_0, E_1 \cap [H_0]^2)$ be the subgraph of $G_1$ induced by $H_0$.  Let $H$ be an infinite homogeneous set for $G_1 \rst H_0$.  Then $H$ is homogeneous for both $G_0$ and $G_1$, so it is a $\wrsg\rst_{V = \omega}$-solution to both $G_0$ and $G_1$.

For item~(\ref{it-cylify}), we first show that $\id \times \wrsg\rst_{V = \omega} \leqsW \wrsg$.  Let $p \in \omega^\omega$, and let $G = (\omega, E)$ be a $\wrsg\rst_{V = \omega}$-instance.  Let $\Phi$ be the functional given by $\Phi(\la p, G \ra) = \wh{G} = (V, \wh{E})$, where $V = \{p \rst n : n \in \omega\}$ and $\wh{E} = \{(p \rst m, p \rst n) : (m, n) \in E\}$.  Let $\wh{H}$ be a $\wrsg$-solution to $\wh{G}$.  Define a functional $\Psi(\wh{H})$ computing a pair $\la q, H \ra$ as follows.  To compute $q$, given $n$ search for a $\sigma \in \wh{H}$ with $|\sigma| > n$ and output $q(n) = \sigma(n)$.  To compute $H$, take $H = \{n : q \rst n \in \wh{H}\}$.  The set $\wh{H}$ consists of infinitely many initial segments of $p$, so in fact we computed $q = p$ and $H = \{n : p \rst n \in \wh{H}\}$.  Furthermore, $H$ is a $\wrsg\rst_{V = \omega}$-solution to $G$ because the function $n \mapsto p \rst n$ is an isomorphism between $G$ and $\wh{G}$.  Thus $\Phi$ and $\Psi$ witness that $\id \times \wrsg\rst_{V = \omega} \leqsW \wrsg$.

Now we show that $\wrsg \leqsW \id \times \wrsg\rst_{V = \omega}$.  Let $G = (V,E)$ be a $\wrsg$-instance.  Let $\Phi$ be the functional given by $\Phi(G) = \la p, \wh{G} \ra$, where $p \colon \omega \imp V$ enumerates $V$ in increasing order, and $\wh{G} = (\omega, \wh{E})$ is the graph with $\wh{E} = \{(m, n) : (p(m), p(n)) \in E\}$.  Then $\la p, \wh{G} \ra$ is a $(\id \times \wrsg\rst_{V = \omega})$-instance.  Let $\la p, \wh{H} \ra$ be a $(\id \times \wrsg\rst_{V = \omega})$-solution.  Define a functional $\Psi(\la p, \wh{H} \ra)$ computing the set $H = \{v : p^{-1}(v) \in \wh{H}\}$.  Then $H$ is a $\wrsg$-solution to $G$ because $p$ is an isomorphism between $\wh{G}$ and $G$.  Thus $\Phi$ and $\Psi$ witness that $\wrsg \leqsW \id \times \wrsg\rst_{V = \omega}$.
\end{proof}

Many of the arguments in the rest of this section are based on the observations made in the following lemma.

\begin{Lemma}\label{lem-wRSgObs}
Let $G = (V, E)$ be an infinite graph.
\begin{enumerate}
\item\label{it-CompInd} If $K \subseteq V$ is an infinite set such that $|K \cap N(x)| < \omega$ for every $x \in K$, then $G \oplus K$ computes an infinite independent set $C \subseteq K$.

\medskip

\item\label{it-FiniteNbd} Let $F = \{x \in V : |N(x)| < \omega\}$.  
\begin{enumerate}
\item\label{it-FiniteNbdFin} If $F$ is finite, then $V \setminus F \leqT G$ is a $\wrsgr$-solution to $G$.

\smallskip

\item\label{it-FiniteNbdInf} If $F$ is infinite, then $G$ has an infinite independent set $C \leqT G'$.
\end{enumerate}

\medskip

\item\label{it-NoRecSoln} Assume that no $H \leqT G$ is a $\wrsgr$-solution to $G$.
\begin{enumerate}
\item\label{it-NoRecSolnInf} Then $G$ has an infinite independent set.

\smallskip

\item\label{it-NoRecSolnFin} Let $D$ be a finite independent set, and let $\sigma \in 2^{<\omega}$ be a characteristic string of $D$:  $|\sigma| > \max(D)$ and $(\forall n < |\sigma|)(\sigma(n) = 1 \biimp n \in D)$.  Then $\sigma$ extends to the characteristic function of a $\wrsgr$-solution to $G$.
\end{enumerate}
\end{enumerate}
\end{Lemma}

\begin{proof}
(\ref{it-CompInd}):  Suppose that $K$ is infinite and that $|K \cap N(x)| < \omega$ for every $x \in K$.  To compute an infinite independent set $C = \{x_0, x_1, \dots\} \subseteq K$ from $G \oplus K$, let $x_0$ be the first element of $K$, and let $x_{n+1}$ be the first element of $K$ that is $> x_n$ and not adjacent to any of $\{x_0, \dots, x_n\}$.

(\ref{it-FiniteNbd}):  Let $F = \{x \in V : |N(x)| < \omega\}$.  If $F$ is finite, then $I = V \setminus F$ is infinite, $I \leqT G$, and $|I \cap N(x)| = \omega$ for every $x \in I$.  Thus $I \leqT G$ is a $\wrsgr$-solution to $G$.  Suppose instead that $F$ is infinite.  Then there is an infinite $F_0 \subseteq F$ with $F_0 \leqT G'$ because $F$ is r.e.\ relative to $G'$.  $F_0$ satisfies $|F_0 \cap N(x)| < \omega$ for every $x \in F_0$, so there is an infinite independent set $C \leqT G \oplus F_0 \leqT G'$ by~(\ref{it-CompInd}) with $K = F_0$.

(\ref{it-NoRecSoln}):  Assume that no $H \leqT G$ is a $\wrsgr$-solution to $G$.  For (\ref{it-NoRecSolnInf}), if $G$ has no infinite independent set, then there would be a $\wrsgr$-solution $H \leqT G$ by~(\ref{it-FiniteNbd}).  For~(\ref{it-NoRecSolnFin}), let $\sigma \in 2^{<\omega}$ be a characteristic string of a finite independent set $D$.  Again, let $F = \{x \in V : |N(x)| < \omega\}$ and let $I = V \setminus F$.  If $I$ is finite, then $F$ is infinite, $F \leqT G$, and, by definition, $|F \cap N(x)| < \omega$ for every $x \in F$.  Thus by~(\ref{it-CompInd}), there is an infinite independent $C \leqT G \oplus F \equivT G$.  This contradicts that no $H \leqT G$ is a $\wrsgr$-solution to $G$.  (In this case, one may alternatively show that $\sigma$ extends to a $\wrsgr$-solution to $G$.)

Now suppose that $I$ is infinite.  Further suppose that there is an $x \in I$ with $|I \cap N(x)| < \omega$.  That is, $x$ has infinitely many neighbors, but only finitely many neighbors of $x$ have infinitely many neighbors.  In this case, let $K = N(x) \setminus I$.  Then $K$ is infinite and $|K \cap N(y)| < \omega$ for every $y \in K$.  Furthermore, $K \leqT G$ because $|I \cap N(x)| < \omega$.  Thus by~(\ref{it-CompInd}), there is an infinite independent $C \leqT G \oplus K \leqT G$.  This again contradicts that no $H \leqT G$ is a $\wrsgr$-solution to $G$.

Finally, suppose that $I$ is infinite and that $|I \cap N(x)| = \omega$ for every $x \in I$.  Let $n$ be greater than $|\sigma|$ and the maximum element of $\bigcup_{v \in D \cap F}N(v)$.  Let $H = D \cup \{x \in I : x > n\}$.  It is clear that $\sigma \subseteq H$.  To see that $H$ is a $\wrsgr$-solution to $G$, consider a $v \in H$.  Either $v \in D \cap F$ or $v \in I$.  If $v \in D \cap F$, then $|D \cap N(v)| = 0$ because $D$ is independent, and $|\{x \in I : x > n\} \cap N(v)| = 0$ by the choice of $n$.  Hence $|H \cap N(v)| = 0$.  If $v \in I$, then $|I \cap N(v)| = \omega$ by assumption, and therefore also $|\{x \in I : x > n\} \cap N(v)| = \omega$.  So $|H \cap N(v)| = \omega$.  Thus $H$ is a $\wrsgr$-solution to $G$.
\end{proof}

First, we show that $\wrsgr \leqc \wrsg$.

\begin{Proposition}\label{prop-WRSGRredWRSG}
$\wrsgr \leqc \wrsg$.  Hence $\wrsg \equivc \wrsgr$.
\end{Proposition}

\begin{proof}
Let $G = (V,E)$ be a $\wrsgr$-instance.  Then $G$ is also a $\wrsg$-instance, so let $H$ be a $\wrsg$-solution to $G$.  We show that there is a $\wrsgr$-solution $\wh{H}$ to $G$ with $\wh{H} \leqT G \oplus H$.

Let $I = \{x \in H : |H \cap N(x)| = \omega\}$.  Notice that also $I = \{x \in H : |H \cap N(x)| \geq 2\}$ because $H$ is a $\wrsg$-solution to $G$.  Therefore $I$ is r.e.\ relative to $G \oplus H$.  Now consider three cases.

Case $1$:  The set $I$ is finite.  Let $K = H \setminus I$.  Then $K$ is infinite, $K \equivT H$, and $|K \cap N(x)| < \omega$ for every $x \in K$.  Thus by Lemma~\ref{lem-wRSgObs} item~(\ref{it-CompInd}), there is an infinite independent $\wh{H} \leqT G \oplus K \equivT G \oplus H$, which is a $\wrsgr$-solution to $G$.

Case $2$:  There is a $v \in I$ with $|I \cap N(v)| < \omega$.  Let $K = (H \cap N(v)) \setminus I$.  Then $K$ is infinite and $K \leqT G \oplus H$ because $H \cap N(v)$ is infinite, $H \cap N(v) \leqT G \oplus H$, and $I \cap N(v)$ is finite.  Furthermore, $|K \cap N(x)| < \omega$ for every $x \in K$.  Thus by Lemma~\ref{lem-wRSgObs} item~(\ref{it-CompInd}), there is an infinite independent $\wh{H} \leqT G \oplus K \leqT G \oplus H$, which is a $\wrsgr$-solution to $G$.

Case $3$:  $I$ is infinite and $|I \cap N(v)| = \omega$ for every $v \in I$.  In this case we compute a set $\wh{H} \leqT G \oplus H$ with $\wh{H} \subseteq I$ and $|\wh{H} \cap N(x)| = \omega$ for each $x \in \wh{H}$.  This $\wh{H}$ is thus a $\wrsgr$-solution to $G$.  To compute $\wh{H} = \{x_0, x_1, \dots\}$, let $x_0$ be the first element of $I$.  To find $x_{n+1}$, decompose $n$ as $n = \la m, s \ra$, search for a $y \in I \cap N(x_m)$ with $y > x_n$, and set $x_{n+1} = y$.  Such a $y$ exists because $x_m \in I$ and every element of $I$ has infinitely many neighbors in $I$.  The search for $y$ can be done effectively relative to $G \oplus H$ because $I$ is r.e.\ relative to $G \oplus H$.  Finally, $|\wh{H} \cap N(x)| = \omega$ for each $x \in \wh{H}$ because $x_{n+1}$ is adjacent to $x_m$ whenever $n$ is of the form $\la m, s \ra$.
\end{proof}

\begin{Question}
Does $\wrsgr \leqW \wrsg$ hold?
\end{Question}

We may situate $\wrsg$ in the computable degrees by combining Lemma~\ref{lem-wRSgObs} and the proof of Theorem~\ref{thm-RT_WRSG_eq} with established results concerning $\rt^2_2$ and its consequences.  Recall that the \emph{chain/anti-chain principle} ($\cac$) of~\cite{HirschfeldtShore} states that every infinite partial order contains either an infinite chain or an infinite anti-chain.

\begin{Proposition}\label{prop-wRSgCompLoc}
In the computable degrees, $\wrsg$ is
\begin{itemize}
\item strictly below $\rt^2_2$ and $\lim$;

\smallskip

\item strictly above $\ads$ and $\srt^2_2$;

\smallskip

\item incomparable with $\cac$.
\end{itemize}
\end{Proposition}

\begin{proof}
Trivially $\wrsg \leqsW \rt^2_2$, hence $\wrsg \leqc \rt^2_2$.  That $\rt^2_2 \nleqc \wrsg$ is because every $\wrsg$-instance $G$ has a solution $H \leqT G'$ by Lemma~\ref{lem-wRSgObs} item~(\ref{it-FiniteNbd}), whereas by~\cite{JockuschRamsey}*{Theorem~3.1} there are recursive $\rt^2_2$-instances with no solution recursive in $0'$.

As mentioned in the proof of Proposition~\ref{prop-RSgLF}, $\lim$ is strongly Weihrauch equivalent, hence computably equivalent, to the Turing jump function $\mathsf{J}$.  Every $\wrsg$-instance $G$ has a solution $H \leqT G'$ by Lemma~\ref{lem-wRSgObs} item~(\ref{it-FiniteNbd}), so $\wrsg \leqc \lim$.  That $\lim \nleqc \wrsg$ follows from the cone-avoidance result for $\rt^2_2$:  by~\cite{SeetapunSlaman}*{Theorem~2.1}, every recursive infinite graph has a homogeneous set, hence $\wrsg$-solution, that does not compute $0'$.

For $\ads \leqc \wrsg$ and $\srt^2_2 \leqc \wrsg$, see the proof of the $\rca \vdash \wrsg \imp \rt^2_2$ direction of Theorem~\ref{thm-RT_WRSG_eq}.  The argument showing that $\rca \vdash \wrsg \imp \ads$ describes a computable reduction from $\ads$ to $\wrsg$.  The argument showing that $\rca \vdash \wrsg \imp \Dd^2_2$ also describes a computable reduction from the multi-valued function corresponding to $\Dd^2_2$ (which we did not explicitly define) to $\wrsg$.  Now recall Proposition~\ref{prop-rt22decomp} item~(\ref{it-SRT22D22Equiv}), which states that $\rca \vdash \srt^2_2 \biimp \Dd^2_2$.  The proof of this fact (see~\cite{CholakJockuschSlaman}*{Lemma~7.10}) describes computable reductions in both directions, so we have that $\srt^2_2 \equivc \Dd^2_2$ and therefore that $\srt^2_2 \leqc \wrsg$.  For the non-reductions, by the results of~\cite{HirschfeldtShore}*{Section~2}, there are $\omega$-models of $\ads$ that are not models of $\rt^2_2$ and therefore not models of $\wrsg$.  Hence $\wrsg \nleqc \ads$.  By impressive recent work of Monin and Patey~\cite{MoninPatey}, there are also $\omega$-models of $\srt^2_2$ that are not models of $\rt^2_2$ and therefore not models of $\wrsg$.  Hence $\wrsg \nleqc \srt^2_2$.

That $\cac \nleqc \wrsg$ is because again every $\wrsg$-instance $G$ has a solution $H \leqT G'$, whereas by~\cite{Herrmann}*{Theorem~3.1} there are recursive $\cac$-instances with no solution recursive in $0'$.  That $\wrsg \nleqc \cac$ follows from fact that there are $\omega$-models of $\cac$ that are not models of $\rt^2_2$ and therefore not models of $\wrsg$, as shown in~\cite{HirschfeldtShore}*{Section~3}. 
\end{proof}

We remark that Proposition~\ref{prop-wRSgCompLoc} implies that $\coh \lstc \wrsg$ as well because $\coh \leqc \ads$ (by Proposition~\ref{prop-CADSCOH}, for example).

We return to the Weihrauch degrees and first show that $\sadc \nleqW \wrsgr$.  As $\sadc$ is below both $\ads$ and $\srt^2_2$ in the Weihrauch degrees (see~\cite{AstorDzhafarovSolomonSuggs}, for example), this implies that the computable reductions $\ads \lstc \wrsgr$ and $\srt^2_2 \lstc \wrsgr$ cannot be improved to Weihrauch reductions.  We also show that $\pdnr \nleqW \wrsgr$.

\begin{Theorem}\label{thm-SADCvWRSGR}
$\sadc \nleqW \wrsgr$.
\end{Theorem}

\begin{proof}
Suppose for a contradiction that $\sadc \leqW \wrsgr$ is witnessed by Turing functionals $\Phi$ and $\Psi$.  By a well-known result independently of Tennenbaum and Denisov (see~\cite{RosensteinBook}*{Theorem~16.54}, for example), there is a recursive linear order $L = (\omega, <_L)$ with $L \cong \omega + \omega^*$ that has no infinite recursive ascending or descending sequence.  If $\ell \in L$ has finitely many $<_L$-predecessors, then say that $\ell$ is in the $\omega$-part of $L$; and if $\ell$ has finitely many $<_L$-successors, then say that $\ell$ is in the $\omega^*$-part of $L$.  Notice that no infinite r.e.\ set is contained entirely in the $\omega$-part of $L$, as such a set could be thinned to an infinite recursive ascending sequence.  Similarly, no infinite r.e.\ set is contained entirely in the $\omega^*$-part of $L$.

The linear order $L$ is a recursive $\sadc$-instance, so $G = \Phi(L)$ is a recursive $\wrsgr$-instance.  Write $G = (V, E)$.  $G$ cannot have a recursive $\wrsgr$-solution because if there were a recursive solution $H$ to $G$, then $\Psi(\la L, H \ra)$ would be a recursive $\sadc$-solution to $L$, which would be an infinite recursive set either entirely contained in the $\omega$-part of $L$ or entirely contained in the $\omega^*$-part of $L$.  Therefore $G$ has an infinite independent set $C$ by Lemma~\ref{lem-wRSgObs} item~(\ref{it-NoRecSolnInf}).  This $C$ is a $\wrsgr$-solution to $G$, so $\Psi(\la L, C \ra)$ is a $\sadc$-solution to $L$.  In particular, $\Psi(\la L, C \ra)$ is infinite.  Fix any $x \in \Psi(\la L, C \ra)$, and assume for the sake of argument that $x$ is in the $\omega$-part of $L$ (the $\omega^*$-part case is symmetric).  Let $R$ be the r.e.\ set
\begin{align*}
R = \bigl\{y : \text{there is a finite independent set $D \subseteq V$ with $x, y \in \Psi(\la L, D \ra)$}\bigr\}.
\end{align*}
Notice that if $y \in \Psi(\la L, C \ra)$, then any sufficiently long initial segment $D$ of $C$ witnesses that $y \in R$.  Thus $\Psi(\la L, C \ra) \subseteq R$.  In particular, $R$ is infinite.  However, $R$ is r.e., so it cannot be entirely contained in the $\omega$-part of $L$.  Therefore there must be a $y \in R$ that is in the $\omega^*$-part of $L$.  Let $D$ be a finite independent set witnessing that $y \in R$.  By Lemma~\ref{lem-wRSgObs} item~(\ref{it-NoRecSolnFin}), the characteristic string of $D$ extends to the characteristic function of a $\wrsgr$-solution $H$ to $G$.  However, $x, y \in \Psi(\la L, H \ra)$, $x$ is in the $\omega$-part of $L$, and $y$ is in the $\omega^*$-part of $L$.  Thus $\Psi(\la L, H \ra)$ can be neither an infinite ascending chain nor an infinite descending chain.  Thus $\Phi$ and $\Psi$ do not witness that $\sadc \leqW \wrsgr$, so $\sadc \nleqW \wrsgr$.
\end{proof}

\begin{Theorem}\label{thm-DNRvWRSGR}
$\pdnr \nleqW \wrsgr$.
\end{Theorem}

\begin{proof}
The proof is similar to the proof of Theorem~\ref{thm-SADCvWRSGR}.  Suppose for a contradiction that $\pdnr \leqW \wrsgr$ is witnessed by Turing functionals $\Phi$ and $\Psi$.  Let $p \colon \omega \imp \omega$ be any recursive function.  Then $p$ is a recursive $\pdnr$-instance, so $G = \Phi(p)$ is a recursive $\wrsgr$-instance.  Write $G = (V, E)$.  $G$ cannot have a recursive $\wrsgr$-solution because if there were a recursive solution $H$ to $G$, then $\Psi(\la p, H \ra)$ would be a contradictory recursive $\pdnr$-solution to $p$.  Thus $G$ has an infinite independent set $C$ by Lemma~\ref{lem-wRSgObs} item~(\ref{it-NoRecSolnInf}).  This $C$ is a $\wrsgr$-solution to $G$, so $\Psi(\la p, C \ra)$ is $\tdnr$ relative to $p$.

Compute a function $g \colon \omega \imp \omega$ as follows.  On input $e$, $g(e)$ searches for a finite independent set $D \subseteq V$ such that $\Psi(\la p, D \ra)(e)\da$ and outputs the value of $\Psi(\la p, D \ra)(e)$ for the first such $D$ found.  The function $g$ is total because $\Psi(\la p, C \ra)$ is total:  for any $e$, any sufficiently long initial segment $D$ of $C$ is a finite independent set for which $\Psi(\la p, D \ra)(e)\da$.  The function $g$ is recursive, so it is not $\tdnr$ relative to $p$.  So there is an $e$ such that $g(e) = \Phi_e(p)(e)$.  By the definition of $g$, there is a finite independent set $D$ such that $\Psi(\la p, D \ra)(e) = g(e) = \Phi_e(p)(e)$.  By Lemma~\ref{lem-wRSgObs} item~(\ref{it-NoRecSolnFin}), the characteristic string of $D$ extends to the characteristic function of a $\wrsgr$-solution $H$ to $G$.  Then $\Psi(\la p, H \ra)(e) = \Phi_e(p)(e)$, so $\Psi(\la p, H \ra)$ is not a $\pdnr$-solution to $p$.  Thus $\Phi$ and $\Psi$ do not witness that $\pdnr \leqW \wrsgr$, so $\pdnr \nleqW \wrsgr$.
 \end{proof}

On the positive side, we show that $\coh \leqsW \wrsg$ and that $\rt^1_{<\infty} \leqsW \wrsg$.

\begin{Theorem}\label{thm-COHredWRSG}
$\coh \leqsW \wrsg$.
\end{Theorem}

\begin{proof}
It suffices to show that $\cads \leqW \wrsg$ because $\cads \equivW \coh$ by Proposition~\ref{prop-CADSCOH} and because $\wrsg$ is a cylinder by Proposition~\ref{prop-wRSgCyl}.

Let $L = (L, <_L)$ be a $\cads$-instance.  Define a functional $\Phi(L)$ computing the graph $G = (V, E)$ where $V = L$ and
\begin{align*}
E = \{(m, n) : (m,n \in V) \andd (m < n) \andd (m <_L n)\}.
\end{align*}
The graph $G$ is a valid $\wrsg$-instance, so let $H$ be a $\wrsg$-solution to $G$.  We define a functional $\Psi(\la L, H \ra)$ computing a set $C \subseteq L$ which will be a suborder of $L$ either of type $\omega^*$, of type $1 + \omega^*$, or of type $\omega + k$ for some finite linear order $k$.

Using $\Phi$, we may compute $\Phi(L) = G$.  Using $G$ and $H$, we may enumerate the set $R = \{x \in H : |H \cap N(x)| \geq 2\}$.  We claim that if $|R| \geq 2$, then every $x \in R$ has infinitely many $<_L$-successors in $R$.  To see this, suppose that $|R| \geq 2$, let $x \in R$, and let $z \in R$ be different from $x$.  Then $|H \cap N(x)| \geq 2$ and $|H \cap N(z)| \geq 2$.  Therefore $|H \cap N(x)| = \omega$ and $|H \cap N(z)| = \omega$ because $H$ is a $\wrsg$-solution to $G$.  Let $w$ denote the $<_L$-maximum of $x$ and $z$.  Then any sufficiently large $y \in H \cap N(w)$ satisfies $y > x$, $y > z$, $y >_L x$, and $y >_L z$.  Thus any such $y$ is in $R$ because $y \in H$ and $x, z \in H \cap N(y)$.  Therefore there are infinitely many $y \in R$ with $y >_L x$.

To compute $C = \{x_0, x_1, \dots\}$, first enumerate $H$ in $<$-increasing order as $h_0 < h_1 < h_2 < \cdots$.  For each $s$, let $H_s = \{h_0, \dots, h_s\}$.  Take $x_n = h_n$ until possibly reaching an $s_0$ for which there are distinct $u, v \in H_{s_0}$ with $|H_{s_0} \cap N(u)| \geq 2$ and $|H_{s_0} \cap N(v)| \geq 2$.  If such an $s_0$ is reached, then $H_{s_0}$ witnesses that $u,v \in R$.  Thus $R$ is infinite by the claim, so we may switch to computing an ascending sequence in $R$.  Search for a $y \in R$ with $y > x_{s_0 - 1}$ and set $x_{s_0} = y$.  Having determined $x_s$ for some $s \geq s_0$, search for a $y \in R$ with $y > x_s$ and $y >_L x_s$, which exists by the claim, and set $x_{s+1} = y$.

We now show that $C$ is a suborder of $L$ either of type $\omega^*$, of type $1 + \omega^*$, or of type $\omega + k$ for some finite linear order $k$.  First suppose that there is an $s_0$ for which there are distinct $u, v \in H_{s_0}$ with $|H_{s_0} \cap N(u)| \geq 2$ and $|H_{s_0} \cap N(v)| \geq 2$.  Then $\{x_n : n \geq s_0\}$ is an ascending sequence in $L$, so $C$ is a suborder of $L$ of type $\omega + k$ for some finite linear order $k$.  If there is no such $s_0$, then $C = H$, which in this case is a suborder of $L$ either of type $\omega^*$ or of type $1 + \omega^*$.  To see this, suppose for a contradiction that there are $a < b$ such that both $h_a$ and $h_b$ have infinitely many $<_L$-successors in $H$.  Then there are infinitely many $n$ with $h_n >_L h_a, h_b$.  In particular, there are $n > m > b$ with $h_m >_L h_a, h_b$ and $h_n >_L h_a, h_b$.  But then $h_a, h_b \in H_n$; $h_a, h_b \in N(h_m)$; and $h_a, h_b \in N(h_n)$.  So for $s_0 = n$ there are $u = h_a$ and $v = h_b$ with $|H_{s_0} \cap N(u)| \geq 2$ and $|H_{s_0} \cap N(v)| \geq 2$, contradicting that there is no such $s_0$.
\end{proof}

The proof that $\rt^1_{<\infty} \leqsW \wrsg$ is similar to Hirst's proof that $\rca + \rt^2_2 \vdash \rt^1_{<\infty}$ from~\cite{HirstThesis}.

\begin{Proposition}\label{prop-PHPredWRSG}
$\rt^1_{<\infty} \leqsW \wrsg$.
\end{Proposition}

\begin{proof}
It suffices to show that $\rt^1_{<\infty} \leqW \wrsg$ because $\wrsg$ is a cylinder by Proposition~\ref{prop-wRSgCyl}.  Let $c$ be an $\rt^1_{<\infty}$-instance.  Define a functional $\Phi(c)$ computing the graph $G = (\omega, E)$ where $E = \{(m,n) : c(m) = c(n)\}$.  The graph $G$ is a valid $\wrsg$-instance, so let $H$ be a $\wrsg$-solution to $G$.  Let $\Psi(\la c, H \ra)$ be a functional that computes $G = \Phi(c)$, searches for an $x \in H$ with $|H \cap N(x)| \geq 2$, and outputs the set $H \cap N(x)$ for the first such $x$ found.  There must be such an $x$ because $G$ is a disjoint union of finitely many complete graphs (depending on the size of $\ran(c)$), and thus $H$ must have infinite intersection with one of these components.  The set $H \cap N(x)$ is infinite because $H$ is a $\wrsg$-solution to $G$, and it is monochromatic because $c(y) = c(x)$ for all $y \in H \cap N(x)$.
\end{proof}

We are ready to summarize the position of $\wrsg$ and $\wrsgr$ in the Weihrauch degrees.  Notice that the uniform computational content of $\wrsg$ and $\wrsgr$ is considerably less than that of $\rt^2_2$:  $\rt^2_2$ is above both $\pdnr$ and $\sadc$ in the Weihrauch degrees, but $\wrsgr$ is above neither of these problems.

\begin{Theorem}\label{thm-wRSgWeiLoc}
In the Weihrauch degrees, $\wrsg$ and $\wrsgr$ are
\begin{itemize}
\item strictly below $\rt^2_2$;

\smallskip

\item strictly above $\coh$ and $\rt^1_{<\infty}$;

\smallskip

\item incomparable with $\lim$, $\srt^2_2$, $\sadc$, and $\pdnr$.
\end{itemize}
\end{Theorem}

\begin{proof}
Trivially $\wrsgr \leqsW \rt^2_2$.  That $\rt^2_2 \nleqW \wrsgr$ follows from the stronger non-reduction $\rt^2_2 \nleqc \wrsgr$ of Proposition~\ref{prop-wRSgCompLoc}.

We have that $\coh \leqsW \wrsg$ and that $\rt^1_{<\infty} \leqsW \wrsg$ by Theorem~\ref{thm-COHredWRSG} and Proposition~\ref{prop-PHPredWRSG}.  These reductions are strict (indeed, the corresponding computable reductions are strict) because there are $\omega$-models of $\coh$ that are not models of $\rt^2_2$, hence not models of $\wrsg$, by the results of~\cite{HirschfeldtShore}*{Section~2}, for example; and because every recursive $\rt^1_{<\infty}$-instance has a recursive solution.

We now show the incomparabilities.  Straightforward arguments show that $\sadc \leqsW \srt^2_2$ and that $\pdnr \leqsW \lim$, so it suffices to show that $\wrsgr$ is above neither $\sadc$ nor $\pdnr$ and that $\wrsg$ is below neither $\srt^2_2$ nor $\lim$.  Theorems~\ref{thm-SADCvWRSGR} and~\ref{thm-DNRvWRSGR} give $\sadc \nleqW \wrsgr$ and $\pdnr \nleqW \wrsgr$.  We have that $\wrsg \nleqW \srt^2_2$ because $\coh \leqsW \wrsg$ as mentioned above, but $\coh \nleqW \srt^2_2$ by~\cite{DzhafarovStrongRed}*{Corollary~4.5}.  Finally, $\wrsg \nleqW \lim$ because $\rt^1_{<\infty} \leqsW \wrsg$ as mentioned above, but $\rt^1_{<\infty} \nleqW \lim$ by~\cite{BrattkaRakotoniaina}*{Corollary~4.20}.
\end{proof}

From Proposition~\ref{prop-PHPredWRSG} and Theorem~\ref{thm-wRSgWeiLoc}, one may deduce that $\wrsg$ and $\wrsgr$ are Weihrauch incomparable with a number of other principles, such as $\ads$, $\cac$, and its stable version $\scac$.  Figure~\ref{fig-wRSg} depicts the position of $\wrsg$ and $\wrsgr$ relative to a number of principles below $\rt^2_2$ in the Weihrauch degrees.

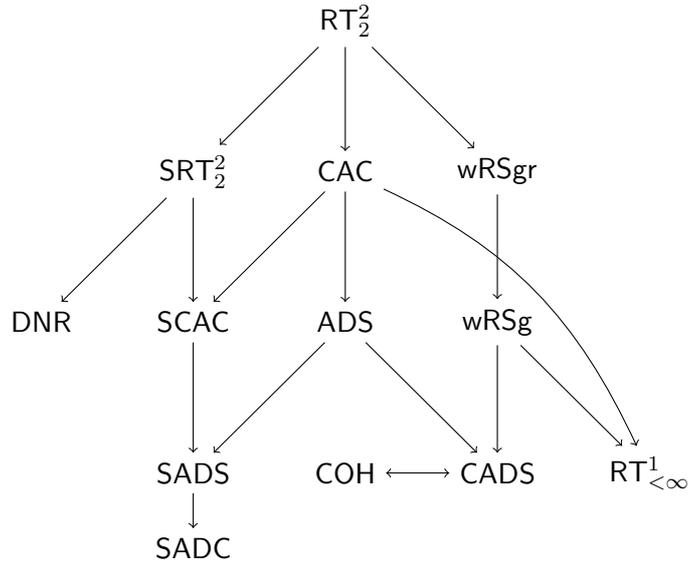
\begin{figure}[b]
\begin{tikzpicture}
\node (coh) at (-2,-4)  {$\coh$};

\node (DNR) at (-6,-2) {$\pdnr$};

\node (wrsgr) at (0,0)  {$\wrsgr$};
\node (wrsg) at (0,-2)  {$\wrsg$};

\node (RT22) at (-2,2) {$\rt^2_2$};

\node (CAC) at (-2,0) {$\cac$};
\node (ADS) at (-2,-2) {$\ads$};
\node (cads) at (0,-4) {$\cads$};

\node (SRT22) at (-4,0) {$\srt^2_2$};
\node (SCAC) at (-4,-2) {$\scac$};
\node (SADS) at (-4,-4) {$\sads$};
\node (SADC) at (-4,-5) {$\sadc$};

\node (RT1) at (2,-4) {$\rt^1_{<\infty}$};

\draw[->] (SRT22) --  (DNR);

\draw[->] (RT22) -- (wrsgr);
\draw[->] (RT22) --  (SRT22);
\draw[->] (RT22) --  (CAC);
\draw[->] (SRT22) --  (SCAC);
\draw[->] (CAC) --  (SCAC);
\draw[->] (CAC) --  (ADS);
\draw[->] (SCAC) --   (SADS);
\draw[->] (ADS) --   (cads);
\draw[->] (ADS) --   (SADS);
\draw[->] (SADS) --   (SADC);

\draw[->] (wrsgr) --(wrsg);
\draw[->] (wrsg) --   (cads);
\draw[<->] (coh) --   (cads);

\draw[->] (wrsg) -- (RT1);
\draw[->] (CAC) to [out=-25,in=115] (RT1);
\end{tikzpicture}
\caption{Weihrauch reductions and non-reductions in the neighborhood of $\rt^2_2$.  An arrow indicates that the target principle Weihrauch reduces to the source principle.  No further arrows may be added, except those that may be inferred by following the arrows drawn.  No arrows reverse, except the double arrow indicating that $\coh \equivW \cads$ and possibly the arrow indicating that $\wrsg \leqW \wrsgr$.  The reductions and non-reductions (often in the form of $\omega$-model separations) not proved here may be found in~\cites{AstorDzhafarovSolomonSuggs, BrattkaRakotoniaina, DzhafarovStrongRed, HJKHLS, HirschfeldtJockusch, HirschfeldtShore, LermanSolomonTowsner, PateyImmunity}.}
\label{fig-wRSg}
\end{figure}

As $\rca + \wrsg \vdash \rt^2_2$ but $\rt^2_2 \nleqW \wrsg$, it is natural to ask what must be added to $\wrsg$ to obtain $\rt^2_2$ in the Weihrauch degrees.  In particular, we ask how many applications of $\wrsg$ are necessary to obtain $\rt^2_2$.  This question can be formalized by considering \emph{compositional products}.

\begin{Theorem}[see~\cites{BrattkaGherardiMarcone, BrattkaPaulyAlg}]
Let $\mc{F}$ and $\mc{G}$ be multi-valued functions.  The set
\begin{align*}
\{\mc{F}_0 \circ \mc{G}_0 : (\mc{F}_0 \leqW \mc{F}) \andd (\mc{G}_0 \leqW \mc{G}) \andd (\text{$\mc{F}_0$ and $\mc{G}_0$ are composable})\}
\end{align*}
has a $\leqW$-maximum element up to Weihrauch degree.  In a slight abuse of terminology and notation, this maximum element is called the \emph{compositional product} of $\mc{F}$ and $\mc{G}$ and is denoted $\mc{F} \ast \mc{G}$.
\end{Theorem}

To show that $\mc{H} \leqW \mc{F} \ast \mc{G}$ for multi-valued functions $\mc{F}, \mc{G}, \mc{H} \colon \wsubseteq \omega^\omega \rra \omega^\omega$, it suffices to exhibit Turing functionals $\Phi$, $\Theta$, $\Psi$ such that for all $p \in \dom(\mc{H})$,
\begin{itemize}
\item $\Phi(p) \in \dom(\mc{G})$;

\smallskip

\item $\Theta(\la p, q \ra) \in \dom(\mc{F})$ for all $q \in \mc{G}(\Phi(p))$;

\smallskip

\item $\Psi(\la \la p, q \ra, r \ra) \in \mc{H}(p)$ for all $q \in \mc{G}(\Phi(p))$ and all $r \in \mc{F}(\Theta(\la p, q \ra))$.
\end{itemize}

We show that an application of two parallel instances of the \emph{limited principle of omniscience} ($\lpo$) suffices to overcome the non-uniformities in the proof that $\srt^2_2 \leqc \wrsg$, yielding that $\srt^2_2 \leqW (\lpo \times \lpo) \ast \wrsg$.  In the case of $\wrsgr$, one application of $\lpo$ suffices:  $\srt^2_2 \leqW \lpo \ast \wrsgr$.  As $\rt^2_2 \leqW \srt^2_2 \ast \coh$, we conclude that $\rt^2_2 \leqW (\lpo \times \lpo) \ast \wrsg \ast \coh$.  It follows that $\rt^2_2 \leqW \wrsg \ast \wrsg \ast \wrsg$ because below we observe that $(\lpo \times \lpo) \leqW \wrsg$, and $\coh \leqW \wrsg$ by Theorem~\ref{thm-COHredWRSG}.  Thus three applications of $\wrsg$ suffice to obtain $\rt^2_2$.  We do not know if two applications suffice.

A function corresponding to $\lpo$ is defined as follows.
\begin{Definition}
$\lpo$ is the following function.
\begin{itemize}
\item Input/instance:  A function $p \in \omega^\omega$.

\smallskip

\item Output/solution:  Output $0$ if there is an $n$ such that $p(n) = 0$.  Output $1$ if $p(n) \neq 0$ for every $n$.
\end{itemize}
\end{Definition}

\begin{Theorem}\label{thm-2LPOwRsg}{\ }
\begin{enumerate}
\item\label{it-LPOwRSgr} $\srt^2_2 \leqW \lpo \ast \wrsgr$.

\smallskip

\item\label{it-2LPOwRSg} $\srt^2_2 \leqW (\lpo \times \lpo) \ast \wrsg$.
\end{enumerate}
\end{Theorem}

\begin{proof}
For~(\ref{it-LPOwRSgr}), let $c \colon [\omega]^2 \imp \{0,1\}$ be an $\srt^2_2$-instance.  Using $c$, compute the graph $G = (\omega, E)$ with $E = \{(n, s) : (n < s) \andd (c(n,s) = 1)\}$.  Let $H$ be a $\wrsgr$-solution to $G$.  We use an application of $\lpo$ to determine whether or not $H$ contains two adjacent vertices.  Using $G$ and $H$, uniformly compute a function $p \colon \omega \imp \{0,1\}$ by setting $p(n) = 0$ if any two of the least $n$ elements of $H$ are adjacent, and by setting $p(n) = 1$ otherwise.  Let $b = \lpo(p)$.  If $b = 1$, then $H$ is an independent set and hence an $\srt^2_2$-solution to $c$.  Thus output $H$.  If $b = 0$, then $H$ contains a pair of adjacent vertices.  Notice that if $u \in H$ has a neighbor in $H$, then $H \cap N(u)$ is infinite because $H$ is a $\wrsgr$-solution to $G$.  Furthermore, such a $u$ is adjacent to almost every vertex in $G$ because $c$ is stable.  Compute an infinite clique $K = \{x_0, x_1, \dots\}$ uniformly from $G$ and $H$ as follows.  First, search for any $x_0 \in H$ with $|H \cap N(x_0)| \geq 1$.  Having determined a finite clique $\{x_0, \dots, x_n\} \subseteq H$, search for the first vertex $x_{n+1} \in H$ that is adjacent to each $x_i$ for $i \leq n$.  Such an $x_{n+1}$ exists because each $x_i$ for $i \leq n$ is adjacent to almost every vertex of $H$.  The resulting $K$ is an infinite clique and hence an $\srt^2_2$-solution to $c$.

For~(\ref{it-2LPOwRSg}), again let $c \colon [\omega]^2 \imp \{0,1\}$ be an $\srt^2_2$-instance, and again compute the graph $G = (\omega, E)$ with $E = \{(n, s) : (n < s) \andd (c(n,s) = 1)\}$.  Let $H$ be a $\wrsg$-solution to $G$.  Refine $H$ to eliminate \emph{bars}, i.e., pairs of vertices in $H$ where each is the only vertex of $H$ adjacent to the other.  To do this, compute an infinite $\wh{H} \subseteq H$ by skipping the first neighbor of each vertex already added to $\wh{H}$.  Enumerate $H$ in increasing order as $h_0 < h_1 < h_2 < \cdots$.  Let $\wh{H}_0 = \{h_0\}$.  Given $\wh{H}_n$, consider $h_{n+1}$.  If there is a $u \in \wh{H}_n$ such that $h_{n+1}$ is the least element of $H \cap N(u)$, then skip $h_{n+1}$ by putting $\wh{H}_{n+1} = \wh{H}_n$.  Otherwise, put $\wh{H}_{n+1} = \wh{H}_n \cup \{h_{n+1}\}$.  Let $\wh{H} = \bigcup_{n \in \omega}\wh{H}_n$, which can be computed uniformly from $G$ and $H$ because at stage $n$ we determine whether or not $h_n$ is in $\wh{H}$.  If $x, y \in H$ are adjacent to each other but to no other vertices of $H$, then only $\min\{x, y\}$ is in $\wh{H}$.  If $x \in \wh{H}$ has infinitely many neighbors in $H$, then it is adjacent to almost every vertex in $G$ because $c$ is stable, and therefore $x$ also has infinitely many neighbors in $\wh{H}$.

Call a clique of size three a \emph{triangle}.  The set $\wh{H}$ is either an independent set, contains edges but no triangles, or contains triangles.  Using $G$ and $\wh{H}$, uniformly compute two $\lpo$-instances $p, q \colon \omega \imp \{0,1\}$ to determine if $\wh{H}$ contains edges or triangles.  Set $p(n) = 0$ if any two of the least $n$ elements of $\wh{H}$ are adjacent, and set $p(n) = 1$ otherwise.  Set $q(n) = 0$ if any three of the least $n$ elements of $\wh{H}$ form a triangle, and set $q(n) = 1$ otherwise.  Let $(a,b) = (\lpo \times \lpo)(p, q)$.  If $(a,b) = (1,1)$, then $\wh{H}$ contains no edges; if $(a,b) = (0,1)$, then $\wh{H}$ contains edges but not triangles; and if $(a,b) = (0,0)$, then $\wh{H}$ contains triangles.  Output $(1,0)$ is not possible because if $\wh{H}$ contains triangles, then it also contains edges.

If $\wh{H}$ contains no edges, then it is an independent set and hence an $\srt^2_2$-solution to $c$.  Thus output $\wh{H}$.

Suppose that $\wh{H}$ contains edges but not triangles, and suppose that $x, y \in \wh{H}$ are adjacent.  If neither $x$ nor $y$ has any other neighbors in $H$, then only one of them would be in $\wh{H}$.  Therefore either $x$ or $y$ has at least two, and therefore infinitely many, neighbors in $H$.  So either $x$ or $y$ has infinitely many neighbors in $\wh{H}$.  Thus there is a $z \in \wh{H}$ with $\wh{H} \cap N(z)$ infinite.  We can therefore compute an infinite independent set, hence an $\srt^2_2$-solution to $c$, uniformly from $G$ and $\wh{H}$ by searching for a $z \in \wh{H}$ with $|\wh{H} \cap N(z)| \geq 2$ and outputting $\wh{H} \cap N(z)$.  We have just seen that such a $z$ exists.  If $|\wh{H} \cap N(z)| \geq 2$, then $|H \cap N(z)| \geq 2$, so $H \cap N(z)$ is infinite, so $\wh{H} \cap N(z)$ is infinite.  Finally, $\wh{H} \cap N(z)$ is independent because $\wh{H}$ contains no triangles.

If $\wh{H}$ contains a triangle, then $H$ contains a triangle, so there are distinct $x, y \in H$ with $|H \cap N(x)| \geq 2$ and $|H \cap N(y)| \geq 2$.  Then $H \cap N(x)$ and $H \cap N(y)$ are both infinite because $H$ is a $\wrsg$-solution to $G$.  The coloring $c$ is stable, which means that $x$ and $y$ are adjacent to almost every vertex of $G$.  Thus almost every vertex of $H$ is adjacent to both $x$ and $y$, and therefore is adjacent to almost every other vertex of $H$.  Compute an infinite clique $K$ as in~(\ref{it-LPOwRSgr}), except this time start by searching for any distinct $x_0, x_1 \in H$ with $|H \cap N(x_0)| \geq 2$ and $|H \cap N(x_1)| \geq 2$.  The resulting clique $K$ is an $\srt^2_2$-solution to $c$.
\end{proof}

\begin{Corollary}
$\rt^2_2 \leqW (\lpo \times \lpo) \ast \wrsg \ast \coh$.  Therefore $\rt^2_2 \leqW \wrsg \ast \wrsg \ast \wrsg$.
\end{Corollary}

\begin{proof}
We have that $\rt^2_2 \leqW \srt^2_2 \ast \coh$, and $\srt^2_2 \leqW (\lpo \times \lpo) \ast \wrsg$ by Theorem~\ref{thm-2LPOwRsg}.  Therefore $\rt^2_2 \leqW (\lpo \times \lpo) \ast \wrsg \ast \coh$.  That $\rt^2_2 \leqW \wrsg \ast \wrsg \ast \wrsg$ follows because $\lpo \times \lpo \leqW \wrsg$ and $\coh \leqW \wrsg$.  Theorem~\ref{thm-COHredWRSG} gives us $\coh \leqW \wrsg$.  It is straightforward to show that $\lpo \leqW \rt^1_2$ and that $\rt^1_2 \times \rt^1_2 \leqW \rt^1_4$ (see also~\cite{DoraisDzhafarovHirstMiletiShafer}*{Proposition~2.1}).  Therefore
\begin{align*}
\lpo \times \lpo \leqW \rt^1_2 \times \rt^1_2 \leqW \rt^1_4 \leqW \rt^1_{<\infty} \leqW \wrsg,
\end{align*}
where the last reduction is by Proposition~\ref{prop-PHPredWRSG}.
\end{proof}

Hence three applications of $\wrsg$ (or of $\wrsgr$) suffice to obtain $\rt^2_2$.  We do not know if two applications suffice.

\begin{Question}
Does $\rt^2_2 \leqW \wrsg \ast \wrsg$ hold?  Does $\rt^2_2 \leqW \wrsgr \ast \wrsgr$ hold?
\end{Question}

\section*{Acknowledgments}
We thank Jeffry Hirst, Steffen Lempp, Alberto Marcone, Ludovic Patey, and Keita Yokoyama for helpful discussions.  We thank our anonymous reviewer for many helpful comments.  We thank the \emph{Workshop on Ramsey Theory and Computability} at the University of Notre Dame Rome Global Gateway; Dagstuhl Seminar 18361:  \emph{Measuring the Complexity of Computational Content:  From Combinatorial Problems to Analysis}; and BIRS/CMO workshop 19w5111: \emph{Reverse Mathematics of Combinatorial Principles} for their generous support.  This project was partially supported by a grant from the John Templeton Foundation (\emph{A new dawn of intuitionism: mathematical and philosophical advances} ID 60842).  The opinions expressed in this work are those of the authors and do not necessarily reflect the views of the John Templeton Foundation.  Additionally, Fiori-Carones was partially supported by the Italian PRIN 2017 grant \emph{Mathematical Logic: models, sets, computability}.

\bibliography{FioriCaronesShaferSoldaInsideOutsideRamseyRecursionTheory}

\begin{bibdiv}
\begin{biblist}

\bib{AstorDzhafarovSolomonSuggs}{article}{
      author={Astor, Eric~P.},
      author={Dzhafarov, Damir~D.},
      author={Solomon, Reed},
      author={Suggs, Jacob},
       title={The uniform content of partial and linear orders},
        date={2017},
        ISSN={0168-0072},
     journal={Annals of Pure and Applied Logic},
      volume={168},
      number={6},
       pages={1153\ndash 1171},
         url={https://doi.org/10.1016/j.apal.2016.11.010},
      review={\MR{3628269}},
}

\bib{BrattkaGherardi}{article}{
      author={Brattka, Vasco},
      author={Gherardi, Guido},
       title={Weihrauch degrees, omniscience principles and weak
  computability},
        date={2011},
        ISSN={0022-4812},
     journal={The Journal of Symbolic Logic},
      volume={76},
      number={1},
       pages={143\ndash 176},
         url={https://doi.org/10.2178/jsl/1294170993},
      review={\MR{2791341}},
}

\bib{BrattkaGherardiMarcone}{article}{
      author={Brattka, Vasco},
      author={Gherardi, Guido},
      author={Marcone, Alberto},
       title={The {B}olzano-{W}eierstrass theorem is the jump of weak
  {K}\"{o}nig's lemma},
        date={2012},
        ISSN={0168-0072},
     journal={Annals of Pure and Applied Logic},
      volume={163},
      number={6},
       pages={623\ndash 655},
         url={https://doi.org/10.1016/j.apal.2011.10.006},
      review={\MR{2889550}},
}

\bib{BrattkaGherardiPaulySurvey}{unpublished}{
      author={Brattka, Vasco},
      author={Gherardi, Guido},
      author={Pauly, Arno},
       title={Weihrauch complexity in computable analysis},
        date={2018},
        note={ArXiv preprint arXiv:1707.03202. 61pp.},
}

\bib{BrattkaHendtlassKreuzer}{article}{
      author={Brattka, Vasco},
      author={Hendtlass, Matthew},
      author={Kreuzer, Alexander~P.},
       title={On the uniform computational content of computability theory},
        date={2017},
        ISSN={1432-4350},
     journal={Theory of Computing Systems},
      volume={61},
      number={4},
       pages={1376\ndash 1426},
         url={https://doi.org/10.1007/s00224-017-9798-1},
      review={\MR{3712309}},
}

\bib{BrattkaPaulyAlg}{article}{
      author={Brattka, Vasco},
      author={Pauly, Arno},
       title={On the algebraic structure of {W}eihrauch degrees},
        date={2018},
     journal={Logical Methods in Computer Science},
      volume={14},
      number={4:4},
       pages={1\ndash 36},
         url={https://doi.org/10.1080/10724117.2007.11974704},
      review={\MR{3868998}},
}

\bib{BrattkaRakotoniaina}{article}{
      author={Brattka, Vasco},
      author={Rakotoniaina, Tahina},
       title={On the uniform computational content of {R}amsey's theorem},
        date={2017},
        ISSN={0022-4812},
     journal={The Journal of Symbolic Logic},
      volume={82},
      number={4},
       pages={1278\ndash 1316},
         url={https://doi.org/10.1017/jsl.2017.43},
      review={\MR{3743611}},
}

\bib{CholakJockuschSlaman}{article}{
      author={Cholak, Peter~A.},
      author={Jockusch, Carl~G., Jr.},
      author={Slaman, Theodore~A.},
       title={On the strength of {R}amsey's theorem for pairs},
        date={2001},
        ISSN={0022-4812},
     journal={The Journal of Symbolic Logic},
      volume={66},
      number={1},
       pages={1\ndash 55},
         url={https://doi.org/10.2307/2694910},
      review={\MR{1825173}},
}

\bib{CholakJockuschSlamanCorrigendum}{article}{
      author={Cholak, Peter~A.},
      author={Jockusch, Carl~G., Jr.},
      author={Slaman, Theodore~A.},
       title={Corrigendum to: ``{O}n the strength of {R}amsey's theorem for
  pairs''},
        date={2009},
        ISSN={0022-4812},
     journal={The Journal of Symbolic Logic},
      volume={74},
      number={4},
       pages={1438\ndash 1439},
         url={https://doi.org/10.2178/jsl/1254748700},
      review={\MR{2583829}},
}

\bib{ChongSlamanYang}{article}{
      author={Chong, C.~T.},
      author={Slaman, Theodore~A.},
      author={Yang, Yue},
       title={The metamathematics of stable {R}amsey's theorem for pairs},
        date={2014},
        ISSN={0894-0347},
     journal={Journal of the American Mathematical Society},
      volume={27},
      number={3},
       pages={863\ndash 892},
         url={https://doi.org/10.1090/S0894-0347-2014-00789-X},
      review={\MR{3194495}},
}

\bib{DoraisDzhafarovHirstMiletiShafer}{article}{
      author={Dorais, Fran\c{c}ois~G.},
      author={Dzhafarov, Damir~D.},
      author={Hirst, Jeffry~L.},
      author={Mileti, Joseph~R.},
      author={Shafer, Paul},
       title={On uniform relationships between combinatorial problems},
        date={2016},
        ISSN={0002-9947},
     journal={Transactions of the American Mathematical Society},
      volume={368},
      number={2},
       pages={1321\ndash 1359},
         url={https://doi.org/10.1090/tran/6465},
      review={\MR{3430365}},
}

\bib{DzhafarovStrongRed}{article}{
      author={Dzhafarov, Damir~D.},
       title={Strong reductions between combinatorial principles},
        date={2016},
        ISSN={0022-4812},
     journal={The Journal of Symbolic Logic},
      volume={81},
      number={4},
       pages={1405\ndash 1431},
         url={https://doi.org/10.1017/jsl.2016.1},
      review={\MR{3579116}},
}

\bib{Fiori-CaronesMarconeShaferSolda}{unpublished}{
      author={Fiori-Carones, Marta},
      author={Marcone, Alberto},
      author={Shafer, Paul},
      author={Sold\`{a}, Giovanni},
       title={({E}xtra)ordinary equivalences with the ascending/descending
  sequence principle},
        date={2021},
        note={ArXiv preprint arXiv:2107.02531. 33pp.},
}

\bib{Friedman}{inproceedings}{
      author={Friedman, Harvey},
       title={Some systems of second order arithmetic and their use},
        date={1975},
   booktitle={Proceedings of the {I}nternational {C}ongress of {M}athematicians
  ({V}ancouver, {B}. {C}., 1974), {V}ol. 1},
       pages={235\ndash 242},
      review={\MR{0429508}},
}

\bib{GavalecVojtas}{article}{
      author={Gavalec, Martin},
      author={Vojt\'{a}\v{s}, Peter},
       title={Remarks to a modification of {R}amsey-type theorems},
        date={1980},
        ISSN={0010-2628},
     journal={Commentationes Mathematicae Universitatis Carolinae},
      volume={21},
      number={4},
       pages={727\ndash 738},
      review={\MR{597762}},
}

\bib{HajekPudlak}{book}{
      author={H\'{a}jek, Petr},
      author={Pudl\'{a}k, Pavel},
       title={{M}etamathematics of {F}irst-{O}rder {A}rithmetic},
      series={Perspectives in Mathematical Logic},
   publisher={Springer-Verlag, Berlin},
        date={1998},
        ISBN={3-540-63648-X},
        note={Second printing},
      review={\MR{1748522}},
}

\bib{Herrmann}{article}{
      author={Herrmann, E.},
       title={Infinite chains and antichains in computable partial orderings},
        date={2001},
        ISSN={0022-4812},
     journal={The Journal of Symbolic Logic},
      volume={66},
      number={2},
       pages={923\ndash 934},
         url={https://doi.org/10.2307/2695053},
      review={\MR{1833487}},
}

\bib{HirschfeldtBook}{book}{
      author={Hirschfeldt, Denis~R.},
       title={{S}licing the {T}ruth: On the computable and reverse mathematics
  of combinatorial principles},
      series={Lecture Notes Series. Institute for Mathematical Sciences.
  National University of Singapore},
   publisher={World Scientific Publishing Co. Pte. Ltd., Hackensack, NJ},
        date={2015},
      volume={28},
        ISBN={978-981-4612-61-6},
      review={\MR{3244278}},
}

\bib{HirschfeldtJockusch}{article}{
      author={Hirschfeldt, Denis~R.},
      author={Jockusch, Carl~G., Jr.},
       title={On notions of computability-theoretic reduction between
  {$\Pi_2^1$} principles},
        date={2016},
        ISSN={0219-0613},
     journal={Journal of Mathematical Logic},
      volume={16},
      number={1},
       pages={1650002, 59pp.},
         url={https://doi.org/10.1142/S0219061316500021},
      review={\MR{3518779}},
}

\bib{HJKHLS}{incollection}{
      author={Hirschfeldt, Denis~R.},
      author={Jockusch, Carl~G., Jr.},
      author={Kjos-Hanssen, Bj{\o}rn},
      author={Lempp, Steffen},
      author={Slaman, Theodore~A.},
       title={The strength of some combinatorial principles related to
  {R}amsey's theorem for pairs},
        date={2008},
   booktitle={{C}omputational {P}rospects of {I}nfinity. {P}art {II}.
  {P}resented talks},
      series={Lect. Notes Ser. Inst. Math. Sci. Natl. Univ. Singap.},
      volume={15},
   publisher={World Sci. Publ., Hackensack, NJ},
       pages={143\ndash 161},
         url={https://doi.org/10.1142/9789812796554_0008},
      review={\MR{2449463}},
}

\bib{HirschfeldtShore}{article}{
      author={Hirschfeldt, Denis~R.},
      author={Shore, Richard~A.},
       title={Combinatorial principles weaker than {R}amsey's theorem for
  pairs},
        date={2007},
        ISSN={0022-4812},
     journal={The Journal of Symbolic Logic},
      volume={72},
      number={1},
       pages={171\ndash 206},
         url={https://doi.org/10.2178/jsl/1174668391},
      review={\MR{2298478}},
}

\bib{HirstThesis}{thesis}{
      author={Hirst, Jeffry~L.},
       title={{C}ombinatorics in {S}ubsystems of {S}econd {O}rder
  {A}rithmetic},
        type={Ph.D. Thesis},
   publisher={ProQuest LLC, Ann Arbor, MI},
        date={1987},
        note={The Pennsylvania State University},
}

\bib{JockuschRamsey}{article}{
      author={Jockusch, Carl~G., Jr.},
       title={Ramsey's theorem and recursion theory},
        date={1972},
        ISSN={0022-4812},
     journal={The Journal of Symbolic Logic},
      volume={37},
       pages={268\ndash 280},
         url={https://doi.org/10.2307/2272972},
      review={\MR{376319}},
}

\bib{JockuschSoare}{article}{
      author={Jockusch, Carl~G., Jr.},
      author={Soare, Robert~I.},
       title={{$\Pi ^{0}_{1}$} classes and degrees of theories},
        date={1972},
        ISSN={0002-9947},
     journal={Transactions of the American Mathematical Society},
      volume={173},
       pages={33\ndash 56},
         url={https://doi.org/10.2307/1996261},
      review={\MR{316227}},
}

\bib{JockuschStephan}{article}{
      author={Jockusch, Carl~G., Jr.},
      author={Stephan, Frank},
       title={A cohesive set which is not high},
        date={1993},
        ISSN={0942-5616},
     journal={Mathematical Logic Quarterly},
      volume={39},
      number={4},
       pages={515\ndash 530},
         url={https://doi.org/10.1002/malq.19930390153},
      review={\MR{1270396}},
}

\bib{LermanSolomonTowsner}{article}{
      author={Lerman, Manuel},
      author={Solomon, Reed},
      author={Towsner, Henry},
       title={Separating principles below {R}amsey's theorem for pairs},
        date={2013},
        ISSN={0219-0613},
     journal={Journal of Mathematical Logic},
      volume={13},
      number={2},
       pages={1350007, 44pp.},
         url={https://doi.org/10.1142/S0219061313500074},
      review={\MR{3125903}},
}

\bib{LiuRT22vsWKL}{article}{
      author={Liu, Jiayi},
       title={{${\mathsf{RT}}^2_2$} does not imply {${\mathsf{WKL}}_0$}},
        date={2012},
        ISSN={0022-4812},
     journal={The Journal of Symbolic Logic},
      volume={77},
      number={2},
       pages={609\ndash 620},
         url={https://doi.org/10.2178/jsl/1333566640},
      review={\MR{2963024}},
}

\bib{LiuRT22vsWWKL}{article}{
      author={Liu, Lu},
       title={Cone avoiding closed sets},
        date={2015},
        ISSN={0002-9947},
     journal={Transactions of the American Mathematical Society},
      volume={367},
      number={3},
       pages={1609\ndash 1630},
         url={https://doi.org/10.1090/S0002-9947-2014-06049-2},
      review={\MR{3286494}},
}

\bib{MiletiThesis}{thesis}{
      author={Mileti, Joseph~Roy},
       title={{P}artition {T}heorems and {C}omputability {T}heory},
        type={Ph.D. Thesis},
   publisher={ProQuest LLC, Ann Arbor, MI},
        date={2004},
        note={University of Illinois at Urbana-Champaign},
}

\bib{MoninPatey}{unpublished}{
      author={Monin, Beno\^{i}t},
      author={Patey, Ludovic},
       title={$\mathsf{SRT}^2_2$ does not imply $\mathsf{RT}^2_2$ in
  $\omega$-models},
        date={2019},
        note={ArXiv preprint arXiv:1905.08427. 21pp.},
}

\bib{NiesBook}{book}{
      author={Nies, Andr\'{e}},
       title={{C}omputability and {R}andomness},
      series={Oxford Logic Guides},
   publisher={Oxford University Press, Oxford},
        date={2009},
      volume={51},
        ISBN={978-0-19-923076-1},
         url={https://doi.org/10.1093/acprof:oso/9780199230761.001.0001},
      review={\MR{2548883}},
}

\bib{PateyImmunity}{article}{
      author={Patey, Ludovic},
       title={Partial orders and immunity in reverse mathematics},
        date={2018},
        ISSN={2211-3568},
     journal={Computability},
      volume={7},
      number={4},
       pages={323\ndash 339},
         url={https://doi.org/10.3233/com-170071},
      review={\MR{3877170}},
}

\bib{PateyYokoyama}{article}{
      author={Patey, Ludovic},
      author={Yokoyama, Keita},
       title={The proof-theoretic strength of {R}amsey's theorem for pairs and
  two colors},
        date={2018},
        ISSN={0001-8708},
     journal={Advances in Mathematics},
      volume={330},
       pages={1034\ndash 1070},
         url={https://doi.org/10.1016/j.aim.2018.03.035},
      review={\MR{3787563}},
}

\bib{RivalSands}{article}{
      author={Rival, Ivan},
      author={Sands, Bill},
       title={On the adjacency of vertices to the vertices of an infinite
  subgraph},
        date={1980},
        ISSN={0024-6107},
     journal={Journal of the London Mathematical Society. Second Series},
      volume={21},
      number={3},
       pages={393\ndash 400},
         url={https://doi.org/10.1112/jlms/s2-21.3.393},
      review={\MR{577715}},
}

\bib{RosensteinBook}{book}{
      author={Rosenstein, Joseph~G.},
       title={Linear orderings},
      series={Pure and Applied Mathematics},
   publisher={Academic Press, Inc. [Harcourt Brace Jovanovich, Publishers], New
  York-London},
        date={1982},
      volume={98},
        ISBN={0-12-597680-1},
      review={\MR{662564}},
}

\bib{SeetapunSlaman}{incollection}{
      author={Seetapun, David},
      author={Slaman, Theodore~A.},
       title={On the strength of {R}amsey's theorem},
        date={1995},
      volume={36},
       pages={570\ndash 582},
         url={https://doi.org/10.1305/ndjfl/1040136917},
        note={Special Issue: Models of arithmetic},
      review={\MR{1368468}},
}

\bib{SimpsonSOSOA}{book}{
      author={Simpson, Stephen~G.},
       title={{S}ubsystems of {S}econd {O}rder {A}rithmetic},
     edition={Second edition},
      series={Perspectives in Logic},
   publisher={Cambridge University Press, Cambridge; Association for Symbolic
  Logic, Poughkeepsie, NY},
        date={2009},
        ISBN={978-0-521-88439-6},
         url={https://doi.org/10.1017/CBO9780511581007},
      review={\MR{2517689}},
}

\bib{SoareBook}{book}{
      author={Soare, Robert~I.},
       title={Recursively {E}numerable {S}ets and {D}egrees},
      series={Perspectives in Mathematical Logic},
   publisher={Springer-Verlag, Berlin},
        date={1987},
        ISBN={3-540-15299-7},
         url={https://doi.org/10.1007/978-3-662-02460-7},
      review={\MR{882921}},
}

\bib{Specker}{inproceedings}{
      author={Specker, E.},
       title={Ramsey's theorem does not hold in recursive set theory},
        date={1971},
   booktitle={Logic {C}olloquium '69 ({P}roc. {S}ummer {S}chool and {C}olloq.,
  {M}anchester, 1969)},
   publisher={North-Holland, Amsterdam},
       pages={439\ndash 442},
      review={\MR{0278941}},
}

\bib{WeihrauchDegrees}{techreport}{
      author={Weihrauch, Klaus},
       title={The degrees of discontinuity of some translators between
  representations of the real numbers},
 institution={International Computer Science Institute},
     address={Berkeley, California},
        date={1992},
}

\bib{WeihrauchTTE}{techreport}{
      author={Weihrauch, Klaus},
       title={The {TTE}-interpretation of three hierarchies of omniscience
  principles},
 institution={FernUniversit\"{a}t in Hagen},
     address={Hagen},
        date={1992},
}

\end{biblist}
\end{bibdiv}

\vfill

\end{document}